\documentclass[12pt]{article}
\usepackage{amssymb,latexsym, amsmath, amsthm, amsfonts, amssymb, enumerate, paralist,cite}
\usepackage[notref,notcite]{showkeys}
\usepackage[usenames,dvipsnames,svgnames,x11names, table]{xcolor}
\usepackage{graphicx, overpic}
\usepackage{color}

\def\A{{\mathbf{A}}}
\def\a{{\widetilde{a}}}
\def\B{{\mathbf{B}}}
\def\b{{\widetilde{b}}}
\def\E{{\mathbf{E}}}
\def\OA{\Omega_*}
\def\I{{\mathbb{I}}}
\def\U{{\mathbb{U}}}

\def\G{{\mathbb{G}}}
\def\M{{\mathbb{M}}}
\def\X{{\mathrm{X}}}

\def\T{{\mathrm{T}}}
\def\R{\mathbb R}
\def\L{\mathbb L}
\def\O{\mathbb O}
\def\N{\mathbb N}
\def\C{\mathbb C}
\def\H{\mathbb H}
\def\Z{\mathbb Z}
\def\K{{(\mathrm{K})}}
\def\xx{\widetilde{x}_1}
\def\xxa{\widetilde{x}_*}
\def\yy{\widetilde{x}^*}
\def\xnx{\mathbf{x}}
\def\CR{{(\mathrm{CR})}}
\def\u{\mathrm{u}}
\def\UU{\mathrm{U}}
\def\g{\mathrm{g}}
\def\x{\mathrm{x}}
\def\t{\mathrm{t}}
\def\v{\mathrm{v}}
\def\w{\mathrm{w}}

\newtheorem{theorem}{Theorem}
\newtheorem{lemma}{Lemma}
\newtheorem{remark}{Remark}

\newtheorem{corollary}{Corollary}
\newtheorem{proposition}{Proposition}

\title{Sub-Riemannian geometry on some step-two Carnot groups}
\author{Hong-Quan Li, Ye Zhang}
\date{}
\begin{document}

\renewcommand{\theequation}{\thesection.\arabic{equation}}
\setcounter{equation}{0} \maketitle

\vspace{-1.0cm}

\bigskip

{\bf Abstract.}
This paper is a continuation of the previous work of the first author. We characterize a class of step-two groups introduced in \cite{Li19}, saying GM-groups, via some basic sub-Riemannian geometric properties, including the squared Carnot-Carath\'{e}odory distance, the cut locus, the classical cut locus, the optimal synthesis, etc. Also, the shortest abnormal set can be exhibited easily in such situation.
Some examples of such groups are step-two groups of corank $2$, of Kolmogorov type, or those associated to quadratic CR manifolds. As a byproduct, the main goal in \cite{BBG12} is achieved from the setting of step-two groups of corank $2$ to all possible step-two groups, via a completely different method. A partial answer to the open questions \cite[(29)-(30)]{BR19} is provided in this paper as well. Moreover, we provide a entirely different proof, based yet on \cite{Li19}, for the Gaveau-Brockett optimal control problem on the free step-two Carnot group with three generators.  As a byproduct, we provide a new and independent proof for the main results obtained in \cite{MM17}, namely, the exact expression of $d(g)^2$ for $g$ belonging to the classical cut locus of the identity element $o$, as well as the determination of all shortest geodesics joining $o$ to such $g$.

\medskip


\medskip


\medskip

\renewcommand{\theequation}{\thesection.\arabic{equation}}
\section{Introduction}
\setcounter{equation}{0}

In the past several decades, step-two groups and their sub-Laplacians,
as special Lie groups of polynomial volume growth or perfect  sub-Riemannian manifolds,
have attracted wide attention from experts in various fields, such as complex analysis, control theory, geometric measure theory, harmonic analysis, heat kernel, Lie group theory, probability analysis, PDE, sub-Riemannian geometry, etc.
We only mention some relevant works here,
\cite{G77,A92,A94,BGG00,A02,BFKG16,
BCC19,B82,M02,AR04,BGG96,BLU07,BR96,
BR18,LLMV13,MM17,RS17,Ma02,HK00,BR19,
B84,R13,VSC92,FS82,BTW09,BFMT12,RS76,
MaM16,BGX00,BBBC08,BGL14,BKS18,BKS19,
BGM18,BB16,BG17,BGM19,BHT08,BK14,C79,
CWW95,CL08,CC15,CFZ20,C85,C91,CG84,
C93,CMZ96,CS08,DP89,DH05,DLLW19,E09,
E10,ERS99,FF15,FJ08,F73,F75,FS74,
FLW95,FSS01,FSS03,FL12,HZ10,HL10,
H76,H67,HM89,J86,JL88,J09,J14,K80,
KR95,LO14,Li06,Li07,Li09,Li10,LQ14,
LZ19,LDR17,LDMOPV16,NRS01}. The list is far from exhaustive and in fact is rather limited. More related papers can be found in the references therein as well as their subsequent researches.

In this present paper, we will restrict our attention to the sub-Riemannian geometry on step-two Carnot groups. Many relevant works can be found in the literature as cited before. However, some most fundamental problems are far from being solved, or even poorly known, in this very fine framework. Recently, in \cite{Li19} (cf. also \cite{Li19O}),
the first author used Loewner's theorem to study two basic problems of sub-Riemannian geometry on $2$-step groups: one is to obtain the exact formula for the sub-Riemannian distance, that is the Gaveau-Brockett optimal control problem; another is to characterize all (shortest) normal geodesics from the identity element $o$ to any given $g \neq o$. In particular, there exists an enormous class of $2$-step groups, saying GM-groups (see Subsection \ref{s22} below for the definition),
which have some consummate sub-Riemannian geometric properties. More precisely,
the squared Carnot-Carath\'eodory distance $d(g)^2 := d(o, g)^2$ and the cut locus of $o$, $\mathrm{Cut}_o$ (namely the set of points where $d^2$ is not smooth) can be characterized easily in such situation. For example, all Heisenberg groups even generalized Heisenberg-type groups (so step-two groups of corank $1$), and star graphs are GM-groups. We emphasize that in general, the expression of $d(g)^2$ is extremely complicated. It is impossible to provide an explicit expression, via a relatively simple inverse function, as in the most special situation of generalized Heisenberg-type groups (cf. \cite{G77}, \cite{BGG00} and \cite{Li19}). We refer the reader to \cite{Li19} for more details. The work is a continuation of \cite{Li19}, one of our main goals is to provide various equivalent characterizations of GM-groups via basic sub-Riemannian geometric properties.

Moreover, the sub-Riemannian geometry in the setting of $2$-step groups is not well understood. Roughly speaking, the main reason for this is that the well-understood examples are merely the Heisenberg group (cf. \cite{G77}) and generalized Heisenberg groups (cf. \cite{BGG00}). Other known cases, such as generalized Heisenberg-type groups as well as the direct product of a generalized Heisenberg group with a Euclidean space (in particular, step-two groups of corank $1$), are essentially the same. Hence, it is very meaningful and exigent to supply some examples possessing richer sub-Riemannian geometric properties. Here, we will provide more examples of GM-groups,
such as groups of corank $2$, of Kolmogorov type, or those associated to quadratic CR manifolds.
In particular, the aforementioned groups may have complicated shortest abnormal set of $o$, $\mathrm{Abn}^*_o$,
that is, the set of the endpoints of abnormal shortest geodesics starting from $o$.
The existence of non-trivial abnormal shortest geodesics is closely related to the regularity of the Carnot-Carath\'eodory distance. And its appearance makes an obstacle for us to deal with some topics, such as the heat kernel asymptotics and geometric inequalities, etc. See for example \cite{B84,B88,M02,BR19,Li19, ABB20} and the references therein for more details.
Recall that (cf. \cite{R13,R14}) a sub-Riemannian manifold is called \textit{ideal} if it is complete and has no non-trivial abnormal shortest geodesics. In our setting, a step-two group $\G$ is ideal if and only if it is of M\'etivier type (see Subsection \ref{ns221} for the definition).

Optimal syntheses (namely the collection of all arclength parametrized geodesics with their cut times) are generally very difficult to obtain. In the setting of step-two groups, as far as we know, a correct result about them can be found only on nonisotropic Heisenberg groups. See \cite[\S~13]{ABB20} and Remark \ref{Rkn1} below for more details. However, we can now give the optimal synthesis from the identity element $o$ on GM-groups. As a result, the classical cut locus of $o$, $\mathrm{Cut}^{\mathrm{CL}}_o$, that is the set of points where geodesics starting at $o$ cease to be shortest, can be characterized on such groups as well.

We say that $d^2$ is \textit{semiconcave} (resp. \textit{semiconvex}) \textit{in a neighborhood of $g_0$}
if there exist $C > 0$ and $\delta > 0$ such that
\begin{align} \label{DSSCC}
d(g_1 + g')^2 + d(g_1 - g')^2 - 2 \, d(g_1)^2 \le C \, |g'|^2  \, \mbox{(resp. $ \ge - \, C \, |g'|^2 $)},
\end{align}
for all $g_1 \pm g' \in  B(g_0, \delta) = \{ g \in \R^q \times \R^m; \, |g - g_0| < \delta \}$. Here we stress that $| \cdot |$ denotes the usual Euclidean norm and $g_1 \pm g^\prime$ the usual operation in the Euclidean space. We also remark that this definition is independent of the choice of local coordinates around $g_0$ (here we use the canonical one) since $d^2$ is locally Lipschitz w.r.t. the usual Euclidean distance (see for example \cite{R13,R14}). Set in the sequel{\footnote{We would like to thank L. Rizzi for informing us of the addendum of \cite{BR19} that the definition of the failure of semiconcavity/semiconvexity for the open questions should be the one stated here in our situation (which is consist with the classical definition of local semiconcavity/semiconvexity), rather than the one given in \cite{BR19}. For more details, we refer to the addendum on L. Rizzi's homepage.} }
\begin{align}
\mathrm{SC}^-_o &:= \left\{g; \, d^2 \mbox{ fails to be semiconcave in any neighborhood of $g$}  \right\}, \label{DoFSC1} \\
\mathrm{SC}^+_o &:= \left\{g; \, d^2 \mbox{ fails to be semiconvex in any neighborhood of $g$}  \right\}. \label{DoFSC2}
\end{align}

Recall that $\mathrm{SC}^-_o = \mathrm{Abn}^*_o$ in the setting of M\'etivier groups, all free Carnot groups of step $2$, as well as some other sub-Riemannian structures. See \cite{CL08}, \cite{FR10}, \cite{MM16}, \cite[\S~4.1 and \S~4.2]{BR19} and references therein for more details. And an open problem is raised in \cite[(29)]{BR19}, which asks whether it holds $\mathrm{SC}^-_o = \mathrm{Abn}^*_o$ in the more general sub-Riemannian setting (where our $\mathrm{Abn}^*_o$ is noted by $\mathrm{Abn}(o)$). In the framework of GM-groups, $\mathrm{Abn}^*_o$ can be described easily; as a byproduct, we give a positive answer to this open problem.
Also, other related results and step-two groups can be found in Subsection \ref{s23}.

In addition, the most challenging problem should be to study the sub-Riemannian geometry in the setting of free step-two groups with $k$ generators $N_{k, 2} \cong \R^k \times \R^{\frac{k (k -1)}{2}}$ ($k \geq 3$).
Indeed, for any step-two group $\G$ with $k$ generators, that is the first layer in the stratification of Lie algebra has dimension $k$, there exists some relation between $\G$ and $N_{k, 2}$ by Rothschild-Stein lifting theorem (see \cite{RS76} or \cite{BLU07}). Observe that $N_{2, 2}$ is exactly the Heisenberg group, which is well-known (cf. \cite{G77} or \cite{BGG00}). Recall that (cf. \cite{G77} and \cite{B82}) the original Gaveau-Brockett optimal control problem is to determine the sub-Riemannian distance on $N_{k, 2}$ with $k \ge 3$. This is a long-standing open problem.
Recently, it is completely solved on $N_{3, 2}$ in \cite[\S~11]{Li19}. Also remark that the main idea and method in \cite{Li19} can be adapted to general step-two groups and other situations.

In the setting of $N_{3, 2}$, the classical cut locus of $o$, $\mathrm{Cut}^{\mathrm{CL}}_o$,
has been determined in \cite{My02} and \cite{MM17} by completely different techniques; furthermore, the expression of $d(g)^2$ with $g \in \mathrm{Cut}^{\mathrm{CL}}_o$ has been obtained in \cite{MM17}. Strictly speaking, we have used the above known results in the proof of \cite{Li19}. Also notice that $\mathrm{Abn}_o^*$ and $\mathrm{Cut}^{\mathrm{CL}}_o$ on $N_{3, 2}$ are relatively very simple. However, it is still an open problem to characterize the classical cut locus of $o$ on $N_{k, 2}$ with $k \ge 4$, see \cite{RS17} for more details. Motivated by this problem, we ask naturally if we can determine first $d(g)^2$ for any $g$ then $\mathrm{Cut}^{\mathrm{CL}}_o$ on $N_{3, 2}$. This is exactly another main purpose of this work.

In the framework of step-two groups, first we recall that all shortest geodesics are normal (cf. \cite{AS04} or \cite[\S~2.4]{R14}). Next, up to a subset of measure zero, all normal geodesics from $o$ to any given $g \neq o$ have been characterized by \cite[Theorem~2.4]{Li19}. Moreover, it follows from \cite[Theorem~2.5]{Li19} that the squared distance has been determined in a symmetric, scaling invariant subset with non-empty interior. In particular, for the special case of $N_{3, 2}$, we can simplify
the Gaveau-Brockett problem via an orthogonal-invariant property, and some useful results can be found in \cite[\S~11]{Li19}. Based on these known results, we can describe the squared distance on $N_{3, 2}$ first on some dense open subset, then on whole space via a limiting argument.  Then,
from the regularity of the squared sub-Riemannian distance, we can further determine the cut locus $\mathrm{Cut}_o$.
Finally, all shortest geodesics joining $o$ to any given point in $\mathrm{Cut}_o^{\mathrm{CL}}$ are obtained by approximating them with that joining $o$ to some points in $(\mathrm{Cut}_o)^c$, which are relatively easy to describe. As a consequence, we supply an independent and new proof for the main results obtained in \cite{MM17}.

Some applications will be given in a future work.

\medskip

This paper is organized as follows. In Section \ref{s2}, we collect some preliminary materials and give our main results, which will be proven in Section \ref{nsP}. In Section \ref{s3}, we provide a sufficient condition for a step-two group to be GM-group by using semi-algebraic theory. As a consequence,  we find that all step-two groups of corank $2$ are GM-groups.  Furthermore, we also prove in this section that  there exist M\'etivier groups of corank $3$ and of sufficiently large dimension which are not of GM-type. In Section \ref{s6}, we consider the sub-Riemannian geometry in the setting of step-two K-type groups. Step-two groups associated to quadratic CR manifolds will be studied in Section \ref{s7}.
Finally, we give in Section \ref{s5} a completely different proof, based on [87], for the Gaveau-Brockett optimal control problem on $N_{3, 2}$. As an application, we provide a new and independent proof for the main results obtained in \cite{MM17}.

\medskip

\renewcommand{\theequation}{\thesection.\arabic{equation}}
\section{Preliminaries and main results} \label{s2}
\setcounter{equation}{0}

\subsection{Step-two Carnot groups}\label{s21}

Recall that a connected and simply connected Lie group $\G$ is a step-two Carnot group
if its left-invariant Lie algebra $\mathfrak{g}$ admits a stratification
\begin{align*}
\mathfrak{g} = \mathfrak{g}_1 \oplus \mathfrak{g}_2, \quad
[\mathfrak{g}_1, \mathfrak{g}_1] = \mathfrak{g}_2, \quad
[\mathfrak{g}_1, \mathfrak{g}_2] = \{0\},
\end{align*}
where $[\cdot,\cdot]$ denotes the Lie bracket on $\mathfrak{g}$. We identify $\G$ and $\mathfrak{g}$ via the exponential map. As a result, $\G$ can be considered as $\R^q \times \R^m$, $q, m \in \N^* = \{1, 2, 3, \ldots\}$ (in this paper we use $\N$ to denote the set of natural numbers $\{0,1,2,\ldots\}$), with the group law
\begin{align*}
(x , t) \cdot (x^{\prime}, t^{\prime}) =
\left(x + x^{\prime}, t + t^{\prime} + \frac{1}{2}\langle  \U x, x^{\prime} \rangle \right), \quad g := (x, t) \in \R^q \times \R^m,
\end{align*}
where
\begin{align*}
\langle\U x,  x^{\prime} \rangle := (\langle U^{(1)} x, x^{\prime} \rangle, \ldots, \langle   U^{(m)} x,x^{\prime} \rangle) \in \R^m.
\end{align*}
Here $\U = \{U^{(1)},\ldots,U^{(m)}\}$ is an $m$-tuple of linearly independent $q \times q$ skew-symmetric matrices with real entries and $\langle \cdot, \cdot \rangle$ (or $\mbox{} \cdot \mbox{}$ in the sequel when there is no ambiguity) denotes the usual inner product on $\R^q$. Furthermore, in this article, we will not distinguish row vectors from column vectors and we may write a column vector $t$ with scalar coordinates $t_1, \ldots, t_m$, simply as $(t_1,\ldots,t_m)$ unless otherwise stated in the context. Note that $m \le \frac{q(q - 1)}{2}$. We call such a group a step-two group of type $(q,m,\U)$, which is denoted by $ \G(q, m, \U)$ or $\G$ for simplicity. One can refer to \cite{E03} or \cite{BLU07} for more details.

Let $U^{(j)} = (U^{(j)}_{l, k})_{1 \leq l, k \leq q}$ ($1 \leq j \leq m$). The canonical basis of $\mathfrak{g}_1$ is defined by the left-invariant vector fields on $\G$:
\begin{align*}
\X_l(g) : = \frac{\partial}{\partial x_l} + \frac{1}{2} \sum_{j = 1}^m \Big( \sum_{k = 1}^{q} U^{(j)}_{l, k} x_k \Big) \frac{\partial}{\partial t_j}, \qquad 1 \leq l \leq q.
\end{align*}
And the canonical sub-Laplacian is
$\Delta = \sum\limits_{l = 1}^q \X_l^2$.

\medskip

\subsection{Left-invariant sub-Riemannian geometry on $\G$: some elementary properties}

Let us first recall some basic facts  about the sub-Riemannian geometry in the framework of $2$-step groups.
In our setting, we will sometimes use equivalent definitions for some concepts in order to avoid recalling too many notations.
We refer the reader to \cite{B84, S86, B88, BR96, M02, R14, ABB20} and references therein for further details.
Also notice that partial but not all results below remain valid in some more general setting.

The group $\G = \G(q,m,\U)$ is endowed with
the sub-Riemannian structure, namely a scalar product on $\mathfrak{g}_1$, with respect to which $\{\X_l\}_{1 \le l \le q}$ are orthonormal (and the norm induced by this scalar product is denoted by $\| \cdot \|$). In the sequel, $m$ is called the \textit{corank} of $\G(q, m, \U)$.

A \textit{horizontal curve} $\gamma: [0, \ 1] \to \G$ is an absolutely continuous path such that
\[
\dot{\gamma}(s) = \sum_{j = 1}^q u_j(s) \X_j(\gamma(s)) \qquad \mbox{for a.e. } s \in [0, \ 1],
\]
and we define its \textit{length} as follows
\begin{align*}
\ell(\gamma) := \int_0^1 \|\dot{\gamma}(s)\| \, ds = \int_0^1 \sqrt{\sum_{j = 1}^q |u_j(s)|^2} \, ds.
\end{align*}
The \textit{Carnot-Carath\'eodory} (or \textit{sub-Riemannian}) \textit{distance} between $g, g' \in \G$ is then
\[
d(g, g') := \inf\left\{\ell(\gamma); \ \gamma(0) = g, \ \gamma(1) = g', \gamma \mbox{ horizontal} \right\}.
\]
A \textit{geodesic} is a horizontal curve $\gamma$ satisfying: $\| \dot{\gamma}(s) \|$ is constant and for any $s_0 \in [0, \ 1]$ there exists a neighborhood $I$ of $s_0$ in $[0, \ 1]$ such that $\ell(\gamma|_I)$ is equal to the distance between its endpoints.
And a \textit{shortest geodesic} is a geodesic $\gamma$ which realizes the distance between its extremities, that is, $\ell(\gamma) = d(\gamma(0), \gamma(1))$.

By slightly abusing of notation in the sequel, $0$ denotes the number $0$ or the origin in the Euclidean space. Let $o = (0, 0)$ denote the identity element of $\G$.
It is well-known that $d$ is a left-invariant distance
on $\G$. Hence we set in the following $d(g) := d(g, o)$. Recall that
\textit{$d^2$ is locally Lipschitz on $\G$ with respect to the usual Euclidean distance.}
The dilation on $\G$ is defined by
\begin{align} \label{nDS}
\delta_r(x, t) := (r \, x, r^2 \, t), \quad \forall \, r > 0, \ (x, t) \in \G.
\end{align}
And the following scaling property is well-known:
\begin{align} \label{scap}
d(r \, x, r^2 \, t) = r \, d(x, t), \quad \forall \, r > 0, \ (x, t) \in \G.
\end{align}

\subsubsection{Sub-Riemannian Hamiltonian and normal geodesics starting from $o$}\label{ns221}

In the setting of step-two Carnot groups,
it is well-known that all shortest geodesics are projections of normal Pontryagin extremals, that is integral curves of the sub-Riemannian Hamiltonian in $T^* \G$. See for example \cite[\S~20.5]{AS04} or \cite[Theorem 2.22]{R14}.

More precisely, the sub-Riemannian Hamiltonian in $T^* \G \cong (\R^q \times \R^m) \times (\R^q \times \R^m)$ is defined by
\begin{align*}
H = H(x, t, \xi, \tau) := \frac{1}{2} \sum_{j = 1}^q \zeta_j^2, \quad \zeta_j :=  \xi_j + \frac{1}{2} \sum_{k = 1}^m \left( \sum_{l = 1}^q U_{j, l}^{(k)} x_l \right) \tau_k, \quad 1 \le j \le q.
\end{align*}
And a \textit{normal Pontryagin extremal},
\[
\Big( \gamma(s):=(x(s), t(s)), \xi(s), \tau(s) \Big): [0, \ 1] \longrightarrow T^* \G, \qquad \mbox{with } \gamma(0) = o,
\]
is a solution of
\begin{align}\label{HJE}
\dot{x}_k  &= \frac{\partial H}{\partial \xi_k}, \quad  \dot{t}_j  = \frac{\partial H}{\partial \tau_j}, \quad
\dot{\xi}_k  = -\frac{\partial H}{\partial x_k}, \quad  \dot{\tau}_j  = - \frac{\partial H}{\partial t_j}, \quad
1 \le k \le q, \, 1 \le j \le m.
\end{align}

The covector $(\xi(0), \tau(0))$ (resp. $(\xi(1), \tau(1))$) is called the initial (resp. final) covector of $(\gamma(s), \xi(s), \tau(s))$. Its projection
\[
\gamma(s):= \gamma(\xi(0), \tau(0); s) = \gamma_{(\xi(0), \tau(0))}(s) = (x(s), t(s)): [0, \ 1] \longrightarrow \G
\]
is said to be the \textit{normal geodesic starting from $o$ with initial covector $(\xi(0), \tau(0))$}.

Note that $H$ is independent of $t$. Hence we have
\begin{align}
\tau(s) \equiv \tau(0) := 2 \, \theta \in \R^m.
\end{align}
Set in the following
\begin{align} \label{Du1}
\widetilde{U}(\theta) := \sum_{j = 1}^m \theta_j \, U^{(j)} \mbox{ and } U(\theta) := i \, \widetilde{U}(\theta), \quad \mbox{for $\theta = (\theta_1, \ldots, \theta_m) \in \R^m$.}
\end{align}
Recall that a step-two group $\G$ is a \textit{M\'etivier group} (or \textit{of M\'etivier type})
if $U(\theta)$ is invertible for any $\theta \neq 0$ (cf. \cite{M80}).

Let $\zeta(s) := \xi(s) + \widetilde{U}(\theta) \, x(s)$. Remark that $\xi(0) = \zeta(0)$. A simple calculation implies that
\begin{align}\label{GEn}
\zeta(s) = e^{2 \, s \, \widetilde{U}(\theta)} \, \zeta(0), \quad
x(s) = \int_0^s \zeta(r) \, dr, \quad
t(s) = \frac{1}{2} \int_0^s \langle \U x(r), \zeta(r) \rangle \, dr.
\end{align}

In particular, we have
\begin{align}\label{endpointx}
x(1) = \int_0^1 \zeta(r) \, dr = \frac{\sin{U(\theta)}}{U(\theta)} e^{\widetilde{U}(\theta)} \, \zeta(0),
\end{align}
and $\gamma(\xi(0), \tau(0); s) = (x(s), t(s))$ is extendable and real analytic on $[0, \ +\infty)$. It is easy to check the  following homogeneity property:
\[
\gamma(\alpha \, \zeta_0, 2 \, \alpha \, \theta_0; s) = \gamma(\zeta_0, 2 \, \theta_0; \alpha \, s), \qquad \forall \, \alpha > 0, \ s \geq 0, \ (\zeta_0, 2 \, \theta_0) \in \R^q \times \R^m.
\]

From now on, the domain of the normal geodesic
$\gamma(s) = \gamma(\zeta_0, 2 \, \theta_0; s)$ is $[0, \ +\infty)$ and that of $\gamma(s) = \gamma_{(\zeta_0, 2 \, \theta_0)}(s)$ is $[0, \ 1]$ by default. Also remark that $\gamma = o$ if $\zeta_0 = 0$, which is trivial. And all normal geodesics are by convention starting from $o$ in this work.

Let $(x, t)$ denote the endpoint of $\gamma_{(\zeta, \tau)}$, then by \eqref{GEn}, that of $\gamma_{(-\zeta, \tau)}$ is $(- x , t)$. Moreover, both $\gamma_{(\zeta, \tau)}$ and $\gamma_{(-\zeta, \tau)}$ have length $|\zeta|$, where $| \cdot |$ denotes the usual Euclidean norm. Combining this with the fact that the Carnot-Carath\'eodory distance is a left-invariant distance on $\G$, we have the following simple but useful observation:

\begin{lemma}
In the setting of step-two groups, it holds that
\begin{align} \label{te1}
d(x, t) = d(-x, t) = d(x, -t) = d(-x, -t), \quad \forall \, (x, t) \in \G.
\end{align}
\end{lemma}

The following basic property is well-known:

\begin{lemma} \label{BPG}
Let $0 \le s_1 < s_2$.  Assume that $\gamma(\zeta_0, 2 \, \theta_0; s) = \gamma(\zeta', 2 \, \theta'; s)$ for all $s_1 \le s \le s_2$. Then we have $\gamma(\zeta_0, 2 \, \theta_0; \cdot) \equiv \gamma(\zeta', 2 \, \theta'; \cdot)$.
Moreover, it holds that $\zeta_0 = \zeta'$.
\end{lemma}

More information about such geodesics can be found in Proposition \ref{cAg} below.

\subsubsection{Sub-Riemannian exponential map, cut point and optimal synthesis}

The \textit{sub-Riemannian exponential map} based at $o$ is the smooth map defined by
\begin{align*}
\exp: \, \R^q \times \R^m &\longrightarrow \G \\
(\zeta_0, 2 \, \theta_0) &\longmapsto \gamma(\zeta_0, 2 \, \theta_0; 1).
\end{align*}
In our setting, it is surjective and has the following property:
\[
\gamma(\zeta_0, 2 \, \theta_0; s) = \exp\{s \, (\zeta_0, 2 \, \theta_0)\}, \qquad \forall \, s \geq 0, \ (\zeta_0, 2 \, \theta_0) \in \R^q \times \R^m.
\]
See for example \cite[\S~8.6]{ABB20}.
Furthermore, we have the following simple observations:

\begin{lemma}
Suppose that $\exp(w, \tau) = (x, t)$. Then we have
\begin{gather} \label{symN2}
\exp(r \, w, \tau) = ( r \, x, r^2 \, t),  \quad \forall \, r \neq 0, \\
\label{symN3}
\exp(- e^{\widetilde{U}(\tau)} \, w, - \tau) = (-x , -t).
\end{gather}
\end{lemma}

Indeed, using \eqref{GEn}, \eqref{symN2} is trivial, and an elementary computation implies \eqref{symN3}.

Now assume
that $\gamma(s) = \exp\{s \, (\zeta_0, 2 \, \theta_0)\}$ is parametrized by arclength (or \textit{arclength parametrized}), namely $|\zeta_0| = 1$. Let
$g_0 = \gamma(s_0)$. We say that $g_0$ is \textit{conjugate to $o$} along $\gamma$ if $s_0 \, (\zeta_0, 2 \, \theta_0)$ is a critical point of $\exp$. The \textit{cut time} along $\gamma$ is defined as
\begin{align}
h_{\mathrm{cut}} :=
h_{\mathrm{cut}}(\gamma) =
\sup\{s > 0; \ \gamma|_{[0, \ s]} \mbox{ is a shortest geodesic} \}.
\end{align}
When $h_{\mathrm{cut}} < +\infty$, $\gamma(h_{\mathrm{cut}})$ is said to be the \textit{cut point} of $o$ along $\gamma$. And we say $\gamma$ has no cut point if $h_{\mathrm{cut}} = +\infty$. The \textit{optimal synthesis} from $o$ is the collection of all arclength parametrized geodesics with their cut times.

\subsubsection{Shortest abnormal set and cut locus} \label{ns223}

A normal geodesic is said to be \textit{abnormal (i.e. singular)} if it has two (so infinitely many) different normal lifts (see \cite[Remark 8]{RT05} and \cite[Remark 2.4]{R14}). However, we stress that our definition of abnormal geodesic is not complete in general. In particular, on some sub-Riemannian manifolds, excluding our step-two groups, there are shortest geodesics which are not
projections of normal Pontryagin extremals. For the original definition of abnormal geodesic as well as counter-examples, we refer the reader to \cite{Mont94,LS95,M02,R14,ABB20} and the references therein for more details.

In this work, the \textit{(normal-) abnormal set} of $o$, $\mathrm{Abn}_o$ is defined by
\begin{align*}
\mathrm{Abn}_o := \{g; \, \mbox{there exists an abnormal (which is also normal) geodesic joining $o$ to $g$}\}.
\end{align*}
And we define the \textit{shortest abnormal set} of $o$ as follows:
\begin{align*}
\mathrm{Abn}_o^* := \{g; \, \mbox{there exists an abnormal \textit{shortest} geodesic joining $o$ to $g$}\},
\end{align*}
which is a subset of $\mathrm{Abn}_o$. The main difference between the two sets is that: in the definition of $\mathrm{Abn}_o$, we do not care about minimality of geodesics, while this is needed in that of $\mathrm{Abn}_o^*$.
Notice that $o \in \mathrm{Abn}^*_o$. Also remark that our $\mathrm{Abn}_o$ is exactly $\mathrm{Abn}^{nor}(e)$ in \cite[\S~2.7]{LDMOPV16}.

The following characterization of abnormal geodesics, which can be also considered as an improvement of Lemma \ref{BPG}, can be easily verified by \eqref{GEn} (see also \cite[\S~3.1.1]{MM16} for an explanation from the original definition of abnormal (-normal) geodesics).

\begin{proposition} \label{cAg}
Let $\theta \neq \theta'$. Then $\gamma(w, 2 \, \theta; \cdot) \equiv \gamma(w, 2 \, \theta'; \cdot)$ if and only if for $\sigma = \theta - \theta' \in \R^m \setminus \{ 0 \}$, we have
\begin{align}\label{CharaAbn}
U(\sigma) \, U(\theta)^k \, w = 0, \quad \forall \, k \in \N,
\end{align}
or equivalently,
\begin{align}\label{AGc}
U(\sigma) \, e^{s \, \widetilde{U}(\theta)} \, w = 0, \quad \forall \, s \in \R.
\end{align}
That is, $\gamma_{(w, 2 \, \theta)}$ (or $\gamma(w, 2 \, \theta; \cdot)$) is abnormal if and only if there exists some $\sigma \neq 0$ such that \eqref{CharaAbn} (or equivalently \eqref{AGc}) satisfies.
\end{proposition}

As a consequence, we get the following known fact:

\begin{corollary}
If $\gamma_{(w, 2 \, \theta)}$ is abnormal, then so does $\gamma_{(a \, w, 2 \, b \, \theta)}$ for any $a, b \in \R$. In particular, for any $0 < a < 1$, the restriction of $\gamma_{(w, 2 \, \theta)}$ in $[0, \, a]$, $\gamma_{(w, 2 \, \theta)}|_{[0, \, a]} = \gamma_{(a \, w, 2 \, a \, \theta)}$ is also abnormal.
\end{corollary}

\medskip

A normal geodesic is called  \textit{strictly normal} if it is not abnormal.
Let $\gamma = \gamma_{(w, 2 \, \theta)}$ (resp. $\gamma(w, 2 \, \theta; \cdot)$) and $0 \le s_1 < s_2 \le 1$ (resp. $0 \le s_1 < s_2 < +\infty$). We consider the restriction of $\gamma$ in $[s_1, \, s_2]$, $\gamma|_{[s_1, \, s_2]}$ as well as
\[
\gamma^{s_1, s_2}(s) := \gamma(s_1)^{-1} \cdot \gamma(s_1 + s(s_2 - s_1)), \qquad s \in [0, \ 1].
\]
By \eqref{GEn}, a simple calculation shows that
\[
\gamma^{s_1,s_2} = \gamma_{((s_2 - s_1) \, e^{2 \, s_1 \widetilde{U}(\theta)} \, w, \ 2 \, (s_2 - s_1) \, \theta)},
\]
which is a normal geodesic starting from $o$. If $\gamma^{s_1, s_2}$ is abnormal,
then it follows from Proposition \ref{cAg} that there exists a $\sigma \in \R^m \setminus \{0\}$ such that
\begin{align*}
(s_2 - s_1)^{k+1} \, U(\sigma) \, U(\theta)^k \, e^{2 \, s_1  \, \widetilde{U}(\theta)} \, w = 0, \qquad \forall \, k \in \N,
\end{align*}
which implies that $\gamma = \gamma_{(w,2 \, \theta)}$ itself (so $\gamma(w,2 \, \theta; \cdot)$) is also abnormal by \eqref{AGc}.
We say a normal geodesic $\gamma_{(w, 2 \, \theta)}$ (resp. $\gamma(w, 2 \, \theta; \cdot)$) \textit{does not contain abnormal segments} if
$\gamma^{s_1,s_2}$ is not abnormal for any $0 \le s_1 < s_2 \le 1$ (resp. $0 \le s_1 < s_2 < +\infty$). In conclusion, we get the following:

\begin{lemma} \label{nLN}
In the framework of step-two groups, any strictly normal geodesic does not contain abnormal segments.
\end{lemma}

It is worthwhile to point out that the above property is no longer valid in general. See \cite{MR20} for more details.

\medskip

In this paper, the \textit{cut locus} of $o$, $\mathrm{Cut}_o$, is defined as
\begin{align} \label{DCUT}
\mathrm{Cut}_o := \mathcal{S}^c, \quad \mbox{ with } \, \mathcal{S} := \{g; \, \mbox{$d^2$ is $C^{\infty}$ in a neighborhood of $g$}\}.
\end{align}

Recall that (see for example \cite[\S~11.1]{ABB20})
\begin{align}\nonumber
\mathcal{S} = \{g; \, & \mbox{there exists a unique shortest geodesic $\gamma$ from $o$ to $g$, which is} \\
\label{CCL1}
&\mbox{ not abnormal, and $g$ is not
conjugate to $o$ along $\gamma$}\},
\end{align}
and it is open and dense in $\G$. Hence $\mathrm{Cut}_o$ is closed. Furthermore, it has measure zero (cf. \cite[Proposition 15]{R13}). Remark also that $o \in \mathrm{Abn}^*_o \subseteq \mathrm{Cut}_o$.

The \textit{classical cut locus} of $o$, $\mathrm{Cut}_o^\mathrm{CL}$ is defined as the set of points where geodesics starting at $o$ cease to be shortest, that is
\begin{align*}
\mathrm{Cut}_o^\mathrm{CL} := \{g; g \mbox{ is the cut point of $o$ along some arclength parametrized normal geodesic}\}.
\end{align*}

Now, we can give
an affirmative answer to the open question \cite[first part of (30)]{BR19} in our framework,  which follows from \eqref{CCL1},
Lemma \ref{nLN} and \cite[Theorem 8.72]{ABB20}.

\begin{theorem}\label{t2}
In the setting of step-two Carnot groups, it holds that $\mathrm{Cut}_o = \mathrm{Cut}_o^{\mathrm{CL}} \cup \mathrm{Abn}_o^*$.
\end{theorem}

\medskip

\subsection{Notations and results from \cite{Li19}} \label{s22}

\medskip

Let us begin by recalling the initial reference set and the reference function, introduced in \cite{Li19}, which are defined respectively by
\begin{gather}
\OA := \left\{\tau \in \R^m; \, \max_{|x| = 1}\langle U(\tau)^2 x,x\rangle < \pi^2\right\} = \{\tau \in \R^m;\, \|U(\tau)\| < \pi\}, \label{oa} \\
\phi(g; \tau) = \langle U(\tau) \cot{U(\tau)} \, x, \ x \rangle + 4 \, t \cdot \tau, \quad \tau \in \OA, \  g = (x, t) \in \G. \label{RFn}
\end{gather}

Notice that the function $\phi(g; \cdot)$ is well-defined provided the spectrum of $U(\tau)$ does not contain any $k \, \pi$ ($k \in \Z \setminus \{ 0 \}$). Also, we will use its usual extension on $\overline{\OA}$ (which is denoted by $\phi(g; \cdot)$ as well). And we have

\begin{proposition}[\cite{Li19}, Proposition~2.1 and Remark~2.1] \label{eP1}
For any $g$, $\phi(g; \cdot)$ is smooth and concave in $\OA$. Moreover, for every $g$, there exists an $\theta_g \in \overline{\OA}$ such that
\[
\phi(g; \theta_g) = \sup_{\tau \in \OA} \phi(g; \tau).
\]
\end{proposition}

Let $\nabla_\theta = \left(\frac{\partial}{\partial \theta_1}, \ldots, \frac{\partial}{\partial \theta_m}\right)$ denote the usual gradient on $\R^m$. Recall that (cf. \cite[\S~2]{Li19})
\begin{align} \label{nOM}
\widetilde{\M} := \left\{\left( x, - \frac{1}{4} \nabla_\theta \langle U(\theta) \, \cot{U(\theta)} \, x, \ x\rangle \right); \ x \in \R^q, \ \theta \in \OA \right\},
\end{align}
which is the union of disjoint and nonempty subsets
\begin{align} \label{m}
\M &:= \left\{g; \ \exists \, \theta \in \OA  \mbox{ s.t. $\theta$ is a nondegenerate critical point of $\phi(g; \cdot)$ in $\OA$}  \right\} \nonumber \\
& = \left\{g; \ \exists \, \theta \in \OA  \mbox{ s.t. the set of global maximizers of $\phi(g; \cdot)$ in $\OA$ is $\{\theta\}$} \right\},
\end{align}
and
\begin{align}\label{dtM2}
\widetilde{\M}_2  := \{g; \ \mbox{the set of global maximizers of $\phi(g; \cdot)$ in $\OA$ has at least two points} \}.
\end{align}

Also recall that $\M$ is an open set, $\M \subseteq \mathcal{S}$, $o \in \widetilde{\M}_2 \subseteq \mathrm{Abn}^*_o \subseteq \mathrm{Cut}_o$ and (cf. \cite[\S~2]{Li19})
\begin{align} \label{Csrd}
d(g)^2 = \max_{\tau \in \OA} \phi(g; \tau) \ \mbox{ for $g \in \widetilde{\M}$, and } \  d(g)^2 = \sup_{\tau \in \OA} \phi(g; \tau) \ \mbox{ for $g \in \overline{\widetilde{\M}}$}.
\end{align}
And we have the following

\begin{theorem}[\cite{Li19}, Theorems~2.4 and 2.5] \label{RLT}
Assume that $\zeta_0 \in \R^q \setminus \{ 0 \}$ and $\theta_0 \in \OA$. Then $\exp\{(\zeta_0, 2 \, \theta_0) \} = g_0 := (x_0, t_0)$ if and only if
\begin{gather*}
x_0 = \left( \frac{U(\theta_0)}{\sin{U(\theta_0)}} e^{-\widetilde{U}(\theta_0)} \right)^{-1} \zeta_0, \quad t_0 = - \frac{1}{4} \nabla_\theta \langle U(\theta_0) \, \cot{U(\theta_0)} \, x_0, \ x_0 \rangle.
\end{gather*}
Furthermore, in such case, we have
\begin{gather*}
d(g_0)^2 = |\zeta_0|^2 = \left| \frac{U(\theta_0)}{\sin{U(\theta_0)}} \, x_0 \right|^2 = \phi(g_0; \theta_0),
\end{gather*}
and the unique shortest geodesic from $o$ to $g_0$ is $\exp\{s \, (\zeta_0, 2 \, \theta_0) \}$ ($0 \le s \le 1$),
which is strictly normal if and only if $g_0 \in \M$.
\end{theorem}

As a consequence, we yield immediately

\begin{corollary} \label{Nc1}
Let $\gamma(s) := \exp\{ s \, (\zeta_0, \tau_0)\}$ be an arclength parametrized geodesic, that is $|\zeta_0| = 1$. Then its cut time $h_{\mathrm{cut}} = +\infty$ if $\tau_0 = 0$, and in such case $\gamma$ is a ray in the first layer. In addition, we have $h_{\mathrm{cut}} \geq 2 \pi/\|U(\tau_0)\|$ when $\tau_0 \neq 0$.
\end{corollary}

\begin{remark} \label{Rkn1}
It follows from  \cite[Proposition~5.1 and/or Corollary~2.2]{Li19} that if $\G$ is not of M\'etivier type, then there exist $0 \neq \theta_0 \in \R^m$ and $0 \neq x_0 \in \ker \, \widetilde{U}(\theta_0)$ such that $|x_0| = 1$ and $\exp\{ s \, (x_0, \theta_0)\} = \exp\{ s \, (x_0, 0)\}$ for all $s > 0$. Hence the cut time of $\exp\{ s \, (x_0, \theta_0)\}$ is equal to $+\infty$ and the statement of \cite[Theorems~6 and 7]{BBG12} is misleadingly phrased. However, a correct statement and their generalization can be found in
Theorem \ref{NTh1}, Corollaries \ref{NcMc} or \ref{NThA1} below. Also notice that our method to determine the cut time is completely different from theirs.
\end{remark}

Another easy but very useful consequence is the following:

\begin{corollary} \label{nCcc}
It holds that $\mathrm{Cut}_o^{\mathrm{CL}} \subseteq \widetilde{\M}^c$.
\end{corollary}

Combining \cite[Proposition~5.1 (b)]{Li19} with Theorem \ref{RLT} as well as Proposition \ref{cAg}, we have the following
characterization of $\widetilde{\M}_2$, $\mathrm{Abn}_o^*$ and $\mathrm{Abn}_o$:

\begin{proposition} \label{NPA1}
It holds that:
\begin{align*}
\widetilde{\M}_2 &=  \left\{\gamma(s) = \gamma(\zeta, \tau; s); \, \mbox{$\gamma$ is abnormal, $|\zeta| = 1$ and $0 \le s < \frac{2\pi}{\|U(\tau)\|}$} \right\}, \\
\mathrm{Abn}_o^* &= \left\{\gamma(s) = \gamma(\zeta, \tau; s); \, \mbox{$\gamma$ is abnormal, $|\zeta| = 1$ and $0 \le s \le h_{\mathrm{cut}}(\gamma)$}  \right\}, \\
\mathrm{Abn}_o &= \left\{\gamma(s) = \gamma(\zeta, \tau; s); \, \mbox{$\gamma$ is abnormal, $|\zeta| = 1$, $s \ge 0$} \right\}.
\end{align*}
\end{proposition}

\begin{remark}
Obviously, the arclength parametrized geodesic $\gamma(\zeta_0, \tau_0; s)$ is abnormal if and only if there exists a $s_0 > 0$ such that $\| U(s_0 \, \tau_0) \| < 2 \pi$ and $\exp\{ s_0 (\zeta_0, \tau_0) \} \in \widetilde{\M}_2$.
\end{remark}

In order to describe $\mathrm{Abn}_o^*$ and $\mathrm{Abn}_o$, it suffices to determine $\widetilde{\M}_2$, which is much less difficult. Moreover, we have

\begin{corollary}\label{M2AbsAb}
In the setting of step-two groups, if $\widetilde{\M}_2 \subseteq \left\{ (x, 0); \ x \in \R^q \right\}$, then $\mathrm{Abn}_o = \mathrm{Abn}_o^* = \widetilde{\M}_2$.
\end{corollary}

\begin{remark} \label{nRna}
(a) Recall that $\R^q \times \{0 \} \subseteq \widetilde{\M}$. If $(x_0, 0) \in \mathrm{Abn}_o$ for some $x_0 \in \R^q$, then we have also $(x_0, 0) \in \widetilde{\M}_2 \subseteq \mathrm{Abn}_o^*$.

(b) In general, $\mathrm{Abn}_o = \mathrm{Abn}_o^*$ does not imply $\widetilde{\M}_2 \subseteq \R^q \times \{0\}$.
In fact, let $\H^3 = \R^2 \times \R$ denote the Heisenberg group of real dimension 3 and consider $\G = \H^3 \times \H^3 \cong \R^4 \times \R^2$. We have
\[
\mathrm{Abn}_o = \mathrm{Abn}_o^* = \left( \{ o_{\H^3} \} \times \H^3 \right) \cup \left( \H^3 \times \{ o_{\H^3} \} \right),
\]
but
\[
 \{ o \} \cup \left( \{ o_{\H^3} \} \times \M_{\H^3} \right) \cup \left( \M_{\H^3} \times \{ o_{\H^3} \} \right) = \widetilde{\M}_2 \not\subseteq \R^4 \times \{ 0 \},
\]
where $\M_{\H^3} = \{(x, t); \ x \in \R^2 \setminus \{ 0 \} \}$ and $o_{\H^3}$ denote the corresponding set $\M$ and identity element in the setting of $\H^3$ respectively.

(c) The example in (b) also provides a group on which $\widetilde{\M}_2 \subsetneqq \mathrm{Abn}_o^*$ since $((0, 0, 1), o_{\H^3}) \in \mathrm{Abn}_o^* \setminus \widetilde{\M}_2 \subseteq \H^3 \times \H^3 = \G$. Furthermore, another
example of $\widetilde{\M}_2 \subsetneqq \mathrm{Abn}_o^*$ can be found in the proof of Proposition \ref{p5} (see Subsection \ref{s71} below), and that of $\mathrm{Abn}_o^* \subsetneqq \mathrm{Abn}_o$ in Subsection \ref{s72} below.

(d) Obviously, a step-two group is of M\'etivier type if and only if $\mathrm{Abn}_o^* = \{o\}$.
\end{remark}

\medskip

Recall that the nonempty open subset $\M \subset \G$ is the set of points, $g$, where the reference function $\phi(g; \cdot)$ has a nondegenerate critical point in the initial reference set $\OA$. Observe that $\M$ is symmetric and scaling invariant; namely, if $g \in \M$, then we have $g^{-1} = -g \in \M$ and $\delta_r(g) \in \M$ for all $r > 0$.
A step-two group $\G$ is said to be a \textit{GM-group} (or \textit{of type GM}) if it satisfies
$$\overline{\M} = \G.  \eqno(GM) $$
Notice that if both $\G_1$ and $\G_2$ satisfy ($GM$), then so does the direct product $\G_1 \times \G_2$. See Appendix B for more details. Also remark that GM groups form a wild set. Indeed, for any given $\G(q, m, \U)$, we can construct an uncountable number of GM-groups $\G(q + 2 n, m, \widetilde{\U})$. See \cite[\S~8.1]{Li19} for more details.

\medskip

Recall that the global reference set is a compact set in $\R^m$ defined by (cf. \cite[\S~2.6]{Li19})
\begin{align}\label{GRs}
\mathfrak{R} := \overline{\mathcal{R}}, \quad \mbox{with } \mathcal{R} := \left\{\theta = \frac{1}{4}\nabla_t d(g)^2; \, g = (x,t) \notin \mathrm{Cut}_o \right\} \ \mbox{open}.
\end{align}
It follows from \cite[\S~3]{MM16} or \cite[Proposition~5.2]{Li19} that $\mathcal{R} \cap \OA$ is dense in $\OA$. Then
\begin{align}\label{grinc}
\mathfrak{R} \supseteq \overline{\OA} = \left\{\theta; \ \|U(\theta)\| \le \pi \right\}.
\end{align}

\medskip

Set
\begin{align*}
\mathcal{V}  := \{\vartheta \in \R^m; \,
\det(k\pi - U(\vartheta)) \neq 0, \, \forall \, k \in \N^*\}
\end{align*}
and
\begin{align} \label{nW}
\mathcal{W} := \exp(\R^q \times (2 \, \mathcal{V}^c) ),
\end{align}
which is the set of the endpoints of ``bad'' normal geodesics, where \textit{``bad'' normal geodesic} (resp. \textit{``good'' normal geodesic}) means $\gamma = \gamma_{(w,2 \, \theta)}$ with $\theta \in \mathcal{V}^c$ (resp. $\theta \in \mathcal{V}$).
It is clearly that $\mathcal{W}$ is of measure zero.

\medskip

Finally some notations of special functions related to $-s \cot{s}$ are also recalled:
\begin{align}\label{EFs}
f(s) := 1 - s \cot{s}, \quad \mu(s) := f'(s) = \frac{2 s - \sin{(2 s)}}{2 \sin^2{s}}, \quad \psi(s) := \frac{f(s)}{s^2} .
\end{align}

\subsection{Main results}

\medskip

Our first result is the following:

\subsubsection{Properties of the global reference set $\mathfrak{R}$}

The following theorem should be useful to determine all shortest geodesics in the setting of step-two groups.

\begin{theorem} \label{t1}
Let $o \neq g \in \G$, and $\gamma_g(s)$ ($0 \le s \le 1$) be a shortest geodesic joining $o$ to $g$. Then there exist $\zeta \in \R^q$ with $|\zeta| = d(g)$ and $\theta \in \mathfrak{R}$ such that $\gamma_g(s) = \exp\{s \, (\zeta, 2 \, \theta)\}$ for all $0 \le s \le 1$.
\end{theorem}

The following property is a direct consequence of \eqref{CCL1} together with Lemma \ref{nLN} and \cite[Theorem 8.72]{ABB20}:

\begin{lemma} \label{bLn}
Suppose that $g_0 = \exp\{(\zeta_0, \tau_0)\} \in \mathcal{S}$ and $\exp\{s \, (\zeta_0, \tau_0)\}$ ($0 \le s \le 1$) is the unique shortest geodesic between $o$ and $g_0$. Then $\exp\{s \, (\zeta_0, \tau_0)\} \in \mathcal{S}$ and $ 2^{-1} \, s \, \tau_0 \in \mathcal{R}$ for all $0 < s \le 1$.
\end{lemma}

\begin{remark} \label{nRnM}
Recall that $\widetilde{\M} = \M \cup \widetilde{\M}_2$ and $\widetilde{\M}_2 \subseteq \mathrm{Cut}_o = \mathcal{S}^c$. By the fact that $\exp\{(\zeta, 2 \, \tau)\} \in \widetilde{\M}$ whenever $\tau \in \OA$, a direct consequence of Lemma \ref{bLn} is that $\M \neq \emptyset$. See also \cite[Proposition~6 and \S~3]{MaM16} for another explanation.
\end{remark}

It follows from Lemma \ref{bLn} that $\mathfrak{R}$ is star-shaped w.r.t. the origin $0$, that is, if $\tau \in \mathfrak{R}$, then we have $s \, \tau \in \mathfrak{R}$ for all $s \in [0, \ 1]$; in particular, it is path connected. Furthermore, it is the smallest compact set which satisfies the property in Theorem \ref{t1}.
Also notice that Lemma \ref{bLn} provides a theoretical basis for the method proposed in \cite[\S~11]{Li19} to determine the squared sub-Riemannian distance for general non-GM groups.

\begin{remark}\label{RKn2}
Theorem \ref{t1} could be considered as a somewhat converse statement of \cite[Proposition~4]{RT05} in our setting. In fact, from the proof of Theorem \ref{t1}, we know that there exist $\{g_j \}_{j = 1}^{+\infty} \subseteq \mathcal{S}$ with
$\gamma_{(\zeta^{(j)}, 2 \, \theta^{(j)})}$
the shortest geodesic joining $o$ to $g_j$ such that $g_j \to g$ and
$(\zeta^{(j)}, 2 \, \theta^{(j)}) \to (\zeta, 2 \, \theta)$
as $j \to +\infty$. As a result, every shortest geodesic can be induced by some limiting sub-differential in our situation.
\end{remark}

Combining this with another basic property of $\mathfrak{R}$, namely \cite[Corollary~2.4]{Li19}, this is why it is called the global reference set.

\subsubsection{Other sub-Riemannian geometric properties on step-two groups}

Let us begin with an upper bound about the cut time of an arclength parametrized geodesic:

\begin{corollary}\label{c1}
Let
\begin{align*}
\mathrm{C}_{\mathfrak{R}} := \max_{\tau \in \mathfrak{R}} \|U(\tau)\|, \quad \mathrm{C}_{\tau} := \sup\left\{ s > 0; \ r \, \tau \in \mathfrak{R}, \  \forall \, 0 \le r \le s \right\} \ (\tau \in \R^m).
\end{align*}
For any arclength parametrized geodesic $\gamma(s) = \exp\{ s \, (\zeta, \tau)\} = \gamma(\zeta, \tau; s)$,
its cut time satisfies
\begin{align} \label{nCUTUn}
h_{\mathrm{cut}} &\le \sup\left\{2 \, \mathrm{C}_{\sigma} ; \, \sigma \in \R^m, \, \gamma(\zeta, \sigma; \cdot) = \gamma(\zeta, \tau; \cdot)  \right\} \nonumber\\
& \le \sup\left\{ \frac{2 \, \mathrm{C}_{\mathfrak{R}} }{\| U(\sigma) \|} ; \, \sigma \in \R^m, \, \gamma(\zeta, \sigma; \cdot) = \gamma(\zeta, \tau; \cdot)  \right\}
\end{align}
with the understanding $\frac{2 \, \mathrm{C}_{\mathfrak{R}}}{0} = +\infty$.
In particular, assume moreover that $\gamma(s) = \exp\{ s \, (\zeta, \tau)\} $ is not abnormal, then $h_{\mathrm{cut}} \le 2 \, \mathrm{C}_{\tau} \le \frac{2 \, \mathrm{C}_{\mathfrak{R}}}{\|U(\tau)\|}$.
\end{corollary}

\begin{remark} \label{nRKnSM}
Fix $(\zeta, \tau)$ and let $\Pi_{(\zeta, \tau)} := \{\sigma - \tau \in \R^m; \, \gamma(\zeta, \sigma; \cdot) = \gamma(\zeta, \tau; \cdot)\}$. Proposition \ref{cAg} implies that $\Pi_{(\zeta, \tau)}$ is a linear subspace of $\R^m$. It is clear that the continuous function $\sigma \mapsto \|U(\sigma)\|$ defined on $\tau + \Pi_{(\zeta, \tau)}$ attains its minimum. Hence, the last ``$\  \sup$'' in \eqref{nCUTUn} can be replaced by ``$\ \max$''.
\end{remark}

Moreover, we have the following:

\begin{lemma}\label{NLA1} {\footnote{For a general sub-Riemannian manifold $M$, we can define the shortest abnormal set of $y \in M$, $\mathrm{Abn}_y^*$, as the set of the endpoints of abnormal (not necessarily normal) shortest geodesics starting from $y$ and $\mathrm{Abn}_y^*$ is a closed set as well. This is a result of the characterization of abnormal Pontryagin extremals via Lagrange multipliers rule and the compactness of minimal controls. We would like to thank L. Rizzi for informing us of this general result and providing a sketched proof. For the sake of completeness, we give a proof in the setting of step-two groups without using the notion of the endpoint map in Section \ref{nsP}.}}
In the framework of step-two Carnot groups, $\mathrm{Abn}_o^*$ is a closed set.
\end{lemma}

On a step-two group $\G$, notice that $o \not\in \mathrm{Cut}_o^{\mathrm{CL}}$. Now assume that $\mathrm{Cut}_o^{\mathrm{CL}} \cap \mathrm{Abn}_o^* = \emptyset$ and $\mathrm{Abn}_o^* \neq \{ o \}$. Let $\gamma(s) = \exp\{ s \, (\zeta, \tau)\} = \gamma(\zeta, \tau; s)$ be an arclength parametrized abnormal geodesic. Then it follows from Lemma \ref{NLA1} that its cut time is $+\infty$. Using Corollary \ref{c1} and Remark \ref{nRKnSM}, we obtain $\gamma(\zeta, \tau; \cdot) = \gamma(\zeta, 0; \cdot)$, which implies $\{ \gamma(s); \, s \ge 0\} \subseteq \R^q \times \{0\}$ by \eqref{GEn}. In conclusion, combining this with Corollaries \ref{M2AbsAb} and \ref{nCcc}, we get the following:

\begin{corollary}\label{NCorA1}
In the setting of step-two Carnot groups,
$\mathrm{Cut}_o^{\mathrm{CL}} \cap \mathrm{Abn}_o^* = \emptyset$ if and only if $\widetilde{\M}_2 \subseteq \R^q \times \{0\}$.
\end{corollary}

To finish this subsection, we provide the following:

\begin{proposition}\label{CLMMW}
In the context of step-two Carnot groups, it holds that $\overline{\widetilde{\M}} = \overline{\M}$.
\end{proposition}

\subsubsection{Characterizations of GM-groups}

GM-groups have very fine properties of sub-Riemannian geometry. More precisely,

\begin{theorem} \label{NTh1}
The following properties are equivalent:
{\em\begin{compactenum}[(i)]
\item $\G$ is of type GM;
\item $\widetilde{\M}$ is dense in $\G$;
\item $d(g)^2 = \sup\limits_{\theta \in \OA} \phi(g; \theta)$ for all $g \in \G$;
\item The global reference set $\mathfrak{R}$ is equal to $\overline{\OA} = \left\{\theta; \ \|U(\theta)\| \le \pi \right\}$;
\item For any arclength parametrized, strictly normal geodesic $\gamma(s) = \exp\{ s \, (\zeta, \tau)\}$,
    its cut time is equal to $h_{\mathrm{cut}}(\tau) := 2 \pi/\|U(\tau)\|$, with the understanding $h_{\mathrm{cut}}(0) = +\infty$;
\item $\mathrm{Cut}_o^c \cap \partial \M = \emptyset$.
\end{compactenum}}
\end{theorem}

\begin{remark}
(1) The condition (ii) should be the easiest to check among all those.

(2) In Section \ref{s3}, we can find that every step-two group of corank 1 or 2 is of type GM. As a result, from the above Property (v), we obtain the cut time of any arclength parametrized, strictly normal geodesic on $\G(q, m, \U)$ with $m = 1$ or $2$, which coincides with that in \cite[Theorems~6 and 7]{BBG12}. However, considering Remark \ref{Rkn1}, their results for the cut time of abnormal geodesics need more explanations. See (ii) of Corollary \ref{NThA1} and Remark \ref{nRKn9} below for more details. Furthermore, we emphasize that Property (v)
is an equivalent characterization of GM-groups and thus we have found all possible step-two groups satisfying this fine property.
\end{remark}

Recall that $\M$ is an open set.
As a consequence, we obtain the following improvement of \cite[Theorem~2.7]{Li19}:

\begin{corollary}\label{Nc2}
A step-two group $\G$ is of type GM if and only if $\mathrm{Cut}_o = \partial \M$.
\end{corollary}

\begin{corollary} \label{NcMc}
Assume that $\G$ is of M\'etivier type. Then $\G$ is a GM-group iff. for any arclength parametrized geodesic $\gamma(s) = \exp\{ s \, (\zeta, \tau)\}$, its cut time is $2 \pi/\|U(\tau)\|$.
\end{corollary}

If $\G$ is of GM-type, then our Theorem \ref{t1} can be improved. Indeed, the parameter $\theta$ can be further chosen as a maximum point of the reference function $\phi(g;\cdot)$ on $\overline{\OA}$.
In other words, we have

\begin{theorem}\label{nThm1}
Assume that $\G$ is of GM-type and $o \neq g \in \G$. For any shortest geodesic $\gamma_g(s)$ ($0 \le s \le 1$) joining $o$ to $g$, there exist $\zeta \in \R^q$ and $\theta \in \overline{\OA}$ such that
\[
\phi(g;\theta) = d(g)^2 = \sup\limits_{\tau \in \OA} \phi(g; \tau) = |\zeta|^2, \quad \gamma_g(s) = \exp\{s \, (\zeta, 2 \, \theta)\}, \  \forall \, 0 \le s \le 1.
\]
Moreover, we have that $\theta \in \partial \OA$ if $g \in \widetilde{\M}^c$.
\end{theorem}

\begin{remark}
Suppose further that $\overline{\OA}$ is strictly convex, namely, for any $\tau \neq \tau' \in \overline{\OA}$ and $0 < s < 1$, we have $s \, \tau + (1 - s) \, \tau' \in \OA$. For example, all M\'etivier groups satisfy this property, see \cite[Lemma~9.1]{Li19}. Then, for any $g_0 \in \widetilde{\M}^c$, the concave function $\phi(g_0; \cdot)$ has a unique maximum point on $\partial \OA$. This simple observation is very useful to determine all shortest geodesic(s) from $o$ to $g_0 \in \widetilde{\M}^c$.
\end{remark}

Note that in (v) of Theorem \ref{NTh1}, we only consider the cut time of any strictly normal geodesic. However, we can characterize GM-groups via the optimal synthesis from $o$, as well as the classical cut locus of $o$. More precisely, we have the following:

\begin{corollary} \label{NThA1}
The following properties are equivalent:
{\em\begin{compactenum}[(i)]
\item $\G$ is of type GM;
\item For any arclength parametrized geodesic $\gamma(s) = \exp\{ s \, (\zeta, \tau)\} = \gamma(\zeta, \tau; s)$,
    its cut time is given by
    \begin{align*}
    h_{\mathrm{cut}} = \max\left\{ \frac{2\pi}{\| U(\sigma) \|} ; \, \sigma \in \R^m, \, \gamma(\zeta, \sigma; \cdot) = \gamma(\zeta, \tau; \cdot)  \right\},
    \end{align*}
   with the understanding $\frac{2\pi}{0} = +\infty$;
\item $\mathrm{Cut}_o^{\mathrm{CL}} = \widetilde{\M}^c$.
\end{compactenum}}
\end{corollary}

\begin{remark} \label{nRKn9}
It follows from (ii) of Corollary \ref{NThA1} that the statement in \cite[Theorems~6 and 7]{BBG12} is correct if we choose the covector suitably. To be more precise, let $\G$ be a step-two group of corank 1 or 2, so it is a GM-group from Corollary \ref{c2} in Section \ref{s3} below. Let $\gamma(s) = \exp\{ s \, (\zeta, \tau)\} = \gamma(\zeta, \tau; s)$ be an arclength parametrized geodesic on $\G$. Recall that $\Pi_{(\zeta, \tau)} := \{\sigma - \tau \in \R^m; \, \gamma(\zeta, \sigma; \cdot) = \gamma(\zeta, \tau; \cdot)\}$ is a linear subspace. Assume further that the continuous function $\sigma \mapsto \|U(\sigma)\|$ defined on $\tau + \Pi_{(\zeta, \tau)}$ attains its minimum at $\tau$. Then the cut time of $\gamma$ is given by $2 \pi/\|U(\tau)\|$.
\end{remark}

\subsection{On the lack of semi-concavity of $d^2$ on $\widetilde{\M}_2$} \label{s23}

Recall that $o \in \widetilde{\M}_2 \subset \G$ is the set of points, $g$, where the reference function $\phi(g; \cdot)$ has a degenerate (so infinite) critical point in the initial reference set $\OA$.
The following theorem is a kind of generalization of the first result in \cite[Theorem 1.1]{MM16}, which will be proven
by a completely different method.

\begin{theorem}\label{p1}
Let $g_0 = (x_0, t_0) \in \widetilde{\M}_2$. Then there exist a unit vector $\nu_0 \in \R^m$ and a constant $c_0 > 0$ such that
\begin{align}\label{failconc}
d(x_0, t_0 + h \, \nu_0)^2 + d(x_0, t_0 - h  \, \nu_0)^2 - 2 \, d(x_0, t_0)^2 \ge c_0 \, h, \qquad \forall \, h > 0.
\end{align}
\end{theorem}

In particular, together with (a) in Remark \ref{nRna} and \eqref{te1}, it gives immediately

\begin{corollary}
If $g_0 = (x_0, 0) \in \mathrm{Abn}_o^*$, there exist a unit vector $\nu_* \in \R^m$ and a constant $c_* > 0$ such that
\begin{align}\label{failconc0}
d(x_0, h  \, \nu_*)^2 -  d(x_0, 0)^2 \ge c_* |h|, \qquad \forall \, h \in \R.
\end{align}
\end{corollary}

A direct consequence of \eqref{failconc} is the following significantly weaker estimate:
\begin{align} \label{nMBn}
\limsup_{g' \longrightarrow o} \frac{d(g_0 + g')^2 + d(g_0 - g')^2 - 2 \, d(g_0)^2}{|g'|^2} = +\infty,
\end{align}
which implies, from \eqref{DSSCC}, the lack of semi-concavity of $d^2$ for any $g_0 \in \widetilde{\M}_2$.

It follows from Lemma \ref{NLA1} that $\overline{\widetilde{\M}_2} \subseteq \mathrm{Abn}_o^*$. A very interesting phenomenon is that \eqref{nMBn} can be no longer valid for $g_0 \in \overline{\widetilde{\M}_2} \setminus \widetilde{\M}_2$ even in the setting of GM-groups. A concrete example will be provided in Subsection \ref{s71} below.

However, recall that our $\mathrm{SC}^-_o$ is defined by \eqref{DoFSC1} instead of
as the set of points where \eqref{nMBn} satisfies. Obviously, $\mathrm{SC}^-_o$ is closed. When the underlying group is of type GM, combining (ii) of Corollary \ref{NThA1} with Theorem \ref{p1} obtained above,  we can characterize $\mathrm{Abn}^*_o$ via $\widetilde{\M}_2$; as a byproduct, we answer the open problem \cite[(29)]{BR19} affirmatively:

\begin{theorem}\label{NThA2}
In the setting of GM-groups, it holds that $\mathrm{SC}^-_o = \mathrm{Abn}^*_o  = \overline{\widetilde{\M}_2}$.
\end{theorem}

Furthermore, combining Theorem \ref{NThA2}, Lemmas \ref{nLN} and \ref{NLA1} with \cite[Corollary 30]{BR19} (with little modification in its proof), we answer the open question \cite[second part of (30)]{BR19}, when the underlying group is GM-group:

\begin{corollary}
In the framework of GM-groups, we have $\mathrm{Cut}_o = \mathrm{SC}_o^+ \cup \mathrm{SC}_o^-$.
\end{corollary}

Similarly, by Corollaries \ref{M2AbsAb} and \ref{NCorA1}, we have the following result that provides an affirmative, also partial, answer to the open questions \cite[(29)-(30)]{BR19}:

\begin{corollary}
Let $\G$ be a step-two group such that $\widetilde{\M}_2 \subseteq \left\{ (x, 0); \ x \in \R^q \right\}$. Then it holds that $\mathrm{SC}^-_o = \mathrm{Abn}^*_o$ ($= \widetilde{\M}_2$) and $\mathrm{Cut}_o = \mathrm{SC}_o^+ \cup \mathrm{SC}_o^-$.
\end{corollary}

Recall that $\widetilde{\M}_2 = \{ o \}$ if and only if $\G$ is of M\'etivier type. In the sequel, a step-two group $\G$ is said to be a \textit{SA-group} (or \textit{of type SA}) if it satisfies
$$\{ o \} \neq \widetilde{\M}_2 \subseteq \left\{ (x, 0); \ x \in \R^q \right\}.  \eqno(SA) $$
Notice that star graphs and $N_{3, 2}$, namely the free step-two Carnot group with three generators, are SA-groups. See \cite{Li19} or \S \ref{s6} and \S \ref{s5} below for more details.  Also remark that star graphs are both GM-groups and SA-groups, $N_{3, 2}$ is the simplest example of SA-group that is not of type GM. Of course, we can provide an uncountable number of SA but not GM groups.

Furthermore, a trivial method to construct SA-groups can be found in Proposition \ref{NSAa} of Appendix B. See Appendix C for another method which is much more meaningful. In particular, SA-groups of corank $1$ are the direct product of a Euclidean space $\R^k$ with a generalized Heisenberg group. However, for any $m \geq 2$, SA-groups with corank $m$ form a very complicated set.

\medskip

To finish this section, we point out the following facts:
\begin{enumerate}[(1)]
  \item The second result of \cite[Theorem 1.1]{MM16} says that on the free step-two Carnot group with $k$ ($k \ge 4$) generators, $N_{k, 2}$,  for any $g_0 = (x_0, t_0) \in \mathrm{Abn}_o^*$, there exist a unit vector $\nu_* \in \R^{k \, (k - 1)/2}$ and a constant $c_* > 0$ such that we have the following counterpart of \eqref{failconc0}:
\[
d(x_0, t_0 + h  \, \nu_*) -  d(x_0, t_0) \ge c_* |h|, \qquad - c_* \le h \le c_*.
\]
Hence, \eqref{nMBn} is valid on $N_{k, 2}$ for any $g_0 \in \mathrm{Abn}_o^*$.
  \item $N_{k, 2}$ ($k \ge 4$) neither is a GM-group nor satisfies $\widetilde{\M}_2 \subseteq \left\{ (x, 0); \ x \in \R^k \right\}$.
\end{enumerate}

\medskip

\renewcommand{\theequation}{\thesection.\arabic{equation}}
\section{Proof of main results}\label{nsP}
\setcounter{equation}{0}

\medskip

\subsection{Proof of Theorem \ref{t1}}

\begin{proof}
We first assume that the shortest geodesic joining $o$ to $g$ is unique. By \cite[Corollary 2.4]{Li19}, there exist $w \in \R^q$ with $|w| = d(g)$ and $\theta \in \mathfrak{R}$ such that $\gamma_{(w,2 \, \theta)}$ is a shortest geodesic joining $o$ to $g$. By uniqueness we must have $\gamma_g(s) = \gamma_{(w,2 \, \theta)}(s) = \exp(s\, (w,2 \, \theta))$ for all $0 \le s \le 1$, which ends the proof in this case.

In general, from the characterization of the shortest geodesic in Subsection \ref{ns221}, there exist $w_* \in \R^q$ with $|w_*| = d(g)$ and  $\theta_* \in \R^m$ (not necessarily belonging to $\mathfrak{R}$) such that $\gamma_g = \gamma_{(w_*,2 \, \theta_*)}$. We now prove that for each $s_* \in (0, \, 1)$,
the restriction of $\gamma_g$ on the interval $[0, \ s_*]$,  $\gamma_g|_{[0, \, s_*]} = \gamma_{(s_* \, w_*,2 \, s_* \, \theta_*)} := (\gamma_g)^{0,s_*}$
is the unique shortest geodesic joining $o$ to $\gamma_g(s_*) = \exp(s_* \, w_*,2 \, s_* \, \theta_*)$. Notice that $(\gamma_g)^{0,s_*}$ is a shortest geodesic joining $o$ to $\gamma_g(s_*)$ since $\gamma_g$ itself is shortest. To prove uniqueness, we argue by contradiction. Assume that there is another shortest geodesic $\gamma_{s_*} \neq (\gamma_g)^{0,s_*}$ joining $o$ to $\gamma_g(s_*)$ with constant speed $s_*|w_*|$. Then we construct a horizontal curve with constant speed $|w_*| = d(g)$ defined by
\begin{align*}
\widetilde{\gamma}_{s_*}(s) := \left\{ \begin{array}{ll}
\gamma_{s_*}\left(\frac{s}{s_*}\right), \qquad & 0 \le s \le s_*, \\[2mm]
\gamma_g(s),  \qquad  & s_* \le s \le 1. \\
\end{array} \right.
\end{align*}
Obviously,  $\widetilde{\gamma}_{s_*}$ is a shortest geodesic joining $o$ to $g$ as well. Again from the characterization of the shortest geodesic in Subsection \ref{ns221}, there exist $w_{**} \in \R^q$ with $|w_{**}| = d(g)$ and $\theta_{**} \in \R^m$ such that $\widetilde{\gamma}_{s_*} = \gamma_{(w_{**},2 \, \theta_{**})}$.

By the fact that $\gamma_g(s) = \widetilde{\gamma}_{s_*}(s)$ when $s \in [s_*, \ 1]$,
it follows from Lemma \ref{BPG} that $\gamma_g$ and $\widetilde{\gamma}_{s_*}$
coincide on the whole interval $[0, \ 1]$. In particular, $\gamma_g(s) = \widetilde{\gamma}_{s_*}(s)$ for all $0 \le s \le s_*$, that is  $(\gamma_g)^{0,s_*} = \gamma_{s_*}$, which contradicts with our assumption.

For each $s_* \in (0, \ 1)$, what we have proven at the beginning shows that there exist $w(s_*) \in \R^q$ with $|w(s_*)| = s_*d(g)$ and  $\theta(s_*) \in \mathfrak{R}$ such that
\begin{align}\label{equstar}
\exp(s \, (w(s_*),2 \, \theta(s_*))) = (\gamma_g)^{0,s_*}(s) = \exp(s \, (s_* \, w_*,2 \, s_* \, \theta_*)), \quad \forall \, 0 \le s \le 1.
\end{align}
From compactness of $\mathfrak{R}$, we extract a sequence $\{s_j\}_{j = 1}^{+ \infty} \subseteq (0, \ 1)$ such that $s_j \to 1$, $w(s_j) \to w$ with $|w| = d(g)$ and $\theta(s_j) \to \theta \in \mathfrak{R}$ as $j \to +\infty$. With $s_*$ replaced by $s_j$ in \eqref{equstar} and letting $j \to +\infty$, we obtain that $\exp(s \, (w_*,2 \, \theta_*)) = \gamma_g(s) = \exp(s \, (w,2 \, \theta))$ for all $0 \le s \le 1$, which ends the proof of the theorem.
\end{proof}

\subsection{Proof of Corollary \ref{c1}}

\begin{proof}
For convenience, we set
\begin{align*}
\widetilde{h} = \widetilde{h}(\zeta,\tau):= \sup\left\{ 2 \, \mathrm{C}_{\sigma} ; \, \sigma \in \R^m, \, \gamma(\zeta, \sigma; \cdot) = \gamma(\zeta, \tau; \cdot)  \right\}.
\end{align*}

For any $s_* \in (0, \, h_{\mathrm{cut}})$, we know that $\gamma_{(s_* \, \zeta, \, s_* \, \tau)}$ is a shortest geodesic. It follows from Theorem \ref{t1} that there exist $\zeta_{(s_*)} \in \R^q$ and $\theta_{(s_*)} \in \mathfrak{R}$ such that $\gamma_{(s_* \, \zeta, \, s_* \, \tau)} = \gamma_{(\zeta_{(s_*)}, \, 2 \, \theta_{(s_*)})}$. Then Lemma \ref{BPG} implies that $s_* \, \zeta = \zeta_{(s_*)}$ and
\begin{align*}
\gamma(\zeta, \, \sigma_{(s_*)}; \cdot) = \gamma(\zeta, \tau; \cdot), \qquad \mbox{with $\sigma_{(s_*)} := \frac{2 \, \theta_{(s_*)}}{s_*}$}.
\end{align*}

Recalling that $\mathfrak{R}$ is star-shaped w.r.t. $0$, by the fact that $\theta_{(s_*)} = \frac{s_* \, \sigma_{(s_*)}}{2} \in  \mathfrak{R}$, we have
\begin{align*}
\frac{s_*}{2} \le \mathrm{C}_{\sigma_{(s_*)}} \le \frac{\widetilde{h}}{2},
\end{align*}
which implies $s_* \le \widetilde{h}$. Since $s_* \in (0, \, h_{\mathrm{cut}})$ is arbitrary, we obtain $h_{\mathrm{cut}} \le \widetilde{h}$. To prove the second inequality, it suffices to observe that for each $\sigma \in \R^m$, we have $\mathrm{C}_\sigma \le \frac{\mathrm{C}_{\mathfrak{R}}}{\|U(\sigma)\| }$ and this finishes the proof of Corollary \ref{c1}.
\end{proof}

\subsection{Proof of Lemma \ref{NLA1}}

\begin{proof}

For any $\{g_j\}_{j = 1}^{+\infty} \subseteq \mathrm{Abn}_o^*$ such that $g_j \to g \in \G$ as $j \to +\infty$, our aim is to prove $g \in \mathrm{Abn}_o^*$ as well. For each $j \in \N^*$, let $\gamma_j(s)$ ($0 \le s \le 1$) be an abnormal shortest geodesic joining $o$ to $g_j$. By
Theorem \ref{t1}, there exist $\zeta^{(j)} \in \R^q$ with $|\zeta^{(j)}| = d(g_j)$ and $\theta^{(j)} \in \mathfrak{R}$ such that $\gamma_j = \gamma_{(\zeta^{(j)}, \, 2 \, \theta^{(j)})}$. Since $\gamma_{(\zeta^{(j)}, \, 2 \, \theta^{(j)})}$ is abnormal, from Proposition \ref{cAg} there exists a $\sigma^{(j)} \in \mathbb{S}^{m - 1}$ such that
\begin{align}\label{Abnj}
U(\sigma^{(j)}) \, U(\theta^{(j)})^k \, \zeta^{(j)} = 0, \quad \forall \, k \in \N.
\end{align}

Notice that $|\zeta^{(j)}| = d(g_j) \to d(g)$ as $j \to +\infty$. From compactness, up to subsequences, we may assume that $\zeta^{(j)} \to \zeta_0$ with $|\zeta_0| = d(g)$, $\theta^{(j)} \to \theta_0 \in \mathfrak{R}$ and $\sigma^{(j)} \to \sigma_0 \in \mathbb{S}^{m - 1}$ as $j \to +\infty$. Observe that
\begin{align*}
\gamma_{(\zeta_0, 2 \, \theta_0)}(1) = \exp(\zeta_0, 2 \, \theta_0) = \lim_{j \to +\infty} \exp(\zeta^{(j)}, 2 \, \theta^{(j)}) = \lim_{j \to +\infty} g_j = g.
\end{align*}
By the fact that $|\zeta_0| = d(g)$, we obtain that $\gamma_{(\zeta_0, 2\, \theta_0)}$ is a shortest geodesic joining $o$ to $g$. It remains to prove that $\gamma_{(\zeta_0, 2\, \theta_0)}$ is also abnormal. In fact, letting $j \to +\infty$ in \eqref{Abnj}, we get
\begin{align*}
U(\sigma_0) \, U(\theta_0)^k \, \zeta_0 = 0, \quad \forall \, k \in \N,
\end{align*}
which implies $\gamma_{(\zeta_0, 2\, \theta_0)}$ is abnormal by Proposition \ref{cAg}.

This ends the proof of Lemma \ref{NLA1}.
\end{proof}

\subsection{Proof of Proposition \ref{CLMMW}}

\begin{proof}
Set
\begin{align*}
\Xi_1 &:= \{(x,\theta) \in \R^q \times \OA; \, \det(-\mathrm{Hess}_\theta \langle U(\theta) \, \cot{U(\theta)} \, x, \ x\rangle) > 0\}, \\
\Xi_2 &:= \{(x,\theta) \in \R^q \times \OA; \, \det(-\mathrm{Hess}_\theta \langle U(\theta) \, \cot{U(\theta)} \, x, \ x\rangle) = 0\},
\end{align*}
and the map
\begin{align*}
\kappa: \R^q \times \OA & \longrightarrow \widetilde{\M} \\
(x,\theta) &\longmapsto \left(x, -\frac{1}{4} \nabla_\theta \langle U(\theta) \, \cot{U(\theta)} \, x, \ x \rangle\right).
\end{align*}

It follows from Proposition \ref{eP1} that $\Xi_1 \cup \Xi_2 = \R^q \times \OA$. Recall that (cf. \eqref{nOM}-\eqref{dtM2}) $\kappa(\Xi_1) = \M$ and $\kappa(\Xi_2) = \widetilde{\M}_2$.

Observe that the function $(x, \, \theta) \mapsto \det(-\mathrm{Hess}_\theta \langle U(\theta) \, \cot{U(\theta)} \, x, \ x\rangle)$ is real analytic in $\R^q \times \OA$. We claim that $\Xi_2$ has an empty interior.
Otherwise $\Xi_2$ should be $\R^q \times \OA$ (cf. \cite[\S~3.3 (b)]{J82}), which means $\Xi_1 = \emptyset$ and thus $\M = \emptyset$. This leads to a contradiction since $\M \ne \emptyset$ from Remark \ref{nRnM}. In conclusion, $\Xi_1$ is dense in $\R^q \times \OA$.

Now, we shall show that $\widetilde{\M}_2 \subseteq \overline{\M}$. Fix
$(x, t) \in \widetilde{\M}_2$. There exists a $\theta \in \OA$ such that $(x, \, \theta) \in \Xi_2$ and $\kappa(x, \, \theta) = (x, \, t)$. Since $\Xi_1$ is dense in $\R^q \times \OA$, there are $\{(x^{(j)}, \, \theta^{(j)})\}_{j = 1}^{+\infty} \subseteq \Xi_1$ such that $(x^{(j)}, \, \theta^{(j)}) \to (x, \, \theta)$ as $j \to +\infty$. Hence, $\M \ni \kappa(x^{(j)}, \, \theta^{(j)}) \to \kappa(x, \, \theta) = (x, \, t)$ as $j \to +\infty$. As a result, we obtain that $\widetilde{\M}_2 \subseteq \overline{\M}$ and thus $\overline{\widetilde{\M}} = \overline{\M}$, which ends the proof of the proposition.
\end{proof}

\subsection{Proof of Theorem \ref{NTh1}}

\begin{proof}
(i) $\Rightarrow$ (ii): This is evident.

(ii) $\Rightarrow$ (iii): Just use \eqref{Csrd}.

(iii) $\Rightarrow$ (iv): For any given $g = (x,t) \in \mathcal{S}$,  under our assumption, it follows from Proposition \ref{eP1} that there exists a $\theta_0 \in \overline{\OA}$ such that
\begin{align}\label{tmin}
d(x,t)^2 = \phi((x,t);\theta_0) = \langle U(\theta_0) \, \cot{U(\theta_0)} \, x, \ x \rangle + 4 \, t \cdot \theta_0.
\end{align}

Since $\mathcal{S}$ is open, there exists a $r_0 > 0$ such that \[
\{x\} \times B(t,r_0) =  \{x\} \times \{\tau; \, |\tau - t| < r_0\} \subseteq \mathcal{S}.
\]
Then it follows from (iii) that we have for $s \in B(t,r_0)$,  \begin{align}\label{sgmin}
d(x,s)^2 \ge \phi((x,s);\theta_0) = \langle U(\theta_0) \, \cot{U(\theta_0)} \, x, \ x \rangle + 4 \, s \cdot \theta_0.
\end{align}
So, the function $s \longmapsto d(x,s)^2 - 4 \, s \cdot \theta_0$ has a local minimum at the point $s = t$. As a result, we have $\frac{1}{4}\nabla_t d(g)^2 = \theta_0 \in \overline{\OA}$ and consequently $\mathfrak{R} \subseteq \overline{\OA}$. The inverse inclusion is given by \eqref{grinc} and we obtain (iv).

(iv) $\Rightarrow$ (v): Just combine Corollary \ref{Nc1} with Corollary \ref{c1}.

(v) $\Rightarrow$ (vi): We argue by contradiction. Assume that there exists a $g \in \mathrm{Cut}_o^c \cap \partial \M$. Since $\widetilde{\M}_2 \subseteq \mathrm{Abn}_o^* \subseteq \mathrm{Cut}_o$, we have $g \in \overline{\widetilde{\M}} \setminus \widetilde{\M}$. Then from  \cite[(1) of Remark 2.6]{Li19} there exist $\zeta \in \R^q$ with $|\zeta| = d(g)$ and $\theta \in \partial \OA$ such that $\gamma_{(\zeta,2 \, \theta)}$ is a shortest geodesic joining $o$ to $g$. Since $g \in \mathrm{Cut}_o^c = \mathcal{S}$, it follows from \eqref{CCL1} that $\gamma_{(\zeta,2 \, \theta)}$ is strictly normal. As a result, (v) implies $g \in \mathrm{Cut}_o^{\mathrm{CL}} \subseteq \mathrm{Cut}_o$ and we obtain a contradiction.

(vi) $\Rightarrow$ (i): We argue by contradiction. Assume that $\overline{\M} \subsetneqq \G$. Since $\mathrm{Cut}_o^c$ is dense in $\G$ by \cite[Theorem 11.8]{ABB20}, we can pick a $g \in (\G \setminus \overline{\M}) \cap \mathrm{Cut}_o^c$. From the characterization of the smooth points \eqref{CCL1}, there exists a unique shortest geodesic $\gamma = \gamma_{(w,2 \, \theta)}$ joining $o$ to $g$, which is not abnormal. We first
claim that $\theta \notin \OA$, otherwise $\theta$ should be
a critical point of $\phi(g;\cdot)$ in $\OA$ by Theorem \ref{RLT} and thus $g \in \widetilde{\M} = \M \cup \widetilde{\M}_2$. Since $g \notin \M$, then $g \in \widetilde{\M}_2 \subseteq \mathrm{Cut}_o$, which gives a contradiction and proves this assertion.

We further claim that $\theta \notin \partial\OA$. Otherwise, Lemma \ref{bLn} and Theorem \ref{RLT} should imply that for any $s_* \in (0, \ 1)$, we have that $\exp(s_* \, (w,2 \, \theta)) \in \M$. Thus $g = \exp(w, 2 \, \theta) \in \overline{\M}$, contradicting with our assumption that $g \in (\G \setminus \overline{\M}) \cap \mathrm{Cut}_o^c$.

As a result, there exists a $s_0 \in (0, \ 1)$ such that $s_0 \, \theta \in \partial \OA$. Set
\[
g_0 = \exp(s_0 \, (w,2 \, \theta)) \in \mathrm{Cut}_o^c = \mathcal{S},
\]
where the ``$\in$'' is given by Lemma \ref{bLn}. Similarly, we get that $\exp(s_* \, (w,2 \, \theta)) \in \M$ for all $s_* \in (0, \ s_0)$ and $g_0 \in \overline{\M}$.

Now we are in a position to show that $g_0 \notin \M$. We argue by contradiction. Assume that $g_0 \in \M$, then there exists $(w_*,2 \, \theta_*)$ such that $\theta_* \in \OA$ and $\gamma_{(w_*,2 \, \theta_*)}$ is the unique shortest geodesic joining $o$ and $g_0$. So $\gamma_{(w_*,2 \, \theta_*)}$ coincides with the restriction of $\gamma$ in $[0, \  s_0]$, namely  $\gamma^{0,s_0} = \gamma_{(s_0 \, w,2 \, s_0 \, \theta)}$. Hence, $\gamma^{0, s_0}$ admits two different normal lifts. Consequently $\gamma^{0, s_0}$ is also abnormal by definition, which contradicts with, by Lemma \ref{nLN}, the fact that $\gamma$ is strictly normal.

After all, we have that $g_0 \in \partial\M \cap \mathrm{Cut}_o^c$, which leads to a contradiction. Therefore we finishes the proof.
\end{proof}

\subsection{Proof of Theorem \ref{nThm1}}

\begin{proof}
By the definition of $\widetilde{\M}$ (see \eqref{nOM}), the second claim in Theorem \ref{nThm1} is a direct consequence of the first one, that needs to be proven. Indeed, from Corollary \ref{Nc2} we have $\mathcal{S} = \M$. By Remark \ref{RKn2},
there exist $\{g_j = (x^{(j)}, t^{(j)}) \}_{j = 1}^{+\infty} \subseteq \mathcal{S} = \M$ with $\gamma_{(\zeta^{(j)}, 2 \, \theta^{(j)})}$ ($\{\theta^{(j)}\}_{j = 1}^{+\infty} \subseteq \OA$) the unique shortest geodesic joining $o$ to $g_j$ such that:
\[
g_j \longrightarrow g, \ (\zeta^{(j)}, 2 \, \theta^{(j)}) \longrightarrow (\zeta, 2 \, \theta) \  \mbox{ as } \ j \to +\infty, \quad \gamma_{(\zeta, 2 \, \theta)} = \gamma_g.
\]
It remains to prove that $\phi(g; \theta) = d(g)^2$ when $\theta \in \partial \OA $.

Notice that $U(\tau)^2$ is semi-positive definite for every $0 \neq \tau \in \R^m$. Let $0 \leq \lambda_1(\tau)^2 \le \ldots \le \lambda_q(\tau)^2$ ($\lambda_l(\tau) \ge 0, 1 \le l \le q$) denote its eigenvalues and $\{P_l(\tau)\}_{l = 1}^q$ the corresponding set of pairwise orthogonal projections  (that is, $ (P_k(\tau))(\R^q) \, \bot \, (P_l(\tau))(\R^q)$ for $k \ne l$).
Then we have
\begin{align}\label{spec}
U(\tau)^2 = \sum_{l = 1}^q \lambda_l(\tau)^2 P_l(\tau).
\end{align}
It follows from \cite[Chapter two]{K95} that for every $1 \le l \le q$, $\lambda_l(\tau)$ is a continuous function of $\tau \neq 0$ and homogeneous of degree $1$, namely $\lambda_l(s \, \tau) = s \, \lambda_l(\tau)$ for $s > 0$. However, $P_l(\tau)$ is not necessarily continuous, but it can be chosen to be symmetric and homogeneous of degree $0$, namely
\begin{align}\label{homoP}
P_l(r \, \tau) = P_l(\tau) \qquad \forall \, r \ne 0, \, 1 \le l \le q.
\end{align}

For the $\theta \in \partial \OA$ obtained before, there exists an
$L \in \{1, \ldots, q\}$ such that
\begin{align*}
\lambda_{L - 1}(\theta)^2 < \pi^2 \ \mbox{ when $L > 1$, \quad and} \quad \lambda_L(\theta)^2 = \ldots = \lambda_q(\theta)^2 = \pi^2.
\end{align*}
From the continuity of $\{\lambda_l(\tau)\}_{l = 1}^q$, there exist $\delta \in \left(0 , \ \frac{\pi^2}{4} \right)$ and $ r_0 \in \left( 0, \frac{|\theta|}{2} \right)$ such that for $\tau \in B(\theta, r_0) = \{\tau; \, |\tau - \theta| < r_0\}$,
we have
\begin{align*}
\lambda_{L - 1}(\tau)^2 \le \pi^2 - 4\delta \ \mbox{ when $L > 1$, \quad  and} \quad \pi^2 - \delta \le \lambda_L(\tau)^2 \le \ldots \le \lambda_q(\tau)^2 \le \pi^2 + \delta.
\end{align*}

For $\tau \in B(\theta, r_0) $, let us set
\begin{align}
V(\tau) := \left\{ \begin{array}{ll}
\frac{1}{2 \pi i} \int_{\Gamma_1} z \, (z - U(\tau)^2)^{-1} dz = \sum\limits_{l = 1}^{L - 1} \lambda_l(\tau)^2 P_l(\tau) \ &\mbox{ when $L > 1$} \\
\, 0 \ &\mbox{ when $L = 1$} \\
\end{array} \right.,
\end{align}
and the projection on (near $\pi^2$)-eigenspaces of $U(\tau)^2$
\begin{align}
Q(\tau) := \frac{1}{2 \pi i} \int_{\Gamma_2} (z - U(\tau)^2)^{-1} dz = \sum_{l = L}^q  P_l(\tau),
\end{align}
where the contours $\Gamma_1, \Gamma_2 \subseteq \C$ are defined by
\[
\Gamma_1 := \{z; \, \mathrm{dist}(z, [0, \pi^2 - 4 \delta]) = \delta\} \  \mbox{ and } \  \Gamma_2 := \{z; \, \mathrm{dist}(z, [\pi^2 - \delta, \pi^2 + \delta]) = \delta\}
\]
respectively, with the counterclockwise orientation. From the integral representation, it is easy to see that the operator functions $V(\tau)$ and $Q(\tau)$ are continuous in $B(\theta, r_0) $.
It deduces from Theorem \ref{RLT} that
\begin{align*}
\sum_{l = 1}^q \left( \frac{\lambda_l(\theta^{(j)})}{\sin{\lambda_l(\theta^{(j)})} } \right)^2 |P_l(\theta^{(j)}) \, x^{(j)}|^2 =
\left| \frac{U(\theta^{(j)})}{\sin{U(\theta^{(j)})}} \, x^{(j)} \right|^2 = |\zeta^{(j)}|^2 \to |\zeta|^2, \ \mbox{as } \  j \to +\infty.
\end{align*}

As a result, since
\begin{align*}
|Q(\theta) \, x|^2 &= \lim_{j \to + \infty}  |Q(\theta^{(j)}) \, x^{(j)}|^2 = \lim_{j \to + \infty} \sum_{l = L}^q |P_l(\theta^{(j)}) \, x^{(j)}|^2 \\
&= \lim_{j \to + \infty} \sum_{l = L}^q \left( \frac{\lambda_l(\theta^{(j)})}{\sin{\lambda_l(\theta^{(j)})} } \right)^{-2} \left( \frac{\lambda_l(\theta^{(j)})}{\sin{\lambda_l(\theta^{(j)})} } \right)^2 |P_l(\theta^{(j)}) \, x^{(j)}|^2,
\end{align*}
and $\sin{s} \sim (\pi - s)$ for $s$ near $\pi$, we get immediately
\begin{align}\label{Px}
|Q(\theta) \, x|^2 = 0.
\end{align}

Similarly, by the fact that
\begin{align*}
&\left\langle \sqrt{V(\theta^{(j)})}  \cot{\sqrt{V(\theta^{(j)})}} \, x^{(j)}, \ x^{(j)} \right\rangle \\
&= \langle U(\theta^{(j)})  \cot{U(\theta^{(j)})} \, x^{(j)}, \ x^{(j)} \rangle + |Q(\theta^{(j)}) \, x^{(j)}|^2 - \sum_{l = L}^q \left( \lambda_l(\theta^{(j)}) \cot{\lambda_l(\theta^{(j)}) } \right) |P_l(\theta^{(j)}) \, x^{(j)}|^2,
\end{align*}
we yield that
\begin{align} \label{vcotv}
\left\langle \sqrt{V(\theta)}  \cot{\sqrt{V(\theta)}} \, x, \ x \right\rangle &= \lim_{j \to + \infty} \left\langle \sqrt{V(\theta^{(j)})}  \cot{\sqrt{V(\theta^{(j)})}} \, x^{(j)}, \ x^{(j)} \right\rangle
\nonumber \\
&= \lim_{j \to + \infty} \langle U(\theta^{(j)})  \cot{U(\theta^{(j)})} \, x^{(j)}, \ x^{(j)} \rangle.
\end{align}

Hence, it follows from Theorem \ref{RLT} that
\begin{align*}
d(g)^2 &= \lim_{j \to +\infty} d(g_j)^2 = \lim_{j \to +\infty} \phi(g_j; \theta^{(j)}) \\
&= \lim_{j \to +\infty} (\langle U(\theta^{(j)})  \cot{U(\theta^{(j)})} \, x^{(j)}, \ x^{(j)} \rangle + 4 \, t^{(j)} \cdot \theta^{(j)} ) \\
& =  \left\langle \sqrt{V(\theta)}  \cot{\sqrt{V(\theta)}} \, x, \ x \right\rangle + 4 \, t \cdot \theta.
\end{align*}

Combining this with $\theta \in \partial \OA$ and the fact that the orthogonal projection of $x$ on $\pi^2$-eigenspace of $U(\theta)^2$ is zero (cf. \eqref{Px}), it follows from \cite[Remark 2.1]{Li19} that \begin{align} \label{expphibd}
d(g)^2 = \left\langle U(\theta) \, \cot{U(\theta)} \, x, \ x \right\rangle + 4 \, t \cdot \theta = \phi((x, t); \theta).
\end{align}

This ends the proof of Theorem \ref{nThm1}.
\end{proof}

\subsection{Proof of Corollary \ref{NThA1}}

\begin{proof}
(i) $\Rightarrow$ (ii): Just combine Corollary \ref{Nc1} with Corollary \ref{c1} (cf. also Remark \ref{nRKnSM}).

(ii) $\Rightarrow$ (i): It is trivial because Theorem \ref{NTh1} (v) satisfies under our assumption.

(iii) $\Rightarrow$ (i):  It follows from Theorem \ref{t2} and \cite[Proposition 15]{R13} that $\widetilde{\M}^c = \mathrm{Cut}_o^{\mathrm{CL}} \subseteq \mathrm{Cut}_o$ is a set of measure zero, which means $\widetilde{\M}$ is dense in $\G$. This is exactly (ii) of Theorem \ref{NTh1} and we obtain that $\G$ is of type GM.

(i)  + (ii) $\Rightarrow$ (iii):  From Corollary \ref{nCcc} it suffices to prove $\widetilde{\M}^c \subseteq \mathrm{Cut}_o^{\mathrm{CL}}$.
Fix $g \in \widetilde{\M}^c$. It is clear that $g \ne o$. Theorem \ref{nThm1} guarantees that there exist $\zeta \in \R^q \setminus \{0\}$ and $\theta \in \partial \OA$ such that $\gamma_* = \gamma_{(\zeta, 2\, \theta)}$ is a shortest geodesic joining $o$ to $g$. Consider the arclength parametrized geodesic $\widetilde{\gamma} := \gamma(\widehat{\zeta},\tau; \cdot)$, where $\tau := \frac{2 \, \theta}{|\zeta|}$. Here and in the sequel, we adopt the convention
\begin{equation}\label{defhat}
\widehat{u} := \begin{cases}
\frac{u}{|u|}, & \mbox{if \ } u \in \R^{\ell} \setminus \{ 0 \}, \\
0, & \mbox{if \ }  u = 0.
\end{cases}
\end{equation}

To prove that $g \in \mathrm{Cut}_o^{\mathrm{CL}}$, it remains to show that the cut time of $\widetilde{\gamma}$, $h_{\mathrm{cut}}(\widetilde{\gamma})$, equals $|\zeta|$. First, notice that $\widetilde{\gamma}|_{[0, \ |\zeta|]}  = \gamma_{(\zeta, 2\, \theta)} = \gamma_*$ is a shortest geodesic. Hence we get $h_{\mathrm{cut}}(\widetilde{\gamma}) \ge |\zeta|$.

On the other hand,
for any $\sigma \in \R^m$ such that $\gamma(\widehat{\zeta},\sigma; \cdot) = \gamma(\widehat{\zeta},\tau; \cdot)$, we have $\gamma_* = \gamma_{(\zeta, |\zeta| \, \sigma)}$. Since $g \in \widetilde{\M}^c$, it follows from Theorem \ref{RLT} that
\begin{align*}
\left\| U\left( \frac{|\zeta| \, \sigma}{2}\right)\right\| \ge \pi,
\end{align*}
or equivalently $\frac{2\pi}{\|U(\sigma)\|} \le |\zeta|$.
Then from (ii), we obtain that $h_{\mathrm{cut}}(\widetilde{\gamma}) \le |\zeta|$.

Therefore, we finish the proof of Corollary \ref{NThA1}.
\end{proof}

\subsection{Proof of Theorem \ref{p1}}

\begin{proof}
Let $g_0 = (x_0, t_0) \in \widetilde{\M}_2$. Then \eqref{Csrd} and \eqref{dtM2} imply that there exists a $\theta_0 \in \OA$ such that $d(g_0)^2 = \phi(g_0; \theta_0)$. Since $\OA$ is open, there exists a $r_0 > 0$ such that $B(\theta_0, 2 \, r_0) = \{\tau \in \R^m; \, |\tau - \theta_0| < 2 \, r_0 \} \subseteq \OA$. From \cite[Proposition 5.1 (c)]{Li19} there exists a unit vector $\nu_0$ such that
\[
t_0 \cdot \nu_0 = 0, \  \phi(g_0; \theta_0 + s \, \nu_0) = \phi(g_0; \theta_0), \quad \forall \, s \in \R \ \mbox{ with } \  \theta_0 + s \, \nu_0 \in \OA.
\]

Consequently, for any $h > 0$, using \cite[Theorem 2.1]{Li19}, we have that
\begin{align*}
d(x_0, t_0 + h \, \nu_0)^2 \ge \phi((x_0, t_0 + h \, \nu_0); \theta_0 + r_0 \, \nu_0) &= \phi(g_0; \theta_0) + 4 \, h \, r_0 + 4 \, h \, \nu_0 \cdot \theta_0, \\
d(x_0, t_0 - h \, \nu_0)^2 \ge \phi((x_0, t_0 - h \, \nu_0); \theta_0 - r_0 \, \nu_0) &= \phi(g_0; \theta_0) + 4 \, h \, r_0 - 4 \, h  \, \nu_0 \cdot \theta_0.
\end{align*}
As a result, we obtain
\begin{align*}
d(x_0, t_0 + h \, \nu_0)^2 + d(x_0, t_0 - h \, \nu_0)^2 - 2 \, d(g_0)^2 \ge  8 \, r_0 \, h, \quad \forall \, h > 0,
\end{align*}
which finishes the proof of this theorem.
\end{proof}

\subsection{Proof of Theorem \ref{NThA2}}

\begin{proof}
First, for any arclength parametrized abnormal geodesic $\gamma(s) = \exp\{ s \, (\zeta, \tau)\} = \gamma(\zeta, \tau; s)$ with
cut time $h_{\mathrm{cut}}$ and $s_* < h_{\mathrm{cut}}$, we claim that $\gamma(s_*) = \gamma(\zeta, \tau; s_*) \in \widetilde{\M}_2$. In fact, from (ii) of Corollary \ref{NThA1}, there exists a $\sigma \in \R^m$ such that $s_* < \frac{2\pi}{\|U(\sigma)\|}$ and $\gamma(\zeta, \sigma; \cdot) = \gamma(\zeta, \tau; \cdot)$. Then it follows from the first equation of Proposition \ref{NPA1} that
\begin{align*}
\gamma(s_*) = \gamma(\zeta, \tau; s_*) = \gamma(\zeta, \sigma; s_*) \in \widetilde{\M}_2.
\end{align*}
Thus, the second equation in Proposition \ref{NPA1} implies that $\mathrm{Abn}_o^* \subseteq \overline{\widetilde{\M}_2}$.

Next, from definition the set $\mathrm{SC}_o^-$ is closed. Moreover, Theorem \ref{p1} implies that $\widetilde{\M}_2 \subseteq \mathrm{SC}_o^-$. In conclusion, we have that
\begin{align*}
\mathrm{Abn}_o^* \subseteq \overline{\widetilde{\M}_2} \subseteq  \mathrm{SC}_o^-.
\end{align*}
Finally, we recall that $\mathrm{Abn}_o^*$ is closed (cf. Lemma \ref{NLA1}), so the inclusion $\mathrm{SC}_o^- \subseteq \mathrm{Abn}_o^*$ can be deduced from \cite[Theorem 1]{CL08}, which ends the proof of Theorem \ref{NThA2}.
\end{proof}

\medskip

\renewcommand{\theequation}{\thesection.\arabic{equation}}
\section{Step-two groups of Corank $2$ are GM-groups} \label{s3}
\setcounter{equation}{0}

\medskip

The purpose of this section is twofold. On one hand, we provide a sufficient condition for $\overline{\M} = \G$ by means of semi-algebraic theory. As a byproduct, we show that all $\G(q, 2, \U)$ are of type GM. On the other hand, we prove that there exists a M\'etivier group $\G(4 N, 3, \U_N)$, which is not of type GM, for any $N \in \N^*$.

Let us begin by recalling (cf. \cite[Chapter 2]{BCR98}):

\subsection{Semi-algebraic sets, mappings and dimension}

A set $A \subseteq \R^q$ is \textit{semi-algebraic} if it is the result of a finite number of unions and intersections of sets of the form $\{f = 0\}, \{g > 0\}$, where $f,g$ are polynomials on $\R^q$. If $A$ is a semi-algebraic set, then its complement, boundary and any Cartesian projection of $A$ are semi-algebraic sets.
If $A$ and $B$ are semi-algebraic sets, then so does $A \times B$.
Furthermore, any semi-algebraic set $A \subseteq \R^q$ is the disjoint union of a finite number of semi-algebraic sets $M_i$ in $\R^q$ where each $M_i$ is a smooth submanifold in $\R^q$ and diffeomorphic to $(0, \ 1)^{\mathrm{dim} \, M_i}$.

If $A \subseteq \R^q$ and $B \subseteq \R^r$ are two semi-algebraic sets. A mapping $h: A \to B$ is \textit{semi-algebraic} if its graph is a semi-algebraic set in $\R^{q + r}$. If $S \subseteq A$ is a semi-algebraic set and $h: A \to B$ is a semi-algebraic mapping, then $h(S)$ is a semi-algebraic set in $\R^r$.

Let $A \subseteq \R^q$ be a semi-algebraic set.
Its dimension, $\mathrm{dim} \, A$, can be defined in some algebraic way. The basic properties that we will use later are:
(1) If $A$ is the finite union of semi-algebraic sets $A_1, \ldots, A_p$, then $\mathrm{dim} \, A = \max\limits_{1 \le i \le p} \mathrm{dim} \, A_i$. (2) If $A$ and $B$ are semi-algebraic sets, then $\mathrm{dim}(A \times B) = \mathrm{dim} \, A + \mathrm{dim} \, B$.
(3) If $h: A \to B$ is a semi-algebraic mapping, then $\mathrm{dim} \, h(A) \le \mathrm{dim} \, A$.
(4) Moreover, if $A$ is a semi-algebraic set as well as a smooth submanifold in $\R^q$, its dimension as a semi-algebraic set coincides with its dimension as a smooth manifold. As a result, if $A \subseteq \R^q$ is a semi-algebraic set with $\mathrm{dim} \, A < q$, then it has measure 0 in $\R^q$ by the usual Morse-Sard-Federer Theorem (cf. \cite[p. 72]{KP13}).

\medskip

\subsection{A sufficient condition, from an algebraic point of view, for $\overline{\M} = \G$}

Recall that $\OA$ is defined by \eqref{oa} and $U(\theta)$ ($\theta \in \R^m$) by \eqref{Du1}. For $\theta \neq 0$, let $M(\theta)$ denote the multiplicity of the maximal eigenvalue of $U(\theta)^2$, and
\begin{align}
\mathbf{M} := \min_{\theta \neq 0} M(\theta).
\end{align}

We have the following:

\begin{theorem}\label{t3}
If $\mathbf{M} \ge m$, then $\overline{\M} = \G$.
\end{theorem}

\begin{proof}
First observe that the open set $\OA$ is a semi-algebraic set in $\R^m$ by the fact that
\begin{align*}
\OA^c = \pi_2 \left( \{(x,\tau); \, \langle U(\tau)^2 x,x \rangle \ge \pi^2 \} \cap \{(x,\tau); \, |x|^2 = 1\}\right),
\end{align*}
where $\pi_2$ denotes the projection from $\R^q \times \R^m$ to the second entry $\R^m$.
Then, $\partial \OA$
is a semi-algebraic set
and $\dim \, (\partial \OA)  ( = \dim \,( \overline{\OA} \setminus \OA) ) \le m - 1$ by \cite[Propositon 2.8.13]{BCR98}.

Set
\begin{align*}
\Sigma := \{(U(\theta)^2 - \pi^2) \, y; \ \theta \in \partial \OA, \ y \in \R^q\} \subseteq \R^q,
\end{align*}
that is, the set of points $x$ such that there exists a $\theta \in \partial \OA$ satisfying that the orthogonal projection of $x$ on $\pi^2$-eigenspace of
$U(\theta)^2$ is zero. From \cite[Proposition 2.2]{Li19}, it remains to prove that $\Sigma$ has measure $0$.

Now, consider the map defined by
\begin{align*}
\Psi: \R^m \times \R^q &\longrightarrow \R^q  \\
(\theta, y) & \longmapsto \Psi(\theta, y) := (U(\theta)^2 - \pi^2) \, y.
\end{align*}
Notice that it is a semi-algebraic mapping. Then $\Sigma = \Psi(\partial \OA \times \R^q)$ is a semi-algebraic set. It suffices to prove that $\dim \, \Sigma \le q - 1$.

For $r \in \N$ satisfying $r \le q$, set
\begin{align*}
\Pi_{q,r} := \{\L; \, \mbox{there exist $1 \le j_1 < \ldots < j_r \le q$ such that $\L = \mathrm{span}\{e_{j_1}, \ldots, e_{j_r}\}$}\},
\end{align*}
where $\{e_1, \ldots, e_q\}$ denotes the standard orthonormal basis of $\R^q$ and we adopt the convention that $\mathrm{span}\{\emptyset\} = \{0\}$. Then $\Pi_{q,r}$ is a finite set of $C_q^r$ elements. Remark that  for a $q \times q$ real matrix $S$  with $\mathrm{rank}(S) \le r$, there exists an $\L \in \Pi_{q,r}$ such that $S(\R^q) = S(\L)$. Under our assumption, we have $\mathrm{rank}(U(\theta)^2 - \pi^2) \le q - m$ for any $\theta \in
\partial \OA$. Hence, we get
\begin{align*}
\Sigma = \Psi(\partial \OA \times \R^q) = \mathop{\cup}\limits_{\L \in \Pi_{q,q-m}} \Psi(\partial \OA \times \L).
\end{align*}

It is clear that for each $\L \in \Pi_{q,q-m}$, we have
\begin{align*}
\dim(\partial \OA \times \L) = \dim(\partial \OA) + \dim \, \L
\le (m - 1) + (q - m) = q - 1,
\end{align*}
and as a result,
\begin{align*}
\dim \, \Sigma = \max\limits_{\L \in \Pi_{q, q-m}} \dim \, \Psi(\partial \OA \times \L) \le \max\limits_{\L \in \Pi_{q,q-m}} \dim(\partial \OA \times \L) \le q - 1,
\end{align*}
which ends the proof of this theorem.
\end{proof}

By the fact that $\widetilde{U}(\theta)$ is skew-symmetric for any $\theta$, we have $\mathbf{M} \ge 2$. Then

\begin{corollary}\label{c2}
If $\G = \G(q, 1, \U)$ or $\G(q, 2, \U)$, namely $\G$ is a step-two group of Corank $1$ or $2$, then it is a GM-group.
\end{corollary}

\begin{remark}\label{r1}
(1) Theorem \ref{t3} is sharp in the sense that $\G$ may not be of type GM if $\mathbf{M} < m$. The simplest example is the free Carnot group of step two and $3$ generators studied in \cite{Li19}. Notice that in such case, we have $\mathbf{M} = 2 < 3 = m$. Other interesting examples can be found in Subsection \ref{SS43} below.

(2) Remark also that $\mathbf{M} \ge m$ is in general not necessary for $\overline{\M} = \G$. See for example the star graphs studied in \cite{Li19}.

(3) We do not know whether there exists a simple algebraic characterization for GM-groups similar to that of M\'etivier groups.

\end{remark}

\subsection{Not all M\'etivier groups are of type GM} \label{SS43}

In the sequel, we will illustrate that when $m = 3$, the condition that $q$ (instead of $\mathbf{M}$) is sufficiently larger than $m$ cannot guarantee $\overline{\M} = \G$, even in the case that $\G$ is a M\'etivier group.

\begin{proposition}\label{p4}
For any $N \in \N$, there exists a M\'etivier group $\G = \G(4N + 4, 3, \U_N)$ which is not a GM-group.
\end{proposition}

\begin{proof}
Let $\H(4 n, 3) = \G(4 n, 3, \U_{\H(4 n, 3)})$ ($n \in \N^*$) denote the $(4 n + 3)$-dimensional H-type group, that is $\U_{\H(4 n, 3)}$ satisfies the following condition (cf. \eqref{Du1} for the related definition):
\[
U_{\H(4 n, 3)}(\lambda) \, U_{\H(4 n, 3)}(\lambda') + U_{\H(4 n, 3)}(\lambda') \, U_{\H(4 n, 3)}(\lambda) = 2 \, \lambda \cdot \lambda' \, \I_{4 n}, \qquad \forall \, \lambda, \lambda' \in \R^3.
\]
We remark that the $(4 n + 3)$-dimensional quaternionic Heisenberg
group provides a good example for it.

For a row vector $\tau = (\tau_1,\tau_2,\tau_3) \in \R^3$, set
\begin{align*}
\mathfrak{X}(\tau) := i \begin{pmatrix}
0 & 2^{-1}  \tau_1 & 2^{-1}  \tau_2 & 2^{-1}  \tau_3 \\
-2^{-1}  \tau_1 & 0 & -\tau_3 & \tau_2 \\
-2^{-1}  \tau_2 & \tau_3 & 0 & -\tau_1 \\
-2^{-1}  \tau_3 & -\tau_2 & \tau_1 & 0
\end{pmatrix},
\end{align*}
$U_0(\tau) := \mathfrak{X}(\tau)$ and for $N \in \N^*$,
\begin{align*}
U_N(\tau) := \begin{pmatrix}
U_{\H(4N,3)}\left(\frac{\tau}{2}\right) & \mathbb{O}_{(4N) \times 4}\\
\mathbb{O}_{4 \times (4N)} & \mathfrak{X}(\tau)
\end{pmatrix}.
\end{align*}

Observe that
\begin{align} \label{nEe1}
U_N(\tau)^2 = \left(
	\begin{array}{cc}
	\frac{1}{4} |\tau|^2 \, \I_{4 N + 1}&  \ \\
	\ & |\tau|^2 \, \I_3 - \frac{3}{4} \tau^{\T} \, \tau\\
	\end{array}
	\right),
\end{align}
whose eigenvalues are $|\tau|^2$ with the multiplicity $2$ and $\frac{|\tau|^2}{4}$ with the multiplicity $4 N + 2$. Hence, we get a M\'etivier group, saying that $\G = \G(4N + 4, 3, \U_N)$.

We only consider the case $N \in \N^*$ here and the proof is similar when $N = 0$.

From now on, we fix $N \in \N^*$, and write
\[
x  = (\yy, \xx, \xxa) \in \R^{4N} \times \R \times \R^3 = \R^{4N + 4}.
\]

In our situation, by \eqref{nEe1}, a direct calculation shows that the initial reference set, defined by \eqref{oa}, is given by
\[
\OA = \{\tau; \, |\tau| < \pi\},
\]
and the reference function (cf. \eqref{RFn}) is
\begin{align*}
\phi((x,t);\tau) &= 4 \, t \cdot \tau + \left( \frac{|\tau|}{2} \cot{\frac{|\tau|}{2}} \right)
\left(|\yy|^2 + \xx^2 +  |\xxa \cdot \widehat{\tau}|^2 \right) \\
&+ (|\tau|\cot{|\tau|}) \left| \xxa - (\xxa \cdot \widehat{\tau}) \, \widehat{\tau} \right|^2.
\end{align*}
Here we have used the convention \eqref{defhat}.

Observe that
\begin{align}\label{symM1}
\phi(((\yy,\xx,\xxa),t);\tau) &= \phi(((\yy,\xx,\xxa),-t);-\tau), \\
\label{symM2}
\phi(((\yy,\xx,O \, \xxa),O \, t);O \, \tau) &= \phi(((\yy,\xx,\xxa),t);\tau), \quad \forall \, O \in \mathrm{O}_3,
\end{align}
where $\mathrm{O}_3$ denotes the $3 \times 3$ orthogonal group.  Without loss of generality, we may assume in the sequel that
\begin{align} \label{nIc}
\xxa = |\xxa| \, e_1 = |\xxa| (1, 0, 0), \quad t = (t_1, t_2, 0) \mbox{ with $t_1, t_2 \ge 0$.}
\end{align}

Now, by recalling that (see \eqref{EFs})
\[
\psi(s) := \frac{1 - s \cot{s}}{s^2}, \qquad \mu(s) := - (s \cot{s})' = \frac{2 s - \sin{(2 s)}}{2 \sin^2{s}},
\]
we can write
\begin{align*}
\phi((x, t); \tau)& = 4 \, t \cdot \tau  + |\xxa|^2 + \left( \frac{|\tau|}{2} \cot{\frac{|\tau|}{2}} \right)(|\yy|^2 + \xx^2 ) \\
& - \psi\left(\frac{|\tau|}{2}\right)\frac{\tau_1^2}{4} |\xxa|^2
-\psi(|\tau|)(\tau_2^2 + \tau_3^2)|\xxa|^2.
\end{align*}

Suppose that  $\theta \in \OA$ is a critical point of $\phi((x,t); \cdot)$ for some $(x, t)$ satisfying \eqref{nIc}. Then we  have
\begin{align*}
4 \, t  &= \mu\left(\frac{|\theta|}{2}\right) \frac{|\yy|^2 + \xx^2 }{2|\theta|} \, \theta +
\psi^{\prime}\left(\frac{|\theta|}{2}\right) \frac{\theta_1^2}{4} \, \frac{|\xxa|^2}{2|\theta|} \, \theta
+ \psi\left(\frac{|\theta|}{2}\right) \frac{|\xxa|^2}{2} \, \theta_1  \, e_1 \\
& + \psi^{\prime}(|\theta|) \, (\theta_2^2 + \theta_3^2) \, \frac{|\xxa|^2}{|\theta|} \, \theta
+ 2 \, \psi(|\theta|) \, |\xxa|^2 \, (\theta_2 \, e_2 + \theta_3 \, e_3).
\end{align*}

We further assume that $|\yy|^2 + \xx^2 \neq 0$ and $\xxa \neq 0$. Using \cite[Lemmas 3.1 and 3.2]{Li19},  the fact that $t_3 = 0$ implies $\theta_3 = 0$. Thus
\begin{align*}
4 \begin{pmatrix} t_1 \\ t_2 \end{pmatrix} &=
\mu\left(\frac{|\theta|}{2}\right) \frac{|\yy|^2 + \xx^2 }{2|\theta|}  \begin{pmatrix} \theta_1 \\ \theta_2 \end{pmatrix}
+ \psi^{\prime}\left(\frac{|\theta|}{2}\right) \, \frac{\theta_1^2}{4} \frac{|\xxa|^2}{2|\theta|} \begin{pmatrix} \theta_1 \\ \theta_2 \end{pmatrix}
+ \psi\left(\frac{|\theta|}{2}\right) \frac{|\xxa|^2}{2} \begin{pmatrix} \theta_1 \\ 0 \end{pmatrix} \\
& + \psi^{\prime}(|\theta|) \, \theta_2^2  \, \frac{|\xxa|^2}{|\theta|}  \begin{pmatrix} \theta_1 \\ \theta_2 \end{pmatrix}
+ 2 \, \psi(|\theta|) \, |\xxa|^2 \begin{pmatrix} 0 \\ \theta_2 \end{pmatrix} \\
&:= \Upsilon((\yy,\xx,\xxa);(\theta_1,\theta_2)).
\end{align*}

Next, following the proof of \cite[Proposition 10.3]{Li19}, we can establish the following lemma. For completeness, we include its proof in ``Appendix A''.

\begin{lemma} \label{LMng}
Suppose that
$|\yy|^2 + \xx^2 \neq 0$ and $\xxa \neq 0$. Let
\begin{gather*}
B_{\R^2}(0, \pi) := \left\{(v_1, v_2) \in \R^2; \,  \sqrt{v_1^2 + v_2^2} < \pi \right\}, \\
\R^2_{r}(\yy, \xx, \xxa) := \left\{(u_1, u_2) \in \R^2; \, |u_1| < \frac{\pi}{4} \left(\frac{u_2^2}{|\xxa|^2} + |\yy|^2 + \xx^2 + |\xxa|^2 \right)\right\}.
\end{gather*}
Then $\Upsilon((\yy,\xx,\xxa);\cdot)$ is a $C^{\infty}$-diffeomorphism from
$B_{\R^2}(0, \pi)$ onto $\R^2_{r}(\yy,\xx,\xxa)$.
\end{lemma}

Combining this with \eqref{symM1} and \eqref{symM2}, we get that $\G \setminus \widetilde{\M}$ contains the subset
\begin{align*}
\left\{(x,t); \, \xxa \neq 0, |\yy|^2 + \xx^2 \neq 0, \frac{|t \cdot \xxa|}{\pi |\xxa|} > \frac{1}{|\xxa|^2} \left| t - \frac{t \cdot \xxa}{|\xxa|} \frac{\xxa}{|\xxa|} \right|^2
+ \frac{|\yy|^2 + \xx^2 + |\xxa|^2}{16} \right\},
\end{align*}
which implies immediately $\overline{\M} \subseteq \overline{\widetilde{\M}} \subsetneqq \G$.
\end{proof}

\medskip

\renewcommand{\theequation}{\thesection.\arabic{equation}}
\section{Sub-Riemannian geometry on step-two K-type groups}\label{s6}
\setcounter{equation}{0}

\medskip

Let $p_0$, $p_1 \in \{2, 3, 4, \ldots\}$ with $p_0 \ge p_1$. Consider a $p_1 \times p_0$ full-rank real matrix
\begin{align*}
\B = \begin{pmatrix}
\b_1^\mathrm{T} \\
\vdots \\
\b_{p_1}^\mathrm{T}
\end{pmatrix}
\qquad \mbox{with column vectors} \qquad \b_j \in \R^{p_0} \setminus \{ 0 \}, \  1 \leq j \leq p_1.
\end{align*}
A step-two Kolmogorov type group (or K-type group, in short) of type $\B$, $\G_{\B}^{\K}$, is defined by $\G(1 + p_0, p_1, \U_{\B}^{\K})$ with (see \cite[\S~4]{BLU07})
\begin{align*}
U_{\B}^{\K, (j)} =
\begin{pmatrix}
0 & \b_j^\mathrm{T} \\
- \, \b_j & \mathbb{O}_{p_0 \times p_0}
\end{pmatrix}, \quad 1 \le j \le p_1.
\end{align*}

Notice that we have for $\tau \in \R^{p_1}$
\begin{align} \label{Ku2}
U_{\B}^{\K}(\tau) = i
\begin{pmatrix}
0 & \tau^\mathrm{T} \B  \\
- \B^\mathrm{T} \tau & \mathbb{O}_{p_0 \times p_0}
\end{pmatrix},
\qquad
U_{\B}^{\K}(\tau)^2 =
\begin{pmatrix}
|\B^\mathrm{T} \tau|^2 & \mathbb{O}_{1 \times p_0} \\
\mathbb{O}_{p_0 \times 1} & \B^\mathrm{T} \tau \, \tau^\mathrm{T} \B
\end{pmatrix}.
\end{align}
Hence the initial reference set, defined by \eqref{oa}, is
\begin{align*}
\OA :=  \, \Omega_{\B}^{\K} = \{\tau \in \R^{p_1}; \, |\B^\mathrm{T} \tau| < \pi \}.
\end{align*}

In the rest of this section, we write
\[
x = (x_1, x_*) \in \R \times \R^{p_0} = \R^{1 + p_0}.
\]
A simple calculation shows that the reference function in this setting is given by
\begin{align*}
\phi((x,t);  \tau)  :=  \, \phi_{\B}^{\K}((x,t); \tau) = |x|^2 + 4 \, t \cdot \tau - \left[ x_1^2 \, f(|\B^\mathrm{T} \tau|) + |\tau \cdot \B x_*|^2 \, \psi(|\B^\mathrm{T} \tau|) \right].
\end{align*}

Remark that the case $\B = \I_n$ corresponds to the star graph $K_{1, n}$, on which the squared sub-Riemannian distance and the cut locus have been characterized by \cite[Theorem 10.1]{Li19}. We will use this known result to deduce the counterpart for general $\G_{\B}^{\K}$.

Notice that $\B \B^{\T}$ is positive definite, then we can define an isomorphism:
\begin{align} \label{TBn}
\T_{\B}: \, \R^{p_1} \longrightarrow \R^{p_1}, \qquad \T_{\B}(\tau) := \left( \B \B^{\T} \right)^{\frac{1}{2}} \tau.
\end{align}
Let $\T_{\B}^{-1}$ denote its inverse. Set
\begin{align} \label{TX*}
T := T(t) = \T_{\B}^{-1}(t),  \quad t \in \R^{p_1}; \qquad X_* := X_*(x_*) = \T_{\B}^{-1}\left(\B x_*\right), \quad x_* \in \R^{p_0},
\end{align}
and
\begin{align}
X_{**} := X_{**}(x_*) = x_* -  \B^{\T} \left(\T_{\B}^{-1}(X_*)\right) = x_* -  \B^{\T} \, (\B \B^{\T})^{-1} \, \B x_*, \quad x_* \in \R^{p_0}.
\end{align}

Observe that $\B^{\T} \, \T_{\B}^{-1}$ is an isometry on $\R^{p_1}$, then we have
\begin{align} \label{nIso}
\tau \cdot \B x_* = \langle \B^{\T} \, \T_{\B}^{-1} \tau, \B^{\T} \, \T_{\B}^{-1} \B x_* \rangle =
 (\B \B^{\T} \, \T_{\B}^{-1} \tau) \cdot (\T_{\B}^{-1} \B x_*) = (\T_{\B} \tau) \cdot X_*,
\end{align}
and similarly $t \cdot \tau = T \cdot (\T_{\B} \tau)$. Moreover, by the fact that
$\B^{\T} \, (\B \B^{\T})^{-1} \B$ is a projection on $\R^{p_0}$ and $|\T_{\B} (\tau)| = |\B^{\T} \tau|$ for any $\tau \in \R^{p_1}$, we find that
\begin{align*}
\phi_{\B}^{\K}(((x_1,x_*),t); \tau) = \phi_{\I_{p_1}}^{\K}(((x_1,X_*), T); \T_{\B} (\tau)) + |X_{**}|^2.
\end{align*}
Thus, the following result is a direct consequence of \cite[Theorem 10.1]{Li19}:

\begin{theorem} \label{t5}
(1) We have
\[
d(g)^2 = \sup_{\tau \in \Omega_{\B}^{\K}} \phi_{\B}^{\K}(g; \tau), \qquad \forall \, g \in \G_{\B}^{\K}.
\]
(2) The cut locus of $o$, $\mathrm{Cut}_o$, is exactly
\begin{align*}
\M^c = \left\{((0,x_*),t); \, |X_* \cdot T| \le
\frac{|X_*|^2}{\sqrt{\pi}} \sqrt{\left|T -
(\widehat{X_*} \cdot T) \, \widehat{X_*} \right|}
\right\},
\end{align*}
where we have used the convention \eqref{defhat}. And for $((0,x_*),t) \in \M^c$, it holds that
\begin{align*}
d((0,x_*),t)^2 = |x_*|^2 + 4\pi \left|T -
(\widehat{X_*} \cdot T) \, \widehat{X_*} \right|.
\end{align*}
(3) If $(x,t) = ((x_1,x_*),t) \in \M$, then there exists a unique $\theta = \theta(x,t) \in \Omega_{\B}^{\K}$ such that
\begin{align*}
t = \frac{1}{4} \nabla_\theta \left[ x_1^2 \, f(|\B^{\T} \theta|) + |\theta \cdot \B x_*|^2 \, \psi(|\B^{\T} \theta|) \right].
\end{align*}
Furthermore, we have
\begin{align*}
d(x,t)^2 &= \phi_{\B}^{\K}((x,t);\theta) = \left|\frac{U_{\B}^{\K}(\theta)}{\sin{U_{\B}^{\K}(\theta)}}x\right|^2 \\
& = \left( \frac{|\B^{\T}\theta|}{\sin{|\B^{\T} \theta|}} \right)^2 x_1^2 + |x_*|^2
 + \left[ \left( \frac{|\B^{\T} \theta|}{\sin{|\B^{\T} \theta|}} \right)^2 - 1 \right] \frac{|\theta \cdot \B x_*|^2}{|\B^{\T} \theta|^2}.
\end{align*}
\end{theorem}

\begin{remark}\label{r5}
1. Using \cite[Corollary 10.1]{Li19}
and (iii) of Corollary \ref{NThA1},
we find that
\[
\widetilde{\M}_2 = \{((0, x_*), 0); \, x_* \in \R^{p_0}\} = \mathrm{Abn}_o^*  = \mathrm{Abn}_o, \quad \mathrm{Cut}_o^{\mathrm{CL}} = \M^c \setminus \widetilde{\M}_2.
\]

2. $\G_{\B}^{\K}$ is of type GM, then other properties in Theorem \ref{NTh1}, Corollaries \ref{Nc2} and \ref{NThA1}
are also valid.
\end{remark}

\medskip

To end this section, we describe all

\subsection{Shortest geodesic(s) joining $o$ to any given $g \neq o$, as well as ``bad'' normal geodesics}

\medskip

Let us begin with the

\subsubsection{Concrete expression of $\gamma(w, 2 \, \theta; s)$ on K-type groups}

Let $x = (x_1, x_*)$, $\widetilde{x} = (\widetilde{x}_1, \widetilde{x}_*) \in \R \times \R^{p_0} = \R^{1 + p_0}$ and $\tau \in \R^{p_1}$. A simple calculation gives that
\begin{align}\label{skew}
\langle \U_{\B}^{\K} \, x, \widetilde{x} \rangle \cdot \tau = \langle \widetilde{U}_{\B}^{\K} (\tau) \, x, \widetilde{x} \rangle
= (\widetilde{x}_1 \, \B  \, x_* - x_1 \, \B  \, \widetilde{x}_* ) \cdot \tau, \quad \forall \, \tau \in \R^{p_1},
\end{align}
which implies
\begin{align}\label{skew2}
\langle \U_{\B}^{\K} \, x, \widetilde{x} \rangle = \widetilde{x}_1 \, \B  \, x_* - x_1 \, \B  \, \widetilde{x}_* := x \star \widetilde{x}.
\end{align}

Next we consider $\exp\{\widetilde{U}_\B^\K(\tau)\} \, \zeta(0)$ with $\zeta(0) \in \R^{1 + p_0}$ and $\tau \in \R^{p_1}$. It can be treated via the Spectral Theorem or directly by the fact that
\begin{align*}
\exp\{\widetilde{U}_\B^\K(\tau)\} = \exp\{- i \, U_\B^\K(\tau)\} = \cos{U_\B^\K(\tau)} + \widetilde{U}_\B^\K(\tau) \frac{\sin{U_\B^\K(\tau)}}{U_\B^\K(\tau)},
\end{align*}
which is, by \eqref{Ku2}, equal to
\begin{align}
\begin{pmatrix}
\cos{(|\B^{\T} \tau|)} & \mbox{} \\
\mbox{} & \I_{p_0} + \frac{\cos{(|\B^{\T} \tau|)} - 1}{|\B^{\T} \tau|^2} \B^{\T} \tau \tau^{\T} \B
\end{pmatrix}
+ \frac{\sin{(|\B^{\T} \tau|)}}{|\B^{\T} \tau|}
\begin{pmatrix}
0 & \tau^{\T} \B \\
- \B^{\T} \tau & \mathbb{O}_{p_0 \times p_0}
\end{pmatrix}.
\end{align}

Now, we describe $\gamma(w, 2 \, \theta; s) := (x(s), t(s))$ for given $w := (w_1, w_*) \in \R^{1 + p_0}$ and $\theta \in \R^{p_1}$. Let us introduce the column vectors
\begin{gather}
\v_1 := \begin{pmatrix}
\theta \cdot \B w_* \\[2mm]
- w_1 \, \B^\T \theta
\end{pmatrix},
\quad  \v_2 :=
\begin{pmatrix}
w_1 \\[2mm]
\frac{\theta \cdot \B w_* }{|\B^\T \theta|^2} \, \B^\T \theta
\end{pmatrix}, \quad \v_3 :=
\begin{pmatrix}
0 \\[2mm]
w_* - \frac{\theta \cdot \B w_* }{|\B^\T \theta|^2} \, \B^\T \theta
\end{pmatrix}, \label{n1v} \\
\w_1 := \v_1 \star \v_2
= - \left( w_1^2 + \frac{(\theta \cdot \B w_*)^2}{|\B^\T \theta|^2}\right) \B\B^\T \theta, \label{w1}\\
\w_2 := \v_1 \star \v_3
= - (\theta \cdot \B w_*) \, \B w_* + \frac{(\theta \cdot \B w_*)^2}{|\B^\T \theta|^2}  \, \B\B^\T \theta, \\
\w_3 := \v_2 \star
\v_3 = - w_1 \, \B \left( w_* - \frac{\theta \cdot \B w_* }{|\B^\T \theta|^2} \, \B^\T \theta \right). \label{w3}
\end{gather}
It follows from \eqref{GEn} that
\begin{align} \label{expzeta}
\zeta(s) &= \frac{ \sin(2 \, s \, |\B^\T \theta|)}{|\B^\T \theta|} \, \v_1
+ \cos(2 \, s \, |\B^\T \theta|) \, \v_2 + \v_3, \\
x(s) &= \frac{ 1 - \cos(2 \, s \, |\B^\T \theta|)}{2|\B^\T \theta|^2} \, \v_1
+ \frac{\sin(2 \, s \, |\B^\T \theta|)}{2 |\B^\T \theta| }
\, \v_2 + s \, \v_3. \label{expx}
\end{align}
The calculation of $t(s)$ is cumbersome. However, \eqref{skew2}-\eqref{w3},
\eqref{GEn} together with \eqref{expzeta} and \eqref{expx} imply that
\begin{align} \label{expt}
\dot{t}(s) &= \frac{1}{2} \, x(s) \star \zeta(s) \nonumber \\
&= \frac{ \cos(2 \, s \, |\B^\T \theta|) - 1}{4 \, |\B^\T \theta|^2} \, \w_1
+ \frac{1}{2} \left( \frac{ 1 - \cos(2 \, s \, |\B^\T \theta|)}{2 \, |\B^\T \theta|^2} - s \, \frac{\sin(2 \, s \, |\B^\T \theta|)}{ |\B^\T \theta| }\right) \w_2 \nonumber \\
&+ \frac{1}{2} \left( \frac{\sin(2 \, s \, |\B^\T \theta|)}{2 \, |\B^\T \theta| } - s \, \cos(2 \, s \, |\B^\T \theta|) \right) \, \w_3.
\end{align}

Noticing that
\begin{align}
& \int_0^s r \, \sin(2 \, r \, |\B^\T \theta|) \, dr = \frac{- 2 \, s \, |\B^\T \theta| \cos(2 \, s \, |\B^\T \theta|) + \sin(2 \, s \, |\B^\T \theta|) }{4 \, |\B^\T \theta|^2}, \label{n*n1} \\
& \int_0^s r \, \cos(2 \, r \, |\B^\T \theta|) \, dr = \frac{2 \, s \, |\B^\T \theta| \sin(2 \, s \, |\B^\T \theta|) - 1 + \cos(2 \, s \, |\B^\T \theta|) }{4 \,  |\B^\T \theta|^2}, \label{n*n2}
\end{align}
from \eqref{expt}, we get that
\begin{align} \label{exptp}
t(s)&= \frac{\sin(2 \, s \, |\B^\T \theta|) - 2 \, s \, |\B^\T \theta|}{8 \, |\B^\T \theta|^3} \, \w_1 \nonumber \\
&+ \frac{s \, |\B^\T \theta| + s \,  |\B^\T \theta| \cos(2 \, s \, |\B^\T \theta|) - \sin(2 \, s \, |\B^\T \theta|) }{4 \, |\B^\T \theta|^3} \, \w_2 \nonumber \\
& + \frac{ - s \, |\B^\T \theta| \sin(2 \, s \, |\B^\T \theta|) + 1 - \cos(2 \, s \, |\B^\T \theta|) }{4 \, |\B^\T \theta|^2} \, \w_3.
\end{align}

In particular, we can get some information about the set of the endpoints of nontrivial ``bad'' normal geodesics, as well as

\subsubsection{The nontrivial ``bad'' normal geodesics} \label{ss512}

We suppose in this subsection that
\[
w = (w_1, w_*) \neq 0, \  \mbox{ and } \  |\B^{\T} \theta| = k \pi \ \mbox{ with } \  k \in \N^*.
\]
Recall that the isomorphism $\T_{\B}$ is defined by \eqref{TBn} and $|\T_{\B}(\tau)| = |\B^{\T} \tau|$ for any $\tau \in \R^{p_1}$.

Set
\begin{align} \label{nWn}
W_* := \T_{\B}^{-1} \left(\B w_* \right),
\quad \widehat{\eta} := \frac{\T_{\B}(\theta)}{|\T_{\B}(\theta)|} = \frac{\T_{\B}(\theta)}{|\B^{\T} \theta|} = \frac{\T_{\B}(\theta)}{k \, \pi}.
\end{align}
Observe that from \eqref{nIso}, we have that
\begin{align} \label{nX}
\frac{\theta \cdot \B w_*}{|\B^{\T} \theta|} = \frac{\T_{\B}(\theta) \cdot W_*}{|\B^{\T} \theta|} = \widehat{\eta} \cdot W_* := \widetilde{w}_1.
\end{align}

Let
\[
(x, t) = ((x_1, x_*), t) := \exp(w, 2 \, \theta) = \gamma_{(w, 2 \theta)}(1), \quad (x(s), t(s)) := \gamma_{(w, 2 \theta)}(s).
\]

Taking $s = 1$ and $|\B^{\T} \theta| = k \, \pi$ in \eqref{expx}, we deduce that
\begin{align} \label{nxz}
x = \left(0,  w_* - \frac{\theta \cdot \B w_*}{|\B^{\T} \theta|^2} \,  \B^\mathrm{T} \theta \right),
\end{align}
namely $x_1 = 0$ and
\begin{align}\label{xstar}
x_* = w_* - (\widehat{\eta} \cdot W_*) \,  \frac{\B^{\T} \theta}{|\B^{\T} \theta|},
\end{align}
where we have used \eqref{nX}.  Similarly, \eqref{exptp} implies that
\begin{align} \label{ntz}
t &= -\frac{1}{4 \, |\B^\T \theta|^2} \, \w_1 + \frac{1}{2 \, |\B^\T \theta|^2} \, \w_2 \\
&= - \frac{\theta \cdot \B w_* }{2 \, |\B^{\T} \theta|^2} \, \B w_*  + \frac{|\B^{\T} \theta|^2 \, w_1^2 + 3 \, (\theta \cdot \B w_*)^2}{4 \, |\B^{\T} \theta|^4}  \, \B \B^\mathrm{T}\theta. \nonumber
\end{align}
In other words, by \eqref{nX}, we have
\begin{align} \label{n1t}
t = - \frac{\widehat{\eta} \cdot W_*}{2 \, |\B^{\T} \theta|} \, \B w_*  + \frac{w_1^2 + 3 \, (\widehat{\eta} \cdot W_*)^2}{4 \, |\B^{\T} \theta|^2}  \, \B \B^\mathrm{T}\theta.
\end{align}

Recall that $\T_{\B} := (\B \B^{\T})^{\frac{1}{2}}$, $X_* := \T_{\B}^{-1}\left(\B x_*\right)$ and $T := \T_{\B}^{-1}(t)$. Applying $\T_{\B}^{-1} \B$ to both sides of
\eqref{xstar}, it follows from \eqref{nWn} and \eqref{nX} that
\begin{align} \label{Xstar}
X_* = W_* - (\widehat{\eta} \cdot W_*) \, \widehat{\eta} = W_* - \widetilde{w}_1 \, \widehat{\eta}.
\end{align}
And similarly, applying $\T_{\B}^{-1}$ to both sides of \eqref{n1t}, we obtain
\begin{align} \label{TWstar}
T = -\frac{\widetilde{w}_1}{2 \, k \, \pi} \, W_* + \frac{w_1^2 + 3 \, \widetilde{w}_1^2}{4 \, k \, \pi} \, \widehat{\eta}.
\end{align}

We split it into cases.

\paragraph{Case 1. $t = 0$ so $T = 0$.} In such case, taking inner product with $\widehat{\eta}$ on both sides of \eqref{TWstar} and using \eqref{nX}, we yield $w_1 = \widetilde{w}_1 = 0$. So $\theta \cdot \B w_* = 0$, and $w_* = x_*$ because of \eqref{nxz}. Hence, the vectors defined by \eqref{n1v}-\eqref{w3} in this situation are $\w_1 = \w_2 = \w_3 = 0$, $\v_1 = \v_2 = 0$ and $\v_3 = x$ respectively. In conclusion, by \eqref{expx} and \eqref{exptp}, a simple calculation shows that the ``bad'' normal geodesic from $o$ to $((0, x_*), 0)$ is the the straight segment $\gamma_{((0, x_*), 0)}(s)$.

\medskip

From now on, we further assume that:

\paragraph{Case 2. $t \neq 0$ so $T \neq 0$.}
In such case, we have $w_1^2 + \widetilde{w}_1^2 > 0$. Taking inner product with $\widehat{\eta}$ on both sides of \eqref{Xstar}, by \eqref{nX}, we get $X_* \cdot \widehat{\eta} = 0$. Inserting \eqref{Xstar} into \eqref{TWstar}, we obtain that
\begin{align}\label{TXstar}
T =  -\frac{\widetilde{w}_1}{2 \, k \, \pi} \, X_*+ \frac{w_1^2 + \widetilde{w}_1^2}{4 \, k \, \pi}  \, \widehat{\eta}.
\end{align}
And we will consider the cases $X_* \neq 0$ and $X_* = 0$.

{\em (I) Assume that $X_* \neq 0$.} In such case, $X_*$ and the unit vector $\widehat{\eta} = \frac{\T_{\B} \theta}{k \, \pi}$ are orthogonal. It is easy to solve out $(w_1, \widetilde{w}_1)$ as well as $\theta$ from \eqref{TXstar}. Using \eqref{nWn}, a direct calculation shows that
\begin{gather}
\widetilde{w}_1 = - 2 \, k \, \pi \frac{T \cdot \widehat{X_*}}{|X_*|}, \quad \theta = k \, \pi \, \T_\B^{-1} \left( \frac{T - (T \cdot \widehat{X_*}) \, \widehat{X_*}}{|T - (T \cdot \widehat{X_*}) \, \widehat{X_*}|} \right),
 \nonumber \\
\mbox{} \label{nP1n} \\[-6mm]
w_1 = \pm \sqrt{4 \, k \, \pi \left| T - (T \cdot \widehat{X_*}) \, \widehat{X_*} \right| - 4 \, k^2 \, \pi^2 \left(\frac{T \cdot \widehat{X_*}}{|X_*|}\right)^2}.  \nonumber
\end{gather}

{\em (II) Assume that $X_* = 0$, that is $\B x_* = 0$.} In such case, by \eqref{nWn}, \eqref{TXstar} implies that
\begin{align} \label{nP2n}
\theta = k \, \pi \, \T_\B^{-1} \left( \widehat{T} \right), \quad
(w_1, \widetilde{w}_1) = \sqrt{4 \, k \, \pi \,  |T|} \left(\cos{\sigma},  \sin{\sigma} \right) (\sigma \in \R).
\end{align}

In conclusion, we have always
\begin{align} \label{o1n1}
X_* \cdot \widehat{\eta} = 0.
\end{align}
And $(x, t) = ((0, x_*), t)$ ($t \neq 0$) is the endpoint of some ``bad'' normal geodesic $\gamma_{((w_1, w_*), 2 \, \theta)}$ satisfying $|\B^{\T} \theta| = k \, \pi$ ($k \in \N^*$) if and only if
\begin{align}
|X_* \cdot T| \le \frac{|X_*|^2}{\sqrt{k \pi}} \sqrt{\left|T - (\widehat{X_*} \cdot T) \, \widehat{X_*} \right|}.
\end{align}

In such case, taking $(w_1, \widetilde{w}_1, \theta)$ as in \eqref{nP1n} for $X_* \neq 0$, or in \eqref{nP2n} for $X_* = 0$, it follows from \eqref{xstar} and \eqref{nX} that
\begin{align} \label{w*n}
w_* = x_* - 2 \, \frac{T \cdot \widehat{X_*}}{|X_*|}
 \, \B^{\T} \theta = x_* + \frac{\widetilde{w}_1}{k \, \pi} \B^{\T} \theta.
\end{align}
Now recall that (see \eqref{nxz}) $x_* = w_* - \frac{\theta \cdot \B w_* }{|\B^\T \theta|^2} \, \B^\T \theta$. Substituting this as well as \eqref{nX} and $|\B^{\T} \theta| = k \pi$ in \eqref{n1v}-\eqref{w3},  we get that
\begin{gather*}
\v_3 = \begin{pmatrix} 0 \\ x_* \end{pmatrix}, \quad
\v_2 = \begin{pmatrix} w_1 \\ \frac{\widetilde{w}_1}{k \, \pi} \, \B^\mathrm{T} \theta \end{pmatrix}, \quad \v_1 = k \, \pi \,
\begin{pmatrix} \widetilde{w}_1 \\ - \frac{w_1}{k \, \pi} \, \B^\mathrm{T} \theta \end{pmatrix}, \\
\w_2 = - (\theta \cdot \B w_*) \,  \B x_* = - k \, \pi \, \widetilde{w}_1 \, \B x_*, \quad \w_3 = -w_1 \, \B x_*,
\end{gather*}
and by \eqref{ntz},
\begin{align*}
\w_1 = 2 \, \w_2 - 4 \, |\B^\T \theta|^2 \, t = - 2 \, k \, \pi \, \left( \widetilde{w}_1 \, \B x_* + 2 \, k \, \pi \, t \right).
\end{align*}

Substituting them into \eqref{expx} and \eqref{exptp}, and replacing $|\B^{\T} \theta|$ by $k \, \pi$, we obtain finally the expression of $\gamma_{((w_1, w_*), 2 \, \theta)}(s)$ as follows:
\begin{equation} \label{bngE}
\begin{cases}
x(s) &= s \begin{pmatrix} 0 \\ x_* \end{pmatrix}
+\frac{\sin( 2 \, s \, k \, \pi)}{2 \, k \, \pi}
\begin{pmatrix} w_1 \\ \frac{\widetilde{w}_1}{k \, \pi} \, \B^\mathrm{T} \theta \end{pmatrix}
+ \frac{1 - \cos(2 \, s  \, k \, \pi)}{2 \, k \, \pi}
\begin{pmatrix} \widetilde{w}_1 \\ - \frac{w_1}{k \, \pi} \, \B^\mathrm{T} \theta \end{pmatrix} \\[6mm]
t(s) &= \left( s -  \frac{\sin( 2 \, s \, k \, \pi)}{2 \, k \, \pi} \right) \, t
+ \frac{1 - \cos(2 \,s  \, k \, \pi)}{4 \, k \, \pi} \, s \, \widetilde{w}_1  \, \B x_* \\[6mm]
\mbox{} &+ \frac{s \, k \, \pi \, \sin(2 \, s \, k \pi) - 1 + \cos(2 \, s \, k \, \pi)}{4 \, k^2 \, \pi^2} \, w_1 \, \B x_*.
\end{cases}
\end{equation}
Moreover, by \eqref{xstar} and \eqref{nX}, we get that
\[
|w_*|^2 = \left| x_* + \widetilde{w}_1 \, \frac{\B^{\T} \theta}{|\B^{\T} \theta|} \right|^2 = |x_*|^2 + \widetilde{w}_1^2,
\]
since \eqref{nIso} says that
\[
x_* \cdot \frac{\B^{\T} \theta}{|\B^{\T} \theta|} = X_* \cdot \frac{\T_{\B} \theta}{|\B^{\T} \theta|} = X_* \cdot \widehat{\eta} = 0,
\]
where we have used \eqref{o1n1} in the last equality. In conclusion,
\begin{align} \label{Kgl}
\ell^2(\gamma_{((w_1, w_*), 2 \, \theta)}) &= w_1^2 + |w_*|^2 = |x_*|^2 + w_1^2 + \widetilde{w}_1^2 \nonumber \\
&= |x_*|^2 + 4 \, k \, \pi \, \left| T - (T \cdot \widehat{X_*}) \, \widehat{X_*} \right|.
\end{align}

\subsubsection{Shortest geodesic(s), as well as normal geodesics from $o$ to any given $g \neq o$}

Recall that the ``good'' normal geodesics from $o$ to any given $g \neq o$ are characterized by  \cite[Theorem~2.4]{Li19} in the general setting of step-two groups. Also in our special setting of K-type groups, the ``bad'' normal geodesics from $o$ to any given $g \in \mathcal{W} \setminus \{ o \}$ are characterized in the last subsection.
Hence, we only study in the sequel the shortest geodesic(s) from $o$ to any given $g_0 \neq o$.

Consider the cases $g_0 \in \M$, $g_0 \in \mathrm{Abn}_o^* \setminus \{ o \}$ and $g_0 \in \mathrm{Cut}_o \setminus \mathrm{Abn}_o^*   = \mathrm{Cut}_o^{\mathrm{CL}}$.

1. If $g_0 \in \M$, there exists a unique shortest geodesic steering $o$ to $g_0$, and its equation is well-known by Theorem \ref{RLT}
together with \eqref{expx} and \eqref{exptp}.

2. For $g_0 \in \mathrm{Abn}_o^* \setminus \{ o \}$, the unique shortest geodesic is a straight segment and it is abnormal.

3. From now on, assume that $g_0 \in \mathrm{Cut}_o \setminus \mathrm{Abn}_o^*$. Using Theorem \ref{nThm1},
any shortest geodesic joining $o$ to $g_0$ is given by $\gamma_{(w , 2 \, \theta)}$, where
\begin{gather*}
(w, 2 \, \theta) := ((w_1, w_*), 2 \, \theta ) \in  \R^{1 + p_0} \times \R^{p_1}  \quad \mbox{with}  \  |w| = d(g_0),  \\
|\B^{\T} \theta| = \pi, \   \mbox{and}  \  \exp(w, 2 \, \theta) = g_0.
\end{gather*}
More precisely, by the results known in Subsection \ref{ss512}, we yield:

\begin{proposition}
Let $g_0 \in \mathrm{Cut}_o \setminus \mathrm{Abn}_o^* = \mathrm{Cut}_o^{\mathrm{CL}}$, namely $g_0 = ((0, x_*), t)$ with $t \neq 0$ and
\[
|X_* \cdot T| \leq
\frac{|X_*|^2}{\sqrt{\pi}} \sqrt{\left|T -
(\widehat{X_*} \cdot T) \, \widehat{X_*} \right|},
\]
where $T$ and $X_*$ are defined by \eqref{TX*}. Then any shortest geodesic from $o$ to $g_0$,
\[
 \gamma_{((w_1, w_*), 2 \, \theta)}(s) = (x(s), t(s)),
\]
can be written as in \eqref{bngE} with $k = 1$, $w_*$ defined by \eqref{w*n}, and $(w_1, \widetilde{w}_1, \theta)$ by \eqref{nP1n} for $X_* \neq 0$, or by \eqref{nP2n} for $X_* = 0$.
\end{proposition}

\begin{remark}
If we suppose further that  $X_* \ne 0$ and $|X_* \cdot T| <
\frac{|X_*|^2}{\sqrt{\pi}} \sqrt{\left|T -
(\widehat{X_*} \cdot T) \, \widehat{X_*} \right|}$, \eqref{nP1n} implies that there are exactly two distinct shortest geodesics joining $o$ to $g_0$.
\end{remark}

To end this part, we determine on K-type groups

\subsubsection{Optimal synthesis}

In the special case of K-type groups, combining (ii) of Corollary \ref{NThA1} with Proposition \ref{cAg} and the result in Case 1 of Subsection \ref{ss512}, we obtain the following:

\begin{corollary}
Let $\exp\{s \, (w, 2 \, \theta)\}$ be an arclength parametrized geodesic. Then its cut time $h_{\mathrm{cut}}$ equals $+\infty$ when $w = (0, w_*)$ with $\theta \cdot \B w_* = 0$, and $\pi/\| U(\theta) \|$ otherwise.
\end{corollary}

\medskip

\renewcommand{\theequation}{\thesection.\arabic{equation}}
\section{Sub-Riemannian geometry on step-two groups associated to quadratic CR manifolds} \label{s7}
\setcounter{equation}{0}

\medskip

Let $m, n \in \N^*$ with $n \geq m$.  Consider a full-rank $m \times n$ real matrix
\begin{align*}
\A = (\a_1, \ldots, \a_n) \quad \mbox{with column vectors} \quad \a_j := (a_{1, j}, \ldots, a_{m, j})^{\T} \in \R^m.
\end{align*}
A step-two group associated to quadratic CR manifolds of type $\A$, $\G_{\A}^{\CR}$, is defined by $\G(2 n, m, \U_{\A}^{\CR})$ with (cf. for example \cite{NRS01} for more details)
\begin{align*}
U_{\A}^{\CR, (j)} =
\begin{pmatrix}
\begin{pmatrix}
0 & a_{j,1} \\
- a_{j,1} & 0
\end{pmatrix}\\
&\ddots& \\
& & &\begin{pmatrix}
0 & a_{j, n} \\
- a_{j, n} & 0
\end{pmatrix}
\end{pmatrix}, \qquad
1 \le j \le m.
\end{align*}

For example, if $m = 1$ and $\A_1 = (1, \ldots, 1)$, $\G_{\A_1}^{\CR}$ is the Heisenberg group of real dimension $2 n + 1$, $\H^{2 n + 1}$. Moreover, for $m = n$ and $\A = \I_n$, $\G_{\I_n}^{\CR}$ is the direct product of $n$ copies of Heisenberg group $\H^3$, namely
\[
\G_{\I_n}^{\CR} = \H^3 \times \cdots \times \H^3.
\]

Observe that we have for $\tau \in \R^m$,
\begin{align}  \label{ucr}
U_{\A}^{\CR}(\tau) =  i
\begin{pmatrix}
\begin{pmatrix}
0 & \a_1 \cdot \tau \\
- \a_1 \cdot \tau & 0
\end{pmatrix}\\
&\ddots& \\
& & &\begin{pmatrix}
0 & \a_n \cdot \tau \\
- \a_n \cdot \tau & 0
\end{pmatrix}
\end{pmatrix}.
\end{align}
Then the initial reference set in this situation is given by
\begin{align} \label{rsCR}
\OA := \Omega_{\A}^{\CR} = \bigcap_{j = 1}^n \left\{\tau \in \R^m; \, |\a_j \cdot \tau | < \pi \right\}.
\end{align}
We identify $\R^2$ with $\C$ in the usual way, so $\R^{2 n}$ with $\C^n$. Write in the sequel
\[
z = (z_1, \ldots, z_n) \in \C^n.
\]
A simple calculation shows that the reference function is
\begin{align} \label{nRFe}
\phi((z,t); \tau) := \phi_{\A}^{\CR}((z,t); \tau) = \sum_{j = 1}^n |z_j|^2 \, (\a_j \cdot \tau) \cot(\a_j \cdot \tau) + 4 \, t \cdot \tau.
\end{align}

Recall that $f(s) = 1 - s \cot{s}$ and $\mu(s) = f'(s)$ (see \eqref{EFs}). The gradient and the Hessian matrix of $\phi((z, t); \cdot)$ at $\tau \in \OA$ are clearly
\begin{gather}
\nabla_\tau \phi((z,t);\tau) = - \sum_{j = 1}^n \mu(\a_j \cdot \tau) \, |z_j|^2 \, \a_j + 4 \, t, \label{HMnn0} \\
\mathrm{Hess}_{\tau} \phi((z,t);\tau) = - \A \, \Lambda(z; \tau) \, \A^{\T} = - \A \,  \sqrt{\Lambda(z; \tau)} \, \left( \A \, \sqrt{\Lambda(z; \tau)} \right)^{\T}, \label{HMnn}
\end{gather}
respectively, where
\begin{align*}
\Lambda(z; \tau) = \begin{pmatrix}
\mu'(\a_1 \cdot \tau) \, |z_1|^2 \\
& \ddots \\
& & \mu'(\a_n \cdot \tau) \, |z_n|^2
\end{pmatrix}  \ge 0, \quad \forall \, \tau \in \OA,
\end{align*}
since $\mu'(s) > 0$ for all $-\pi < s < \pi$  (see for example Lemma \ref{NL31} in Subsection \ref{s4} below).

Recall that $\M$ is defined in Subsection \ref{s22}. We can characterize the squared sub-Riemannian distance as well as the cut locus in the following theorem.

\begin{theorem} \label{t6}
(1) We have
\[
d(g)^2 = \sup_{\tau \in \Omega_{\A}^{\CR}} \phi_{\A}^{\CR}(g; \tau), \qquad \forall \, g \in \G_{\A}^{\CR}.
\]

(2) The cut locus of $o$, $\mathrm{Cut}_o$, is $\M^c$, where
\begin{align*}
\M = \left\{(z, t);  \, \mathrm{span}\{\a_j; \, |z_j| \neq 0\} = \R^m
\mbox{ and } \exists \, \theta \in \OA  \mbox{ s.t. }  t =  \frac{1}{4}\sum_{j = 1}^n \mu(\a_j \cdot \theta) \, |z_j|^2 \, \a_j \right\}.
\end{align*}

(3) If $(z,t)   \in \M$, then there exists a unique $\theta = \theta(z,t) \in \Omega_{\A}^{\CR}$ such that
\begin{align}
t = \frac{1}{4} \sum_{j = 1}^n \mu(\a_j \cdot \theta) \, |z_j|^2 \, \a_j .
\end{align}
Moreover, we have
\begin{align}
d(z, t)^2 &= \phi((z,t); \theta) = \sum_{j = 1}^n \left(\frac{\a_j \cdot \theta}{\sin(\a_j \cdot \theta)}\right)^2 |z_j|^2 \nonumber \\
&=\sum_{j = 1}^n (\a_j \cdot \theta) \cot(\a_j \cdot \theta) \, |z_j|^2 + 4 \, t \cdot \theta.
\end{align}
\end{theorem}

\begin{proof} Let $r(A)$ denote the rank of a real matrix $A$. Using the basic property $r(A A^{\T}) = r(A)$, it deduces from \eqref{HMnn} that
\[
r\!\left( \mathrm{Hess}_{\tau} \phi((z,t); \tau) \right) = r\!\left( \A \, \sqrt{\Lambda(z; \tau)} \right) = \dim{\mathrm{span}\{\a_j; \, |z_j| \neq 0\}},
\]
since $\mu^\prime(s) > 0$ for $s \in (- \pi, \pi)$ (see Lemma \ref{NL31} in Subsection \ref{s4} below).
Combining this with \eqref{HMnn0}, we get immediately the characterization of $\M$. And the third assertion of this theorem follows directly from Theorem \ref{RLT}.

Moreover,  it follows from \cite[Proposition 2.2]{Li19} that
\begin{align}\label{charM}
\M \supseteq \{(z,t); \, |z_j| \neq 0, \, \forall \, 1 \le j \le n\},
\end{align}
which is dense in $\G_{\A}^{\CR}$.   Then $\G_{\A}^{\CR}$ is of type GM.  By Theorem \ref{NTh1} and Corollary \ref{Nc2}, we deduce the first assertion,  $\mathrm{Cut}_o = \M^c = \partial \M$, as well as other sub-Riemannian geometric properties.
\end{proof}

\medskip

Now we describe

\subsection{Shortest geodesic(s) joining $o$ to any $g \neq o$} \label{nSs61}

Let us begin with the

\subsubsection{Sub-Riemannian exponential map}

In the setting of step-two groups associated to quadratic CR manifolds, the equation of the normal geodesic (cf. \eqref{GEn}) as well as the sub-Riemannian exponential map becomes very concise via their special group structure, namely \eqref{ucr}. More precisely, if $p = (p_1, \ldots, p_n) \in \C^n \cong \R^{2 n}$ and $\theta \in \R^m$, let
\begin{align}
\exp\left\{s \, (p, 2 \, \theta)\right\} := (z(p, 2 \, \theta; s), t(p, 2 \, \theta; s)) := (z(s), t(s)),
\end{align}
and
\begin{align}
\exp(p, 2 \, \theta) := (z(p,2 \, \theta), t(p,2 \, \theta)) := (z, t).
\end{align}

It follows from \eqref{GEn} that
\begin{align*}
\zeta_j(s) = e^{- 2 \, i \, s \, ( \a_j \cdot \theta)} \, p_j, \qquad z_j(s) = \frac{1 - e^{-2 \, i \, s \, (\a_j \cdot \theta)}}{2 \, i \, (\a_j \cdot \theta)}  \, p_j , \quad 1 \le j \le n,
\end{align*}
and for $1 \le k \le m$,
\begin{align*}
t_k(s) &= 2^{-1} \, \sum_{j = 1}^n a_{k, j} \, \int_0^s \left\langle  \left(
   \begin{array}{cc}                    0 & 1 \\
   -1 & 0 \\
   \end{array}
   \right) z_j(r), \, \zeta_j(r) \right\rangle \, dr \\
&= - 2^{-1} \, \sum_{j = 1}^n a_{k, j} \, \int_0^s \Re \left( i \, z_j(r) \cdot \overline{\zeta_j(r)} \right) \, dr \\
&= \sum_{j = 1}^n \frac{2 \, s \, (\a_j \cdot \theta)
 - \sin{(2 \, s \, (\a_j \cdot \theta))}}{8 \, (\a_j \cdot \theta)^2} \, |p_j|^2 \, a_{k, j}.
\end{align*}

In conclusion, the normal geodesic with the initial covector $(p, 2 \, \theta)$, $\gamma(p, 2 \, \theta; s)$, is given by
\begin{align} \label{ngE}
\left\{ \begin{array}{ll}
z_j(s) = \frac{1 - e^{-2 \, i \, s \, (\a_j \cdot \theta)}}{2 \, i \, (\a_j \cdot \theta)}  \, p_j, \quad 1 \le j \le n, \\[3mm]
t(s) = \sum\limits_{j = 1}^n \frac{2 \, s \, (\a_j \cdot \theta)
 - \sin{(2 \, s \, (\a_j \cdot \theta))}}{8 \, (\a_j \cdot \theta)^2} \, |p_j|^2 \, \a_j.
\end{array} \right.
\end{align}
In particular, we yield the expression of $\exp\{(p, 2 \theta)\}$,
\begin{align} \label{expmap3}
\left\{ \begin{array}{ll}
z_j(p,2 \, \theta) = \frac{1 - e^{-2 \, i \, (\a_j \cdot \theta)}}{2 \, i \, (\a_j \cdot \theta)}  \, p_j , \quad 1 \le j \le n, \\[3mm]
t(p,2 \, \theta) = \sum\limits_{j = 1}^n \frac{2 \, (\a_j \cdot \theta) - \sin(2 \, (\a_j \cdot \theta))}{8 \, (\a_j \cdot \theta)^2} \, |p_j|^2 \, \a_j.
\end{array} \right.
\end{align}

\subsubsection{Shortest geodesic(s) joining $o$ to any given $g \neq o$} \label{ss612}

We are in a position to determine all shortest geodesics  from $o$ to any given $o \neq g_0 := (z_0 , t_0)$ with $z_0 := (z_1^{(0)}, \ldots, z_n^{(0)})$. First, notice that in our framework, the set defined by \eqref{nOM} is
\begin{align*}
\widetilde{\M}
= \left\{ (z, t); \, \exists \, \theta \in \OA \  \mbox{s.t.} \  t = \frac{1}{4} \sum_{j = 1}^n \mu(\a_j \cdot \theta) |z_j|^2 \, \a_j  \right\}.
\end{align*}
And we consider the following two cases $g_0 \in \widetilde{\M} \setminus \{o\}$ and $g_0 \in \widetilde{\M}^c$.

{\em Case 1.} $g_0 \in \widetilde{\M} \setminus \{o\}$, namely $\phi(g_0; \cdot)$ attains its maximum at some point in $\OA$, saying $\theta_0$. It follows from Theorem \ref{RLT} that there exists a unique shortest geodesic joining $o$ to $g_0$, that is, $\exp\{s \, (p(g_0), 2 \, \theta_0)\}$ ($0 \le s \le 1$) with
\begin{align}
p(g_0) = \frac{U(\theta_0)}{\sin{U(\theta_0)}} \, e^{-\widetilde{U}(\theta_0)} \, z_0, \mbox{ i.e. } \  p_j(g_0) = \frac{\a_j \cdot \theta_0}{\sin{(\a_j \cdot \theta_0)}} \, e^{i \, (\a_j \cdot \theta_0)} \, z_j^{(0)}, \quad \forall \,  1 \le j \le n.
\end{align}
Substituting this in \eqref{ngE}, we yield its concise expression.
Moreover, it is strictly normal if $g_0 \in \M$ and also abnormal for $g_0 \in \widetilde{\M} \setminus (\M \cup \{ o\})$.

{\em Case 2.} $g_0 \in \widetilde{\M}^c$, that is $\phi(g_0; \cdot)$ only attains its supremum in $\OA$ at some $\theta \in \partial \OA$. In order to characterize all shortest geodesics from $o$ to $g_0$, we set
\begin{align*}
E(g_0) := \left\{\theta \in \partial \OA; \ \phi(g_0; \theta) = \sup_{\tau \in \OA} \phi(g_0; \tau) = d(g_0)^2  \right\}.
\end{align*}

From Theorem \ref{nThm1}, it remains to determine all $(p(g_0), 2 \, \theta(g_0))$ such that
\begin{align} \label{crIC}
\theta(g_0) \in E(g_0), \quad |p(g_0)| = d(g_0) \  \mbox{ and } \ \exp\{ (p(g_0), 2 \, \theta(g_0)) \} = g_0.
\end{align}

Up to rearrangements, we may assume that there exists an $L \in \N$, $L < n$ such that
\begin{align} \label{CRc1}
|\a_j \cdot \theta(g_0) | < \pi \  \mbox{for } \  1 \le j \le L, \quad \mbox{and } \ | \a_j \cdot \theta(g_0) | = \pi \  \mbox{if } \   L + 1 \le j \le n.
\end{align}
By \eqref{expmap3}, we obtain that
\begin{align}\label{equgeocr}
p_j(g_0) = \frac{2 \, i \,(\a_j \cdot \theta(g_0))}{1 - e^{-2 \, i \, (\a_j \cdot \theta(g_0))}} \, z_j^{(0)}, \quad   \forall \, 1 \le j \le L, \quad z_j^{(0)} = 0 \ \mbox{ for $L + 1 \le j \le n$},
\end{align}
and \eqref{crIC} holds if and only if $(p_{L+1}(g_0), \ldots, p_n(g_0))$ is a solution of the following equation:
\begin{align}\label{equgeoCR}
\sum\limits_{j = L + 1}^{n} \frac{1}{\a_j \cdot \theta(g_0)} \, |p_j(g_0)|^2 \, \a_j = 4 \, t_0 -  \sum\limits_{j = 1}^{L} \mu(\a_j \cdot \theta(g_0)) \, |z_j^{(0)}|^2 \, \a_j.
\end{align}

To prove this result, it suffices to show that $|p(g_0)|^2 = d(g_0)^2$ under our assumptions \eqref{CRc1}-\eqref{equgeoCR}, and other claims are clear.

Indeed, from \eqref{expphibd} and \eqref{equgeocr}, we have
\begin{align}\label{expphibdCR}
d(g_0)^2 = \phi(g_0; \theta(g_0)) = \sum_{j = 1}^L (\a_j \cdot \theta(g_0)) \cot(\a_j \cdot \theta(g_0)) |z_j^{(0)}|^2 + 4 \, t_0 \cdot \theta(g_0).
\end{align}
When \eqref{equgeoCR} holds, taking inner product with $\theta(g_0)$ on both sides of \eqref{equgeoCR}, we obtain that
\[
\sum_{j = L + 1}^{n} |p_j(g_0)|^2 = 4 \, t_0 \cdot \theta(g_0) -  \sum_{j = 1}^{L} (\a_j \cdot \theta(g_0)) \, \mu(\a_j \cdot \theta(g_0)) \, |z_j^{(0)}|^2.
\]
Summing with $\sum\limits_{j = 1}^{L} |p_j(g_0)|^2$ on both sides of last equality, it follows from \eqref{equgeocr} that
\[
|p(g_0)|^2 = 4 \, t_0 \cdot \theta(g_0) +  \sum_{j = 1}^{L} \left( \left( \frac{\a_j \cdot \theta(g_0)}{\sin(\a_j \cdot \theta(g_0))} \right)^2 - (\a_j \cdot \theta(g_0)) \, \mu(\a_j \cdot \theta(g_0)) \right) \, |z_j^{(0)}|^2.
\]
By \eqref{expphibdCR} and the elementary identity
\[
\left( \frac{s}{\sin{s}} \right)^2 - s \, \mu(s) = s \, \cot{s},
\]
we get $|p(g_0)|^2 = d(g_0)^2$.

In particular, if $((p_1(g_0), \ldots, p_{L+1}(g_0), \ldots p_n(g_0)), 2 \, \theta(g_0))$ satisfies the condition \eqref{crIC}, then so does $((p_1(g_0), \ldots, w_{L+1} \, p_{L+1}(g_0), \ldots, w_n \, p_n(g_0)), 2 \, \theta(g_0))$ for any complex numbers $(w_{L+1}, \ldots, w_n)$ satifying
\[
|w_{L + 1}| = \ldots = |w_n| = 1.
\]
Thus, there exist infinitely many shortest geodesics from $o$ to $g_0 \in \widetilde{\M}^c$ if
\[
4 \, t_0 -  \sum_{j = 1}^{L}  \mu(\a_j \cdot \theta(g_0)) \, |z_j^{(0)}|^2 \, \a_j \neq 0.
\]
In such case, it follows from \cite[Lemma 9]{RS17} that $g_0$ belongs to the classical cut locus of $o$, $\mathrm{Cut}_o^{\mathrm{CL}}$. This provides another explanation for such $g_0 \in \widetilde{\M}^c$ belonging to $\mathrm{Cut}_o^{\mathrm{CL}}$ besides (iii) of Corollary \ref{NThA1}. Finally we remark that the meaning for the case $L = 0$ is clear in the above discussion.

\medskip

Step-two groups associated to quadratic CR manifolds have very rich sub-Riemannian geometric properties. First, we provide an example of such groups on which \eqref{failconc} is  no longer valid for some $g_0 \in \mathrm{Abn}_o^* \setminus \widetilde{\M}_2$.

\medskip

\subsection{\eqref{nMBn} can be false at a point belonging to the shortest abnormal set on GM-groups} \label{s71}

\medskip

\begin{proposition}\label{p5}
There exist $\G_{\A}^{\CR}$, $g_0 \in \mathrm{Abn}_o^* \setminus \widetilde{\M}_2$ and $c_0 > 0$ such that
\begin{align}\label{semiconc}
d(g_0 + h g)^2 + d(g_0 - h g)^2 - 2 \, d(g_0)^2 \le 0,
\end{align}
for all $g = (z, t) \in \C^n \times \R^m$ with $|g|^2 = |z|^2 + |t|^2 = 1$ and $0 < h < c_0$.
\end{proposition}

\begin{proof}
Setting
\begin{align*}
\a_1 := \begin{pmatrix} 1 \\0 \end{pmatrix}, \quad
\a_2 := \begin{pmatrix} 1 \\1 \end{pmatrix}, \quad
\a_3 := \begin{pmatrix} 1 \\-1 \end{pmatrix}
\end{align*}
and $\A := (\a_1, \a_2, \a_3)$, we obtain a step-two group associated to quadratic CR manifolds $\G_{\A}^{\CR}$.
In our case, we have
\[
\OA = \Omega_{\A}^{\CR} = \{\tau; \, -\pi < \tau_1 \pm \tau_2 < \pi\}.
\]

Consider $g_0 := (0, e_1) \in \C^3 \times \R^2$ with $e_1 = (1, 0)$. Using the first result in Theorem \ref{t6} and \eqref{nRFe}, we have that $g_0 \not\in \widetilde{\M}$ (see \eqref{nOM} for its definition) since the unique $\theta_0 \in \overline{\OA}$ satisfying
\begin{align}\label{element2}
d(0,e_1)^2 = \sup_{\tau \in \OA} \phi((0,e_1);\tau) = \sup_{\tau \in \OA} 4 \, (e_1 \cdot \tau) = 4\pi = 4 \, (e_1 \cdot \theta_0),
\end{align}
is $\theta_0 = (\pi, 0) \in \partial \OA$.

Next, we will show that $g_0 \in \mathrm{Abn}^*_o$. Let
\[
p = (p_1, p_2, p_3) := (2 \sqrt{\pi}, 0, 0) \in \C^3.
\]

Observe that, by \eqref{expmap3},
we have
\begin{align}\label{element1}
\exp(p, 2 \, \theta_0) = (0, e_1).
\end{align}
Combining this with \eqref{element2}, we find that $\gamma_{(p, 2 \, \theta_0)}$ is a shortest geodesic from $o$ to $g_0$. It remains to prove that it is also abnormal.

Indeed, set $e_2 = (0, 1)$. From \eqref{ucr}, $p$ as well as $U_{\A}^{\CR}(\theta_0)^k  \, p$ ($k \in \N^*$) belongs to the kernel of $U_{\A}^{\CR}(e_2)$, namely
\begin{align}\label{element3}
U_{\A}^{\CR}(e_2) \, U_{\A}^{\CR}(\theta_0)^k  \, p = 0, \qquad \forall \, k \in \N.
\end{align}
Hence, Proposition \ref{cAg} implies that $\gamma_{(p, 2 \, \theta_0)}$ is also abnormal. So $(0, e_1) \in \mathrm{Abn}_o^*$.

Now, we are in a position to prove \eqref{semiconc} under our assumptions. Set in the sequel
\[
\Omega_\diamondsuit := \left\{\tau; \, \frac{3 \pi}{4} < \tau_1 \pm \tau_2 < \pi \right\} \subseteq \OA.
\]

First, we suppose that $g = (z, t)$ with $z_j \neq 0$ for all $j = 1, 2, 3$. Then \eqref{charM} implies that $g(\pm h) := g_0 \pm h \, g \in \M$ for any $h \neq 0$. Let us begin with $g(h)$. By Theorem \ref{RLT}, there exists a unique $(\zeta(h), \theta(h)) \in (\C^3
\setminus \{0\}) \times \OA$ such that $\exp\{(\zeta(h), 2 \, \theta(h))\} = g(h)$.

Next, we claim that
$\theta(h) \in \Omega_\diamondsuit$ for any $|g| = 1$ and any $0 < h \le c_0$ with $c_0 < 1/4$ small enough. Otherwise, by compactness there exists a $(\zeta^*, \theta^*) \in (\C^3 \setminus \{0\}) \times \overline{\OA \setminus \Omega_\diamondsuit}$ such that $|\zeta^*| = d(g_0)$ and $\exp(\zeta^*, 2 \, \theta^*) = g_0$, which implies $\gamma_{(\zeta^*, 2 \, \theta^*)}$ is also a shortest geodesic from $o$ to $g_0$. From Case 2 in Subsection \ref{ss612}, we have that $\gamma_{(\zeta^*, 2 \, \theta^*)} = \gamma_{(p^*,2 \, \theta_0)}$ since $E(g_0) = \{\theta_0\}$.  Furthermore, $p^* :=(p_1^*, p_2^*, p_3^*) \in \C^3$, satisfying $|p^*|^2 = 4 \pi$, is a solution of
\begin{align}\label{element4}
\sum_{j = 1}^3 |p_j^*|^2 \, \a_j = 4\pi \, e_1.
\end{align}
From Lemma \ref{BPG}, we get $p^* = \zeta^*$, and Proposition \ref{cAg} implies that
\begin{align}\label{element5}
U_\A^{\CR}(\theta_0 - \theta^*) \,  p^* = 0.
\end{align}
If $p_1^* \ne 0$, using \eqref{ucr}, we obtain that $\a_1 \cdot \theta^* = \a_1 \cdot \theta_0 = \pi$, so $\theta^* = (\pi, 0)$ since $\theta^* \in \overline{\OA}$, which contradicts with the fact $\theta^* \in \overline{\OA \setminus \Omega_\diamondsuit}$. In the opposite case $p_1^* = 0$, it follows from \eqref{element4} that $|p_2^*| = |p_3^*| = \sqrt{2\pi}$. Using \eqref{ucr}, we get that $\a_j \cdot (\theta_0 - \theta^*) = 0$ for $j = 2, 3$ since $p_2^*, p_3^* \neq 0$. Observe that $\a_2$ and $\a_3$ are linearly independent, and thus we have $\theta_0 = \theta^*$, which leads to a contradiction as well.

Consequently, for such $g(h)$, by (3) of Theorem \ref{t6}, we get that
\begin{align*}
d(g(h))^2 = \phi(g(h); \theta(h))
&= \sum_{j = 1}^3 (\a_j \cdot \theta(h)) \cot(\a_j \cdot \theta(h)) \, |z_j|^2 \, h^2 + 4 \, (e_1 + h \, t) \cdot \theta(h) \\
&\le 4 \, (e_1 + h \, t) \cdot \theta(h),
\end{align*}
where we have used in the inequality that $\theta(h) \in \Omega_\diamondsuit$.
Let $(a, b) := 4 \, (e_1 + h \, t)$. Observe that $a \ge 2 \ge |b|$ since $0 < h \le c_0 < \frac{1}{4}$ and $|t| \le 1$ from our assumption. So the function defined on $\overline{\OA}$, $\kappa(\tau_1, \tau_2) := a \, \tau_1 + b \, \tau_2$ attains its maximum at $\theta_0$. In conclusion,
\begin{align} \label{estplus}
d(g(h))^2 \le 4 \, (e_1 + h \, t) \cdot \theta_0, \quad \forall \, |g| = 1, \  0 < h < c_0 \ll 1.
\end{align}
And similarly, we have
\begin{align}\label{estminus}
d(g(-h))^2   \le 4 \, (e_1 - h \, t) \cdot \theta_0, \quad \forall \, |g| = 1, \  0 < h < c_0 \ll 1.
\end{align}
We use \eqref{element2} together with \eqref{estplus} and \eqref{estminus}, and obtain \eqref{semiconc} under the additional condition that $(z, t)$ satisfying $z_j \neq 0$ for all $j = 1,2,3$.

Finally a limiting argument finishes the proof of this proposition.
\end{proof}

\medskip

To finish this section, we provide an example of $\G_{\A}^{\CR}$ on which
$\mathrm{Abn}_o \neq \mathrm{Abn}_o^*$.

\medskip

\subsection{The shortest abnormal set is not always equal to the abnormal set} \label{s72}

\medskip

\begin{proposition}\label{p6}
There exists a step-two group associated to quadratic CR manifolds $\G_{\A}^{\CR}$ such that $\mathrm{Abn}_o^* \subsetneqq \mathrm{Abn}_o$.
\end{proposition}

\begin{proof}
Setting
\begin{align*}
\a_1 = \begin{pmatrix}
2^{-1} \\0 \end{pmatrix}, \quad
\a_2 = \begin{pmatrix} 1 \\1 \end{pmatrix}, \quad
\a_3 = \begin{pmatrix} 1 \\-1 \end{pmatrix} \  \mbox{and $\A = (\a_1, \a_2, \a_3)$},
\end{align*}
we obtain a step-two group associated to quadratic CR manifolds $\G_{\A}^{\CR}$. And we will show that the point $g_0 := (0, e_1) \in \C^3 \times \R^2$ with $e_1 = (1, 0)$ satisfies $g_0 \in \mathrm{Abn}_o \setminus \mathrm{Abn}_o^*$.

In our situation, \eqref{rsCR} implies that the initial reference set is given by
\[
\OA = \Omega_{\A}^{\CR} = \{\tau =(\tau_1, \tau_2); \, - \pi < \tau_1 \pm \tau_2 < \pi\}.
\]
We argue as in the proof of Proposition \ref{p5}, and
get that
\begin{align*}
d(g_0)^2 = \sup_{\tau \in \OA} \phi(g_0; \tau) = \sup_{\tau \in \OA} 4 \, (e_1 \cdot \tau)  = 4 \pi = 4 \, (e_1 \cdot \theta_0),
\end{align*}
where $\theta_0 = (\pi, 0) \in \partial \OA$ is the unique maximum point of $\phi(g_0; \cdot)$ in $\overline{\OA}$. So $g_0 \in \widetilde{\M}^c$.
Via a simple calculation, Case 2 in Subsection \ref{ss612} implies that any shortest geodesic joining $o$ to $g_0$ can be written as $\gamma_p(s) := \exp\{s \, (p, 2 \, \theta_0)\}$ ($0 \le s \le 1$) where $p := (0, p_2, p_3) \in \C^3$ with $|p_2| = |p_3| = \sqrt{2 \pi}$.

We claim that all $\gamma_p$ are strictly normal, so $g_0 \notin \mathrm{Abn}_o^*$. We argue by contradiction: assume that there exists $p^* = (0, p_2^*, p_3^*)$ with $|p_2^*| = |p_3^*| = \sqrt{2 \pi}$ such that $\gamma_{p^*}$ is abnormal.
It follows from Proposition \ref{cAg}
that there exists a $\sigma \in \R^2 \setminus \{0\}$ such that
\begin{align}
U_{\A}^{\CR}(\sigma) \, U_{\A}^{\CR}(\theta_0)^k \, p^* = 0, \qquad \forall \, k \in \N.
\end{align}
In particular, we yield $U_{\A}^{\CR}(\sigma) \, p^* = 0$. Using \eqref{ucr}, we obtain that $\a_j \cdot \sigma = 0$ for $j = 2, 3$ since $p_2^*$, $p_3^* \neq 0$. Notice that $\a_2$ and $\a_3$ are linearly independent. Hence we have $\sigma = 0$, which leads to a contradiction.

On the other hand, we set
\[
p_* = (2 \sqrt{2\pi}, 0, 0) \in \C^3, \quad \theta_* = (2\pi,0),
\]
and consider the normal geodesic $\gamma_*(s) := \exp(s \, (p_*,2 \, \theta_*))$  ($0 \le s \le 1$). It follows from \eqref{expmap3} that $\exp\{(p_*,2 \, \theta_*)\} = g_0$, that is,  $\gamma_*$ is a normal geodesic joining $o$ to $g_0$.

Furthermore, let $e_2 = (0, 1)$. From \eqref{ucr}, a simple computation shows that
\begin{align}
U_{\A}^{\CR}(e_2) \, U_{\A}^{\CR}(\theta_*)^k \, p_* = 0, \qquad \forall \, k \in \N.
\end{align}
Combining this with Proposition \ref{cAg},
$\gamma_*$ is also abnormal. As a result, we have $g_0 \in \mathrm{Abn}_o$, which ends the proof of this proposition.
\end{proof}

\medskip

\renewcommand{\theequation}{\thesection.\arabic{equation}}
\section{Gaveau-Brockett optimal control problem on $N_{3,2}$} \label{s5}
\setcounter{equation}{0}

\medskip

The purpose of this section is to provide a new and independent proof, based on \cite{Li19}, for the Gaveau-Brockett optimal control problem on the free Carnot group of step two and 3 generators $N_{3,2}$. More precisely, we will give a different proof for Theorem \ref{t4} below. For this purpose, we start by

\subsection{Preliminaries and known results obtained in \cite{Li19}}

Recall that $N_{3,2} = \R^3 \times \R^3$ with
\begin{align*}
U(\tau) := i
\begin{pmatrix}
0 & -\tau_3 & \tau_2 \\
\tau_3 & 0 & -\tau_1 \\
-\tau_2 & \tau_1 & 0
\end{pmatrix}, \qquad \tau = (\tau_1, \tau_2, \tau_3) \in \R^3,
\end{align*}
that is, $\langle \U x, x' \rangle = x \times x'$, where ``$\times$'' denotes the cross product on $\R^3$ and hence $\widetilde{U}(\tau) \, x = \tau \times x$.

In our situation, by considering $\tau$ as a column vector,
\[
\frac{U(\tau)}{\sin{U(\tau)}} = \frac{|\tau|}{\sin|\tau|} \I_3 -\left(\frac{|\tau|}{\sin|\tau|} - 1\right) \frac{\tau\tau^\mathrm{T} }{|\tau|^2}
\]
and the initial reference set and the reference function are given respectively by $\OA = \{\tau; \, |\tau| < \pi\}$ and
\begin{align}
\phi((x,t); \tau) =  (|\tau|\cot|\tau|) |x|^2 + \frac{1 - |\tau|\cot|\tau|}{|\tau|^2}(\tau \cdot x)^2 + 4 \, t \cdot \tau.
\end{align}
See \cite[\S~11]{Li19} for more details.

We will solve the Gaveau-Brockett optimal control problem on $N_{3, 2}$, namely to find the exact expression of $d(x, t)^2$. Using a limiting argument, the scaling property (see \eqref{scap}), and an orthogonal invariance,
namely (cf. \cite[Lemma~11.1]{Li19})
\begin{align} \label{nOIP}
d(x, t)^2 = d(O \, x, O \, t)^2, \qquad \forall \, (x, t) \in N_{3, 2}, \ \forall \, O \in \mathrm{O}_3,
\end{align}
where $\mathrm{O}_3$ denotes the $3 \times 3$ orthogonal group, it suffices to determine $d(e_1, t_1 \, e_1 + t_2 \, e_2)^2$ with $t_2 > 0$.

To begin with, we recall some notations and known results:
\begin{gather}
\Omega_+ := \{(v_1,v_2) \in \R^2; \, v_2 > 0, v_1^2 + v_2^2 < \pi^2 \}, \label{O+} \\
\R^2_> := \left\{(u_1, u_2) \in \R^2; \, u_2 > \frac{2}{\sqrt{\pi}} \sqrt{|u_1|} \ge 0\right\}, \\
\R_{<,+}^2 := \left\{ (u_1, u_2); \, u_1 > 0, 0 < u_2 < \frac{2}{\sqrt{\pi}} \sqrt{u_1}\right\}.
\end{gather}
Set in the sequel
\begin{gather}
\sqrt{2} \pi < \vartheta_1  < \frac{3}{2} \pi \mbox{ such that } \tan{\vartheta_1} = \vartheta_1, \\
\mathrm{K}_3(v_1, v_2) := 2 \psi(r) + \frac{\psi'(r)}{r} v_2^2 \ \mbox{ with } \ r := \sqrt{v_1^2 + v_2^2}, \label{nK3N}
\end{gather}
and
\begin{align}
\Omega_{-,4} &:= \left\{(v_1,v_2); \, v_2 < 0, \pi < v_1 < r = \sqrt{v_1^2 + v_2^2} < \vartheta_1, \mathrm{K}_3(v_1, v_2) < 0 \right\} \nonumber \\
\label{Omega-4}
&= \left\{(v_1, v_2); \, v_2 < 0 < v_1, \ \mathrm{K}_3(v_1, v_2) < 0, \ \pi \neq r < \vartheta_1 \right\}.
\end{align}

Indeed, to show the last equality in \eqref{Omega-4}, by \cite[Lemma~3.4]{Li19}, we have that $\pi < r < \vartheta_1$ and
\begin{align*}
0 > \mathrm{K}_3(v_1, v_2)  &= 4 \left[  \sum_{j = 1}^{+\infty} v_2^2 \, \left( (j \, \pi)^2 - r^2 \right)^{-2} +  \sum_{j = 2}^{+\infty} \left( (j \, \pi)^2 - r^2 \right)^{-1} - \frac{1}{r^2 - \pi^2} \right] \\
&>  4 \, \frac{\pi^2 - v_1^2}{(r^2 - \pi^2)^2},
\end{align*}
which implies $v_1 > \pi$ since $v_1 > 0$.

Moreover, for suitable $E \subseteq \R^2$, we define the smooth function $\Lambda$,
\begin{align} \label{32N96}
\Lambda(v_1, v_2) := v_2 \left[ \frac{\psi'(r)}{r} \, v_2 \, v + 2 \, \psi(r) \, e_2 \right], \qquad v = (v_1, v_2) \in E, \  r = |v|.
\end{align}

The following results can be found or deduced directly from \cite[\S~11]{Li19}:

\begin{theorem}[\cite{Li19}] \label{RLT1}
It holds that:
{\em\begin{compactenum}[(i)]
\item $d(e_1, t_1 \, e_1 + t_2 \, e_2) = d(e_1, |t_1| \, e_1 + t_2 \, e_2)$.
\item $\Lambda$ is a $C^{\infty}$-diffeomorphism from $\Omega_+$ onto $\R^2_>$.
\item For suitable $(\theta_1, \theta_2) \in \R^2$, set
\begin{align} \label{n78n}
\theta := (\theta_1, \theta_2, 0), \quad t_{\theta} := \frac{1}{4} (\Lambda(\theta_1, \theta_2), 0), \quad g_{\theta} := (e_1, t_{\theta}).
\end{align}
Then
\begin{align} \label{dEn}
d(g_{\theta})^2 = \phi(g_{\theta}; \theta) = \frac{\theta_1^2}{|\theta|^2} + \left( \frac{\theta_2}{\sin{|\theta|}} \right)^2 = \left| \frac{U(\theta)}{\sin{U(\theta)}} e_1 \right|^2, \quad \forall \, (\theta_1, \theta_2) \in \Omega_+.
\end{align}
\item For any $\alpha \ge 0$, $\widetilde{t}(\alpha) := 4^{-1} (\frac{\alpha^2}{\pi}, \frac{2}{\pi} \alpha) \in \partial \R^2_>$ and $d(e_1,  (\widetilde{t}(\alpha), 0))^2 = 1+ \alpha^2$.
\item $\mathrm{Abn}^*_o = \mathrm{Abn}_o = \left\{(x, 0); \  x \in \R^3 \right\} = \widetilde{\M}_2$ and
    \[
    \left\{(x, t); \ x \neq 0, \ t \neq 0, \  \left| t - \langle t, \, \frac{x}{|x|} \rangle \, \frac{x}{|x|} \right| > \frac{1}{\sqrt{\pi}} \sqrt{|x| \, |t \cdot x|}
    \right\} \subseteq \mathcal{S}.
    \]
\item $\Lambda$ is a $C^{\infty}$-diffeomorphism from $\Omega_{-, 4}$ onto $\R^2_{<, +}$.
\item We have $d(0,t)^2 = 4 \pi |t|$ for all $t \in \R^3$.
\end{compactenum}}
\end{theorem}

In conclusion, via a limiting argument, it remains to determine $d(g_{\theta})^2$ with $(\theta_1, \theta_2) \in \Omega_{-, 4}$. Indeed, we have the following:

\begin{theorem}[\cite{Li19}, Theorem~11.3] \label{t4}
\eqref{dEn} remains valid for any $(\theta_1, \theta_2) \in \Omega_{-, 4}$.
\end{theorem}

\medskip

\subsection{Properties of some functions related to $-s\cot s$}\label{s4}

Recall that the functions $f$, $\mu$ and $\psi$ are defined by \eqref{EFs}.
The following lemma can be found in \cite[Lemme~3, p.~112]{G77} or \cite[Lemma~1.33]{BGG00}:

\begin{lemma} \label{NL31}
The function $\mu$ is an odd function, and a monotonely increasing diffeomorphism between $(-\pi, \, \pi)$ and $\R$.
\end{lemma}

In the sequel, let us define
\begin{align} \label{nvp0}
\varphi_0(s) := \left( \frac{s}{\sin{s}} \right)^2 - 1, \quad s \in \R,
\end{align}
and for $s > 0$,
\begin{gather}
h(s) :=  \psi^{\prime}(s)s^3 \sin^2{s} = s^2 + s \sin{s} \cos{s} - 2 \sin^2{s},  \label{ndH}\\
\varphi_1(s) := \frac{s^2 - \sin^2{s}}{s - \sin{s} \cos{s}}
\left(=  \frac{\varphi_0(s)}{\mu(s)}
= \frac{\varphi_0(s)}{s^2 \, \psi^{\prime}(s) + 2 s \, \psi(s) } \right), \label{nvp1}
\\
\varphi_2(s) := \frac{s \, (s^2 - \sin^2{s})}{s^2 + s \sin{s} \cos{s} - 2 \sin^2{s}}
 \left( =  \frac{\varphi_0(s)}{\mu(s) - 2 s \, \psi(s)} =  \frac{\varphi_0(s)}{s^2 \, \psi^{\prime}(s)} \right), \label{nvp2}
\end{gather}
and
\begin{align} \label{nvp3}
\varphi_3(s) := \sqrt{\varphi_1(s) \,  \varphi_2(s)}.
\end{align}

For $k \in \N^*$, let $\vartheta_k$ denote the unique solution of $s = \tan{s}$ on $\left( k \pi, \ (k + \frac{1}{2}) \pi \right)$.

We will need the following lemma in order to prove Theorem \ref{t4}:

\begin{lemma} \label{l1}
We have
{\em\begin{compactenum}[(1)]
\item $h(r) > 0$ for all $r > 0$. So, $\psi'(r) > 0$ for $0 < r \not\in \{k \pi; \ k \in \N^*\}$.

\item $\varphi_1$ is strictly increasing on $(0, \ +\infty)$.

\item $\varphi_2$ is strictly increasing on $\cup_{k = 1}^{+\infty} (k \pi, \ \vartheta_k)$.

\item $\varphi_3$ is strictly increasing on $(\pi, \ +\infty)$.
\end{compactenum}}
\end{lemma}

\begin{proof}
Notice that (1) can be found in \cite[Lemma 3.1]{MM17}, and (2) as well as the strict monotonicity of $\varphi_3$ on $\cup_{k = 1}^{+\infty} (k \pi, \ \vartheta_k)$ can be found in the proof of \cite[Lemma 3.4]{MM17}. For the sake of clarity, we will provide a complete proof which is not complicated.

We begin with the proof of (1). Obviously, it suffices to prove the first claim. Indeed, when $r \ge \frac{\pi}{2}$, we have
\begin{align}\label{fprime}
h^{\prime}(r) = r \, [2 +  \cos(2r)] - \frac{3}{2} \sin(2r) \ge \frac{\pi - 3}{2} > 0.
\end{align}
In the opposite case $r \in \left(0, \ \frac{\pi}{2}\right)$, we have the elementary inequality $\sin{r} > r \cos{r}$ and it is clear that
\begin{align*}
h^{\prime \prime}(r) = 2 - 2\cos(2r) - 2r \sin(2r) = 4 \sin{r} \, (\sin{r} - r \cos{r}) > 0,
\end{align*}
which implies that $h^{\prime}(r) > \lim\limits_{r \to 0^+}h^{\prime}(r) = 0$ for $r \in \left(0, \ \frac{\pi}{2}\right)$. Combining this with
\eqref{fprime}, we get that $h(r) > \lim\limits_{r \to 0^+}h(r) = 0$ for $r > 0$, which ends the proof of the first assertion.

To prove (2), let us set
\begin{align} \label{dFG}
F(r) := \frac{r - \sin{r} \cos{r}}{r^2 - \sin^2{r}} \quad \mbox{ and } \quad
G(r) := \frac{-\sin^2{r} + r \sin{r} \cos{r}}{r \, (r^2 - \sin^2{r})}.
\end{align}

Remark that
\begin{align}\label{Fp2G}
F(r) = \frac{1}{\varphi_1(r)}, \quad F(r) + 2G(r) = \frac{1}{\varphi_2(r)}, \quad F(r) + G(r) = \frac{1}{r}.
\end{align}
A simple computation gets that
\begin{align} \label{doF}
F^{\prime}(r) = - 2\frac{(r \cos{r} - \sin{r})^2}{(r^2 - \sin^2{r})^2} < 0, \quad \forall \, r \in (0, \ +\infty) \setminus \{\vartheta_k; \  k \in \N^* \},
\end{align}
which implies the strict monotonicity of $\varphi_1$.

We return to the proof of (3). Using \eqref{Fp2G},
we have that
\begin{align*}
\left(\frac{1}{\varphi_2(r)}\right)^{\prime} &= 2(F(r) + G(r))^{\prime} - F^{\prime}(r)
=- \frac{2}{r^2} + 2 \, \frac{(r \cos{r} - \sin{r})^2}{(r^2 - \sin^2{r})^2} \\
&= 2 \, \frac{(r^2 \cos{r} - r \sin{r} - r^2 + \sin^2{r})
(r^2 \cos{r} - r \sin{r} + r^2 - \sin^2{r})}
{r^2 \, (r^2 - \sin^2{r})^2 }.
\end{align*}
Observe that for $r \in (k \pi, \ \vartheta_k)$ with $k$ odd, we have that $ r \cos{r} < \sin{r} < 0$, so
\begin{align*}
\left\{ \begin{array}{ll}
r^2 \cos{r} - r \sin{r} - r^2 + \sin^2{r}
=r \, (r \cos{r} - \sin{r}) - (r^2 - \sin^2{r}) < 0 \\[2mm]
r^2 \cos{r} - r \sin{r} + r^2 - \sin^2{r}
= r^2 \, (\cos{r} + 1) - \sin{r} \, (r + \sin{r}) > 0
\end{array} \right..
\end{align*}
Similarly, if $r \in (k \pi, \ \vartheta_k)$ with $k$ even, we have that $0 < \sin{r} < r \cos{r}$  and
\begin{align*}
\left\{ \begin{array}{ll}
r^2 \cos{r} - r \sin{r} - r^2 + \sin^2{r}
=r^2 \, (\cos{r} - 1) + \sin{r} \, (-r + \sin{r}) < 0 \\[2mm]
r^2 \cos{r} - r \sin{r} + r^2 - \sin^2{r}
= r \, (r \cos{r} - \sin{r}) + (r^2 - \sin^2{r})  > 0
\end{array} \right..
\end{align*}

Hence we have that $\left(\frac{1}{\varphi_2}\right)^{\prime} < 0$ on $\cup_{k = 1}^{+\infty} (k \pi, \ \vartheta_k)$.
Finally, a direct computation gives
\begin{align*}
\varphi_2(k\pi) =  k \pi \quad \mbox{and} \quad \varphi_2(\vartheta_k) = \vartheta_k  \quad \forall \, k \ge 1,
\end{align*}
which finishes the proof of the strict monotonicity of $\varphi_2$ on $\cup_{k = 1}^{+\infty} (k \pi, \ \vartheta_k)$.

We are in a position to prove the strict monotonicity of $\varphi_3$. By the fact that $\varphi_3 = \sqrt{\varphi_1 \, \varphi_2}$, it following from (2) and (3) that $\varphi_3$ is strictly increasing on $\cup_{k = 1}^{+\infty} (k \pi, \ \vartheta_k)$. Then it remains to prove that it is also strictly increasing on $\cup_{k = 1}^{+\infty} (\vartheta_k, \ (k+1) \pi)$.
Indeed, by using \eqref{Fp2G} again, we have that
\begin{align*}
\left(\frac{1}{\varphi_3^2(r)}\right)^{\prime} &= \left[(F(r) + 2G(r)) F(r)\right]^{\prime}
=\left[\left( \frac{2}{r} - F(r) \right) F(r) \right]^{\prime} \\
&=\left[\left( \frac{2}{r} - F(r) \right)\right]^{\prime} F(r) + (F(r) + 2G(r))F^{\prime}(r) \\
&=-\frac{2}{r^2} F(r) + 2G(r) F^{\prime}(r).
\end{align*}
From \eqref{dFG} and \eqref{doF}, the last term equals
\begin{align*}
- 2 \frac{(r - \sin{r} \cos{r}) (r^2 - \sin^2{r})^2 + 2 r \, (r \cos{r} - \sin{r})^2 \,
(-\sin^2{r} + r \sin{r} \cos{r})
}{r^2 \, (r^2 - \sin^2{r})^3}.
\end{align*}
Note that we have for $r > \vartheta_1 > 4$,
\begin{align*}
&\quad(r - \sin{r} \cos{r})(r^2 - \sin^2{r})^2 + 2 \, r \, (r \cos{r} - \sin{r})^2
(-\sin^2{r} + r \sin{r} \cos{r}) \\
&\ge \left(r - \frac{1}{2}\right)(r^2 - 1)^2 - 2 \, r \, (r + 1)^2 \left(\frac{r}{2} + 1\right) \\
&= (r + 1)^2 \left[ (r - \frac{1}{2}) \, (r - 1)^2 - r^2 - 2 r \right] \\
&\ge (r + 1)^2 \left[ 3 \, (r - 1)^2 - r^2 - 2 r \right] > 2 \, r \, (r - 4) + 3 >0,
\end{align*}
which proves our lemma.
\end{proof}

\subsection{Determination of $\mathcal{W}$ in our situation}

In order to prove Theorem \ref{t4}, we will use \cite[Corollary~2.1]{Li19} in which we have assumed that $g$ does not belong to $\mathcal{W}$ (cf. \eqref{nW}).

Let us begin with the

\subsubsection{Expression of $\gamma(w, 2 \, \theta ; s) = (x(s),t(s))$ on $N_{3,2}$}

We first recall the convention \eqref{defhat}. As in the setting of K-type groups, an elementary computation gives that
\begin{gather*}
\cos(2 \, s \, U(\theta)) \, w = \cos{(2 \, s \, |\theta|)} \, (w - (w \cdot \widehat{\theta}) \, \widehat{\theta}) + (w \cdot \widehat{\theta}) \, \widehat{\theta}, \\
\frac{\sin(2 \, s \, U(\theta))}{U(\theta)} \, w = \frac{\sin(2 \, s \, |\theta|)}{|\theta|} \, (w - (w \cdot \widehat{\theta}) \, \widehat{\theta}) +  2 s \,  (w \cdot \widehat{\theta}) \, \widehat{\theta}, \\
\widetilde{U}(\theta) \, \frac{\sin(2 \, s \, U(\theta))}{U(\theta)} \, w = \sin(2 \, s \, |\theta|) ( \widehat{\theta} \times w),
\end{gather*}
where we have used the fact that $\widetilde{U}(\tau) \, x = \tau \times x$. Consequently, using \eqref{GEn}, we obtain that
\begin{align}\nonumber
\zeta(s) &= \dot{x}(s) = \cos(2 \, s \, U(\theta)) \, w +  \widetilde{U}(\theta) \, \frac{\sin(2 \, s \, U(\theta))}{U(\theta)} \, w  \\
\nonumber
&=  \cos(2 \, s \, |\theta|)(w - (w \cdot \widehat{\theta}) \, \widehat{\theta}) + (w \cdot \widehat{\theta}) \, \widehat{\theta}
+  \sin(2 \, s \, |\theta|) ( \widehat{\theta} \times w), \\
\label{expxN}
x(s) &= \frac{\sin(2 \, s \, |\theta|)}{2 |\theta|} (w - (w \cdot \widehat{\theta}) \, \widehat{\theta}) +  s \, (w \cdot \widehat{\theta}) \, \widehat{\theta}
+ \frac{1 - \cos(2 \, s \, |\theta|)}{2 |\theta|} ( \widehat{\theta} \times w).
\end{align}

Then using \eqref{GEn} again we have
\begin{align}\nonumber
\dot{t}(s) &= \frac{1}{2} \, x(s) \times \zeta(s) = \frac{1}{2} \left( \frac{\sin(2 \, s \, |\theta|)}{2 |\theta|} - s\, \cos(2 \, s \, |\theta|) \right) \, \u_1 + \frac{1 - \cos(2 \, s \, |\theta|)}{4 |\theta|} \, \u_2 \\
\label{exptN}
&+ \frac{1}{2} \left( s \, \sin(2 \, s \, |\theta|) - \frac{1 - \cos(2 \, s \, |\theta|)}{2 |\theta|}\right) \, \u_3,
\end{align}
with
\begin{align}
\u_1 &=  (w - (w \cdot \widehat{\theta}) \, \widehat{\theta}) \times [(w \cdot \widehat{\theta}) \, \widehat{\theta}]
= -  (w \cdot \widehat{\theta}) \, ( \widehat{\theta} \times w), \\
\u_2 &=  (w - (w \cdot \widehat{\theta}) \, \widehat{\theta}) \times ( \widehat{\theta} \times w)
= (|w|^2 - (w \cdot \widehat{\theta})^2) \, \widehat{\theta}, \\
\u_3 &=  (w \cdot \widehat{\theta}) \, \widehat{\theta} \times ( \widehat{\theta} \times w)
= -  (w \cdot \widehat{\theta}) \, [w - (w \cdot \widehat{\theta}) \, \widehat{\theta}],
\end{align}
where we have used the well-known vector triple product expansion:
\begin{align}\label{triex}
a \times (b \times c) = (a \cdot c) \, b - (a \cdot b) \, c, \qquad \forall \, a,b,c \in \R^3.
\end{align}
As a result,  by using \eqref{n*n1} and \eqref{n*n2}, we can write
\begin{align}\nonumber
t(s)&= \frac{1 - \cos(2 \, s \, |\theta|) -  s \,|\theta|  \sin(2 \, s \, |\theta|)}{4 |\theta|^2} \, \u_1
+ \frac{2 \, s \, |\theta| - \sin(2 \, s \, |\theta|)}{8 |\theta|^2} \, \u_2 \\
\label{exptpN}
&+ \frac{\sin(2 \, s \, |\theta|) - s \, |\theta| - s \, |\theta| \cos(2 \, s \, |\theta|)}{4 |\theta|^2} \, \u_3.
\end{align}

\medskip

Now, we provide the

\subsubsection{Description of $\mathcal{W} \cap \left\{ (e_1, \frac{1}{4}(u_1, u_2, 0)); \  u_1, u_2 \in \R\right\}$}

In this subsection, we suppose that $|\theta| = k\pi$ with $k \in \N^*$. Substituting this with $s = 1$ in \eqref{expxN} and \eqref{exptpN}, $(x(w,2 \, \theta), t(w,2 \, \theta)) := \exp(w,2 \, \theta)$ is given by
\begin{align}\label{expmap}
\left\{ \begin{array}{ll}
x(w,2 \,\theta) = (w \cdot \widehat{\theta}) \, \widehat{\theta}  \\
t(w,2 \, \theta) = \frac{1}{4 |\theta|} \, \u_2 - \frac{1 }{ 2|\theta|} \, \u_3 = \frac{w \cdot \widehat{\theta}}{2|\theta|} \, w + \frac{|w|^2 - 3 \, (w \cdot \widehat{\theta})^2}{4|\theta|} \, \widehat{\theta} \\
\end{array} \right..
\end{align}
Then we have the following lemma, which implies that the set
$\mathcal{W}$ is negligible in the ``subspace'' as well.

\begin{lemma}\label{l2}
Let $t = \frac{1}{4}(u, 0) = \frac{1}{4}(u_1, u_2, 0)$ such that $(e_1, t) \in \mathcal{W}$.
Then we have
$16 u_1^2 = k^2 \pi^2 u_2^4$ for some $k \in \N^*$.
\end{lemma}

\begin{proof}

Let $(w, 2 \, \theta) \in \R^3 \times \R^3$ with $w := (w_1, w_2, w_3)$ such that $\exp(w, 2 \, \theta) = (e_1, t)$ and $|\theta| = k \pi$ for some $k \in \N^*$. It follows from \eqref{expmap} that
\begin{align} \label{n723n}
e_1 = (w \cdot \widehat{\theta}) \, \widehat{\theta}, \quad
t  = \frac{w \cdot \widehat{\theta}}{2 |\theta|} \, w + \frac{|w|^2 - 3 \, (w \cdot \widehat{\theta})^2}{4|\theta|} \, \widehat{\theta}.
\end{align}
The first equality implies that $\theta = \pm (k \pi, 0, 0) := (\theta_1, 0, 0)$,
$w \cdot \widehat{\theta} = \pm 1$. Furthermore, $w \cdot \widehat{\theta}$ and $\theta_1$ have the same sign.
Taking inner product on both sides of the second identity in \eqref{n723n} with $\widehat{\theta}$, we obtain
\begin{align}\label{n724n}
t \cdot \widehat{\theta} = \frac{|w|^2 - (w \cdot \widehat{\theta})^2}{4 |\theta|},
\end{align}
and $\theta_1 u_1 \geq 0$.

Multiplying both sides of \eqref{n724n} by $- \widehat{\theta}$, and summing with both sides of the second equation in \eqref{n723n} respectively,  we have
\begin{align}\label{tminus}
t - (t \cdot \widehat{\theta} ) \, \widehat{\theta} = \frac{w \cdot \widehat{\theta}}{2 |\theta|} \, (w - (w \cdot \widehat{\theta}) \, \widehat{\theta} ).
\end{align}
In particular, we get that $w_3 = 0$ and $u_2 = 2 \, \frac{w \cdot \widehat{\theta}}{|\theta|} \, w_2$.

By Pythagoras Theorem, we can write
\begin{align*}
|t|^2 & = | (t \cdot \widehat{\theta}) \, \widehat{\theta} |^2 + | t  - (t \cdot \widehat{\theta})  \, \widehat{\theta} |^2
= (t \cdot \widehat{\theta})^2 + \frac{(w \cdot \widehat{\theta})^2}{4|\theta|^2} \, (|w|^2 - (w \cdot \widehat{\theta})^2) \\
&= (t \cdot \widehat{\theta})^2 + \frac{1}{4 |\theta|^2} \, (4 \, t \cdot \theta) \\
&= \frac{1}{|\theta|^2} \left( (t \cdot \theta)^2 +  (t \cdot \theta) \right) ,
\end{align*}
where we have used \eqref{tminus} and Pythagoras Theorem in the second ``$=$'', \eqref{n724n} and $ |w \cdot \widehat{\theta}| = 1$  in the third ``$=$''. Inserting $t =  \frac{1}{4}(u_1, u_2, 0)$ and $\theta = \pm (k \pi, 0, 0)$ in the last equation gives the
desired result.
\end{proof}

\begin{remark}\label{r2}
Assume $t = \frac{1}{4}(u_1, \frac{2}{\sqrt{k \, \pi}} \sqrt{u_1}, 0)$ with $u_1  > 0$ and $k \in \N^*$. It follows from the proof above that the ``bad'' normal geodesic joining $o$ to $(e_1, t)  \in \mathcal{W}$ is
\[
\gamma_{(w, 2 \, \theta)}(s) \  (0 \le s \le 1) \mbox{ with } w = (1, \sqrt{k \pi \, u_1}, 0) \mbox{ and } \theta = (k \pi, 0, 0).
\]
More precisely, from \eqref{expxN} and \eqref{exptpN}, $\gamma_{(w,2 \, \theta)}(s) := (x(s), t(s))$ $(0 \le s \le 1)$ is given by:
\begin{align*}
x(s) & =  \begin{pmatrix} s \\0 \\0 \end{pmatrix}
+ \frac{\sin(2 \, s  \, k \pi)}{2k \pi}
\begin{pmatrix} 0\\ \sqrt{k \pi \, u_1} \\0 \end{pmatrix}
+ \frac{1 - \cos(2 \, s \, k \pi)}{2k \pi}
\begin{pmatrix} 0 \\ 0 \\ \sqrt{k \pi \, u_1} \end{pmatrix}  \\
t(s) & = \frac{1}{4} \left[ - \frac{ \sin(2 \, s \, k \pi) - s \, k \pi - s \, k \pi \cos(2 \, s \, k \pi)}{k^2 \pi^2}
\begin{pmatrix}0 \\ \sqrt{k \pi \, u_1} \\0 \end{pmatrix} \right. \\
& \left. + \frac{s \, k \pi \sin(2 \, s \,k \pi) - 1 + \cos(2 \, s \, k \pi)}{k^2 \pi^2}
\begin{pmatrix} 0 \\ 0 \\  \sqrt{k \pi \, u_1}  \end{pmatrix}
+ \frac{2 \, s \, k \pi - \sin(2 \, s \, k \pi)}{2k^2 \pi^2}
 \begin{pmatrix} k \pi \, u_1 \\0  \\0 \end{pmatrix} \right].
\end{align*}
\end{remark}

\medskip

We are in a position to provide the

\medskip

\subsection{Proof of Theorem \ref{t4}}

\medskip

Recall that $\psi(s) = \frac{1 - s \cot{s}}{s^2}$ and $\varphi_0(s)$ is defined by \eqref{nvp0}. Set in the following:
\[
\Phi(w) := \varphi_0(|w|) \frac{w_2^2}{|w|^2}, \quad w = (w_1, w_2, w_3) \in \R^3.
\]

Let  $(\theta_1, \theta_2) \in \Omega_{-, 4}$. Recall that (cf. \eqref{n78n})
\[
\theta := (\theta_1, \theta_2, 0), \quad (u_1, u_2) := \Lambda(\theta_1,\theta_2) \in \R_{<,+}^2, \quad t_{\theta} := \frac{1}{4} (u_1, u_2, 0), \quad g_{\theta} := (e_1, t_{\theta}).
\]

By Lemma \ref{l2}, via a limiting argument, we may suppose in the sequel that $g_\theta \notin \mathcal{W}$. Then  all the normal geodesics joining $o$ to $g_\theta$ are ``good'' ones, so \cite[Corollary 2.1]{Li19} gives that
\begin{align}
d(e_1, t_\theta)^2 = \inf_{\tau \in \Upsilon_{\theta}}  \left[\frac{\tau_1^2}{|\tau|^2} +
\left(\frac{\tau_2}{\sin|\tau|}\right)^2\right]
=\inf_{\tau \in \Upsilon_{\theta}} \left\{ \Phi(\tau) + 1
\right\},
\end{align}
where $\Upsilon_{\theta}$ denotes the set of $\tau = (\tau_1, \tau_2, 0)$ such that $\Lambda(\tau_1,\tau_2) = (u_1,u_2)$ (see \eqref{32N96}), namely,
\begin{align}\label{nELn}
\left\{ \begin{array}{ll}
(0 < ) \, u_1 = \frac{\psi'(|\tau|)}{|\tau|} \, \tau_1 \, \tau_2^2  \\[4mm]
(0 < ) \, u_2 = \tau_2 \, \mathrm{K}_3(\tau_1, \tau_2) = \tau_2 \, \left( \frac{\psi'(|\tau|)}{|\tau|} \, \tau_2^2 + 2 \, \psi(|\tau|) \right) \\
\end{array} \right..
\end{align}

Here are some direct observations. \\[-8mm]

\paragraph{Observations:} Under the above assumptions, for any $\tau \in \Upsilon_{\theta}$, we have \\
1. $|\tau| \notin \{k \, \pi; \, k \in \N^*\}$ since $(e_1,t_\theta) \notin \mathcal{W}$. \\
2. Moreover $\tau_2 \neq 0$ and $\tau_1 > 0$ since $u_1 > 0$ and $\psi' > 0$ from (1) of Lemma \ref{l1}.  \\
3. Furthermore $|\tau| > \pi$. Otherwise $|\tau| < \pi$, $u_2 > 0$ and the second equation in \eqref{nELn} imply that $\tau_2 > 0$. So $\tau \in  \Omega_+$. Hence it follows from (ii)  of Theorem \ref{RLT1} that $(u_1, u_2) \in \R^2_{>}$. This leads to a contradiction. \\
4. The following equalities hold, in particular for $\theta$,
\begin{align} \label{ndEn}
\Phi(\tau) = \varphi_2(|\tau|) \frac{|\tau|}{\tau_1} u_1 = \varphi_1(|\tau|) \left( \frac{\tau_1}{|\tau|} u_1 + \frac{\tau_2}{|\tau|} u_2 \right) = \varphi_3(|\tau|) \sqrt{u_1 \, \left(u_1 + \frac{\tau_2}{\tau_1} u_2 \right)}.
\end{align}
Indeed, the first ``$=$'' follows from the definition of $\varphi_2$ (see \eqref{nvp2}) and the first equation in \eqref{nELn}, the second (resp. third) one from \eqref{nELn} and \eqref{nvp1} (resp. \eqref{nvp3} and the first two equalities). \\[-2mm]

It remains to show that
\begin{align} \label{nGn}
\Phi(\theta) < \Phi(\tau), \quad \forall \, \tau \in \Upsilon_{\theta} \setminus \{ \theta \}.
\end{align}
And we split the proof into three cases.

\paragraph{Case 1: $\tau \in \Upsilon_{\theta}$ with $|\tau| < |\theta|$.}
In such case, we get $|\tau| \in (\pi, \, |\theta|)$.  Moreover, we have $\tau_2 > 0$. Otherwise $\tau_2 < 0$, and combining with
the fact that $u_2 > 0$ and the second equation in \eqref{nELn}, we have
$\mathrm{K}_3(\tau_1, \tau_2) = \frac{\psi'(|\tau|)}{|\tau|} \tau_2^2 + 2 \psi(|\tau|) < 0$. So $(\tau_1, \, \tau_2) \in \Omega_{-,4}$. Thus it follows from (vi) of Theorem \ref{RLT1} that $(\tau_1, \, \tau_2) = (\theta_1, \, \theta_2)$, which is a contradiction.

Next, remark that $s \, \psi(s) = \frac{1}{s} - \cot{s}$ ($< 0$) is strictly increasing on $(\pi, \  \vartheta_1)$, then we have $2 \, |\tau| \, \psi(|\tau|) < 2 \, |\theta| \, \psi(|\theta|)$. By \eqref{nELn}, we can write
\begin{align*}
u_2 \, \frac{|\tau|}{\tau_2} - u_1 \, \frac{|\tau|}{\tau_1} = 2 \, |\tau| \, \psi(|\tau|)
< 2 \, |\theta| \, \psi(|\theta|) = u_2 \, \frac{|\theta|}{\theta_2} - u_1 \, \frac{|\theta|}{\theta_1}.
\end{align*}
By the fact that $u_1$, $u_2$, $\tau_1$, $\tau_2$, $\theta_1 > 0$ and $\theta_2 < 0$, the last inequality implies that
\[
0 < \frac{\tau_1}{|\tau|} < \frac{\theta_1}{|\theta|} \ \mbox{ and so } \  \frac{|\tau_2|}{|\tau|} = \sqrt{1 - \left( \frac{\tau_1}{|\tau|} \right)^2} > \sqrt{1 - \left( \frac{\theta_1}{|\theta|} \right)^2} = \frac{|\theta_2|}{|\theta|} > 0.
\]

Then we have
\begin{align*}
\Phi(\theta) = \left[\left(\frac{|\theta|}{\sin|\theta|}\right)^2 - 1\right] \left(\frac{\theta_2}{|\theta|}\right)^2
< \left[\left(\frac{|\tau|}{\sin|\tau|}\right)^2 - 1\right] \left(\frac{\tau_2}{|\tau|}\right)^2 = \Phi(\tau),
\end{align*}
since the function $\left(\frac{s}{\sin s}\right)^2$ ($> 1$) is strictly decreasing on $(\pi, \ \vartheta_1)$ (cf. \cite[(1.45)]{BGG00}), which ends the proof in this case.

\paragraph{Case 2: $\tau \in \Upsilon_{\theta}$ with $|\tau| \ge |\theta|$, $\tau \neq \theta$ and $\tau_2 < 0$.}
We argue as in the beginning
of Case 1,  we have that $\mathrm{K}_3(\tau_1, \tau_2) = \frac{\psi'(|\tau|)}{|\tau|} \tau_2^2 + 2 \psi(|\tau|) < 0$ and $|\tau| \notin (\pi, \, \vartheta_1)$. Moreover, since $\psi'$ is always positive (see (1) of Lemma \ref{l1}) and $\psi$ is negative only on $\cup_{k = 1}^{+\infty} (k \pi, \ \vartheta_k)$,  then we get that $|\tau| \in \cup_{k = 2}^{+\infty} (k \pi,\ \vartheta_k)$.

We begin with the case where $\frac{|\tau_2|}{|\tau|} \ge \frac{|\theta_2|}{|\theta|}$. Then we yield that
\[
0 < \frac{\tau_1}{|\tau|} \le \frac{\theta_1}{|\theta|}, \  \mbox{so } \ \frac{|\tau|}{\tau_1} \ge \frac{|\theta|}{\theta_1} > 0.
\]
Combining this with the first equality in \eqref{ndEn}, we get
\begin{align*}
\Phi(\theta) = \varphi_2(|\theta|) \, u_1 \, \frac{|\theta|}{\theta_1} < \varphi_2(|\tau|) \, u_1 \, \frac{|\tau|}{\tau_1}  =
\Phi(\tau),
\end{align*}
where we have used, in the inequality, the fact that $u_1 > 0$ and $\varphi_{2}$ ($> \pi$) is strictly increasing on $\cup_{k = 1}^{+\infty} (k \, \pi, \ \vartheta_k)$ from Lemma \ref{l1}.

We continue with the opposite case $\frac{|\tau_2|}{|\tau|} < \frac{|\theta_2|}{|\theta|}$, which is equivalent to $\frac{\theta_2}{|\theta|} < \frac{\tau_2}{|\tau|} < 0$. Similarly, we have $\frac{\tau_1}{|\tau|} > \frac{\theta_1}{|\theta|} > 0$. Hence, via the second equality in \eqref{ndEn},
\begin{align*}
0 < \Phi(\theta) = \varphi_1(|\theta|)
 \left( u_1 \, \frac{\theta_1}{|\theta|} + u_2 \, \frac{\theta_2}{|\theta|} \right) <  \varphi_1(|\tau|)
 \left( u_1 \, \frac{\tau_1}{|\tau|} + u_2 \,  \frac{\tau_2}{|\tau|}\right) = \Phi(\tau),
\end{align*}
where we have used the fact that $u_1, u_2 > 0$ and $\varphi_{1}$ ($> 0$) is strictly increasing on $(0, \ +\infty)$ from Lemma \ref{l1}.

\paragraph{Case 3: $\tau \in \Upsilon_{\theta}$ with $|\tau| \ge |\theta|$ and $\tau_2 > 0$.}
By the third equality in \eqref{ndEn}, we get
\begin{align*}
\Phi(\theta) =
\varphi_3(|\theta|) \sqrt{u_1 \left(u_1 + u_2 \, \frac{\theta_2}{\theta_1}\right)} <  \varphi_3(|\tau|) \sqrt{u_1 \left(u_1 + u_2 \, \frac{\tau_2}{\tau_1}\right)}
 = \Phi(\tau),
\end{align*}
where we have used, in the inequality, the fact that $u_1, u_2 > 0$, $\frac{\theta_2}{\theta_1} < 0 < \frac{\tau_2}{\tau_1}$, and $\varphi_{3}$ ($> 0$) is strictly increasing on $(\pi, \ +\infty)$ from Lemma \ref{l1}.

This finishes the proof of Theorem \ref{t4}.

\medskip

\subsection{Some consequences}

\medskip

In this sub-section, we provide some applications of Theorems \ref{RLT1} and \ref{t4}. More precisely, we determine the exact formulas of $d(g)^2$ on the whole space via a limiting argument, the cut locus $\mathrm{Cut}_o$ as well as all shortest geodesics from $o$ to any given $g \neq o$. For the sake of clarity, we will first reformulate Theorem \ref{t4}.

Recall that $\vartheta_1$ is the unique solution of $\tan{s} = s$ on $(\pi, \ \frac{3}{2} \pi)$.
Also for $\pi < s < \vartheta_1$,
\begin{gather*}
f(s) = 1 - s \cot{s}, \quad \mu(s) = f'(s), \quad \psi(s) = \frac{f(s)}{s^2}, \quad
\varphi_0(s) = \left( \frac{s}{\sin{s}} \right)^2 - 1, \\
\varphi_1(s) = \frac{\varphi_0(s)}{\mu(s)}, \quad \varphi_2(s) = \frac{\varphi_0(s)}{s^2 \, \psi'(s)}, \quad \varphi_3(s) = \sqrt{\varphi_1(s) \, \varphi_2(s)}.
\end{gather*}

\begin{theorem} \label{nT4n}
Let $u_1, u_2 > 0$ such that $u_2 < \frac{2}{\sqrt{\pi}} \sqrt{u_1}$. Suppose that
\[
\widetilde{\theta} := \widetilde{\theta}(u_1, u_2) =  (\theta_1, \theta_2) \ \mbox{ with } \ \theta_2 < 0 < \theta_1 < |\widetilde{\theta}| \ (\neq \pi) < \vartheta_1
\]
is the unique solution of:
\begin{align*}
u_1 = \frac{\psi'(|\widetilde{\theta}|)}{|\widetilde{\theta}|} \, \theta_1 \, \theta_2^2  \qquad
u_2 = \theta_2 \, \left( \frac{\psi'(|\widetilde{\theta}|)}{|\widetilde{\theta}|} \, \theta_2^2 + 2 \, \psi(|\widetilde{\theta}|) \right).
\end{align*}
Then we have $\theta_1 > \pi$ and
\begin{align} \label{EDnD}
d\left( e_1, \frac{1}{4}(u_1, u_2, 0) \right)^2
&= \left( \frac{\theta_1}{|\widetilde{\theta}|} \right)^2 + \left( \frac{\theta_2}{\sin{|\widetilde{\theta}|}} \right)^2 = \varphi_1(|\widetilde{\theta}|) \left( u_1 \, \frac{\theta_1}{|\widetilde{\theta}|} + u_2 \, \frac{\theta_2}{|\widetilde{\theta}|} \right) + 1 \nonumber \\
&= \varphi_2(|\widetilde{\theta}|) \, u_1 \, \frac{|\widetilde{\theta}|}{\theta_1} + 1 = \varphi_3(|\widetilde{\theta}|) \sqrt{u_1 \, \left( u_1 + u_2 \, \frac{\theta_2}{\theta_1} \right)} + 1.
\end{align}
\end{theorem}

Combining this with Theorem \ref{RLT1}, it only remains to find the

\subsubsection{Exact expression of $d(e_1, \frac{\beta}{4} \, e_1)^2$ with $\beta > 0$}

We have the following result, which is exactly \cite[Theorem 1.4]{MM17} up to a scaling property (cf. \eqref{scap}) and an orthogonal invariance (see \eqref{nOIP}) combining with (i) of Theorem \ref{RLT1}.

\begin{corollary}\label{c3}
Let $t(\beta) = \frac{1}{4}(\beta, 0, 0)$ with $\beta > 0$. Then it holds that
\begin{align*}
d(e_1,t(\beta))^2 =  \varphi_3(r) \, \beta + 1,
\end{align*}
where $r$ is the unique solution of the following equation in $(\pi, \ \vartheta_1)$:
\begin{align}\label{DCUTP}
- 2 \, \psi(r) \, \sqrt{r^2 + 2 \, r \, \frac{\psi(r)}{\psi'(r)}} = \beta.
\end{align}
\end{corollary}

\begin{proof}
Let $0 < \epsilon < \frac{2}{\sqrt{\pi}} \sqrt{\beta}$ and $t(\beta, \epsilon) = \frac{1}{4} (\beta, \epsilon, 0)$. Suppose that $\widetilde{\theta}_{\epsilon} := (\theta_1(\epsilon), \theta_2(\epsilon))$ is the unique solution of
\begin{gather}
\theta_2(\epsilon) < 0 < \pi < \theta_1(\epsilon) < |\widetilde{\theta}_{\epsilon}| < \vartheta_1, \nonumber \\
\beta = \frac{\psi'(|\widetilde{\theta}_{\epsilon}|)}{|\widetilde{\theta}_{\epsilon}|} \, \theta_1(\epsilon) \, \theta^2_2(\epsilon), \qquad \epsilon = \theta_2(\epsilon) \left( \frac{\psi'(|\widetilde{\theta}_{\epsilon}|)}{|\widetilde{\theta}_{\epsilon}|} \, \theta^2_2(\epsilon) + 2 \, \psi(|\widetilde{\theta}_{\epsilon}|) \right). \label{ncin}
\end{gather}

By the compactness of $\overline{B_{\R^2}(0, \vartheta_1)}$, up to subsequences, we may take $\epsilon_j \longrightarrow 0^+$ as $j \longrightarrow +\infty$  such that the corresponding $\widetilde{\theta}_{\epsilon_j} \longrightarrow \widetilde{\theta}_0 := (\theta_1^{(0)}, \theta_2^{(0)})$. Obviously $\pi \le \theta_1^{(0)} \le \vartheta_1$ and $\theta_2^{(0)} \leq 0$.

We claim that $\theta_2^{(0)} \neq 0$ so $\widetilde{\theta}_0 \notin \{(\pi, 0), (\vartheta_1, 0)\}$ and $\pi < r := |\widetilde{\theta}_0| < \vartheta_1$ by \eqref{ncin}. Indeed, this is ensured by the choice of $\Omega_{-, 4}$ in \cite[\S~11]{Li19}. More precisely, we argue by contradiction: suppose that $\theta_2(\epsilon_j) \longrightarrow 0^-$. Then the first equation in \eqref{ncin} implies that
\[
\lim_{j \longrightarrow +\infty} \frac{\psi'(|\widetilde{\theta}_{\epsilon_j}|)}{|\widetilde{\theta}_{\epsilon_j}|} = +\infty, \quad \mbox{so } \ |\widetilde{\theta}_{\epsilon_j}| \longrightarrow \pi^+ \  \mbox{and } \ \theta_1(\epsilon_j)
\longrightarrow \pi^+,
\]
since $\pi < \theta_1(\epsilon_j) <  |\widetilde{\theta}_{\epsilon_j}| < \vartheta_1$ and $\psi'(s) \longrightarrow +\infty$ ($\pi < s < \vartheta_1$) only if $s \longrightarrow \pi^+$. Moreover, a direct calculation shows that (see also \cite[Lemma~3.4]{Li19})
\begin{align*}
\lim_{s \longrightarrow \pi^+} (s - \pi) \, \psi(s) = -\frac{1}{\pi},
\quad \lim_{s \longrightarrow \pi^+} (s - \pi)^2 \, \psi'(s) = \frac{1}{\pi}.
\end{align*}
Combining this with \eqref{ncin}, we get that
\begin{align*}
\lim_{j \longrightarrow +\infty} \frac{1}{\pi} \left( \frac{\theta_2(\epsilon_j)}{|\widetilde{\theta}_{\epsilon_j}| - \pi} \right)^2 = \beta, \quad 0 = - \frac{2}{\pi} \lim_{j \longrightarrow +\infty} \frac{\theta_2(\epsilon_j)}{|\widetilde{\theta}_{\epsilon_j}| - \pi} = \frac{2}{\sqrt{\pi}} \sqrt{\beta} > 0.
\end{align*}
This leads to a contradiction.

In conclusion, by the continuity of $d^2$ and the last equality in \eqref{EDnD}, we obtain that $d(e_1, t(\beta))^2 =  \varphi_3(r) \, \beta + 1$, where $\pi < r < \vartheta_1$ satisfies
\begin{align} \label{NICN}
\beta = \frac{\psi'(r)}{r} \, \theta_1^{(0)} \, (\theta^{(0)}_2)^2, \qquad  \frac{\psi'(r)}{r} \, (\theta^{(0)}_2)^2 + 2 \, \psi(r) = 0.
\end{align}
That is
\begin{align*}
\beta = - 2 \, \psi(r) \sqrt{r^2 - (\theta^{(0)}_2)^2} = - 2 \, \psi(r) \, \sqrt{r^2 + 2 \, r \, \frac{\psi(r)}{\psi'(r)}}.
\end{align*}

Note that the RHS of the last equality is exactly $\frac{4}{P(s)}$ with the function $P$ defined in \cite[(3.3)]{MM17}. Then from
\cite[Lemma 3.5]{MM17}, we know that $P$ is a strictly increasing diffeomorphism between $(\pi , \ \vartheta_1)$ and $(0, \ +\infty)$, which justifies the uniqueness of the solution $r$ in $(\pi, \  \vartheta_1)$.
\end{proof}

\subsubsection{The cut locus on $N_{3,2}$}

We can characterize the cut locus of $o$ in $N_{3,2}$ from the exact formula for $d^2$ as well.
Recall that $\mathcal{S}$ denotes the set of points $g$ such that $d^2$ is $C^{\infty}$ in a neighborhood of $g$. We have the following result:

\begin{proposition} \label{nP6n}
It holds that $\mathcal{S} \supseteq \{(x, t); \ x \ \mbox{and} \  t \ \mbox{are linearly independent}\}$ on $N_{3, 2}$.
\end{proposition}

\begin{proof}
Using the scaling property (cf. \eqref{scap}), the orthogonal invariance (see \eqref{nOIP}) as well as (i) of Theorem \ref{RLT1}, it suffices to show that $d^2$ is smooth at $(e_1, \frac{1}{4} (u_1, u_2, 0))$ with
$u_2 > 0$ and $u_1 \ge 0$.
By recalling the notations defined by \eqref{O+}-\eqref{32N96}, we divide it into cases.

{\em Case (1):} $(u_1, u_2) \in \R^2_{<, +}$. Consider the smooth map
\begin{align*}
\Pi: \mathrm{O}_3 \times (0, \ +\infty) \times \Omega_{-, 4} &\longrightarrow \R^3 \times \R^3 = N_{3,2} \\
(O, r, (\theta_1, \theta_2)) &\longmapsto (r \, O \, e_1, \frac{r^2}{4} \, O \, (\Lambda(\theta_1,\theta_2), 0)),
\end{align*}
where $\mathrm{O}_3$ denotes the $3 \times 3$ orthogonal group. A direct computation shows that the differential of $\Pi$ at the point $(\I_3, 1, \Lambda^{-1}(u_1,u_2))$ is invertible since $\Lambda$ is a $C^\infty$-diffeomorphism from $\Omega_{-, 4}$ onto $\R^2_{<, +}$ by (vi) of Theorem \ref{RLT1}.
As a result, from the inverse function theorem, Theorem \ref{nT4n} and the fact that the function
\[
r^2 \, \frac{\theta_1^2}{\theta_1^2 + \theta_2^2} + r^2 \, \left( \frac{\theta_2}{\sin{\sqrt{\theta_1^2 + \theta_2^2}}} \right)^2
\]
is smooth on $(0, \ +\infty) \times \Omega_{-, 4}$, we have $(e_1, \frac{1}{4}(u_1, u_2, 0)) \in \mathcal{S}$.

{\em Case (2):} $(u_1, u_2) \in \R^2_{>}$. Similarly, we have $(e_1, \frac{1}{4}(u_1, u_2, 0)) \in \mathcal{S}$.

{\em Case (3):} $(u_1, \frac{2}{\sqrt{\pi}} \sqrt{u_1})$ with $u_1 > 0$.
In such case, it can be proven by using \cite[Theorem~26]{BR19} with the function defined in \cite[\S~11.2, Step~2]{Li19},  via Lemma \ref{nLN} and the concrete characterization of $\mathrm{Abn}^*_o$ obtained in (v) of Theorem \ref{RLT1}. However, we will provide here a direct proof, which is of independent interest.

Set
\[
\Omega_{+, 1} := \Omega_+ \cap \{(v_1, v_2); \, v_1 > 0\}, \quad \R_{>, +}^2 := \R_>^2 \cap \{(u_1, u_2); \, u_1 > 0\}.
\]
Using the polar coordinate in $\R^2 \setminus \{ (v_1, 0); \ v_1 \le 0 \}$, $(r, \eta)$ where $v_1 = r \cos{\eta}$ and $v_2 = r \sin{\eta}$ with $-\pi < \eta < \pi$, we introduce another map $\Theta: (r, \eta) \mapsto (r, \rho := \frac{\sin{\eta}}{\sin{r}})$ with suitable domain.

Notice that we have $r > 0$ and $\eta \in \left(0,  \, \frac{\pi}{2}\right)$ (resp. $\left(-\frac{\pi}{2}, \  0\right)$) on $\Omega_{+, 1}$ (resp. $\Omega_{-,4}$). Moreover, it is not hard to show that $\Theta$ is injective on $\Omega_{+, 1}$ (resp. $\Omega_{-,4}$) and its Jacobian determinant is $\frac{\cos{\eta}}{\sin{r}}$. It follows from the global inverse function Theorem that $\Theta$ is a $C^\infty$-diffeomorphism from $\Omega_{+, 1}$ (resp. $\Omega_{-, 4}$) onto $\widetilde{\Omega}_+ := \Theta(\Omega_{+,1})$ (resp. $\widetilde{\Omega}_- := \Theta(\Omega_{-,4})$). From \eqref{O+}, \eqref{Omega-4}, \eqref{nK3N} and \eqref{ndH}, a direct calculation yields that
\begin{gather*}
\widetilde{\Omega}_+ = \{(r,\rho); \, 0 < \rho \, \sin{r} < 1, \ 0 < r < \pi\}, \\ \widetilde{\Omega}_- = \{(r,\rho); \, \pi < r < \vartheta_1, \ \rho > 0,  \ \rho^2  \, h(r) + 2 \, r^2 \, \psi(r)< 0  \}.
\end{gather*}
By (ii) and (vi) of Theorem \ref{RLT1}, $\Xi := \Lambda \circ \Theta^{-1}: (r, \rho) \mapsto (u_1, u_2)$ is a $C^\infty$-diffeomorphism from $\widetilde{\Omega}_+$ (resp. $\widetilde{\Omega}_-$) onto $\R_{>, +}^2$ (resp. $\R_{<, +}^2$). Using \eqref{32N96}, \eqref{ndH} and the definition of $\psi$ (cf. \eqref{EFs}), it can be written explicitly:
\begin{align}\label{Xi}
\left\{ \begin{array}{ll}
u_1 =  \left(r + \sin{r} \cos{r - 2 \, \frac{\sin^2{r}}{r}}\right)
\sqrt{1 - \rho^2 \sin^2{r}} \, \rho^2 \\
u_2 =  \sin{r} \, \left(r + \sin{r} \cos{r} - 2 \frac{\sin^2{r}}{r}\right)  \rho^3
+ 2 \left(\frac{\sin{r}}{r} - \cos{r}\right)\rho \\
\end{array} \right..
\end{align}

A key observation is that $\Xi$ is also meaningful on $(\pi,\rho)$ for $\rho > 0$ and $\Xi(\pi,\rho) = \left(\pi \, \rho^2, 2 \, \rho\right)$ is a bijection from $ \{\pi\} \times (0, \ +\infty) $ to
$\left\{(u_1,u_2): \, u_2 = \frac{2}{\sqrt{\pi}} \sqrt{u_1} > 0 \right\}$. Set
\[
\widetilde{\Omega} := \widetilde{\Omega}_+ \cup \widetilde{\Omega}_- \cup (\{\pi\} \times (0, \ +\infty)).
\]
See the plot in Figure \ref{fig1}.

\begin{figure}
 \centering
\begin{overpic}[width = 15cm, height=10cm]{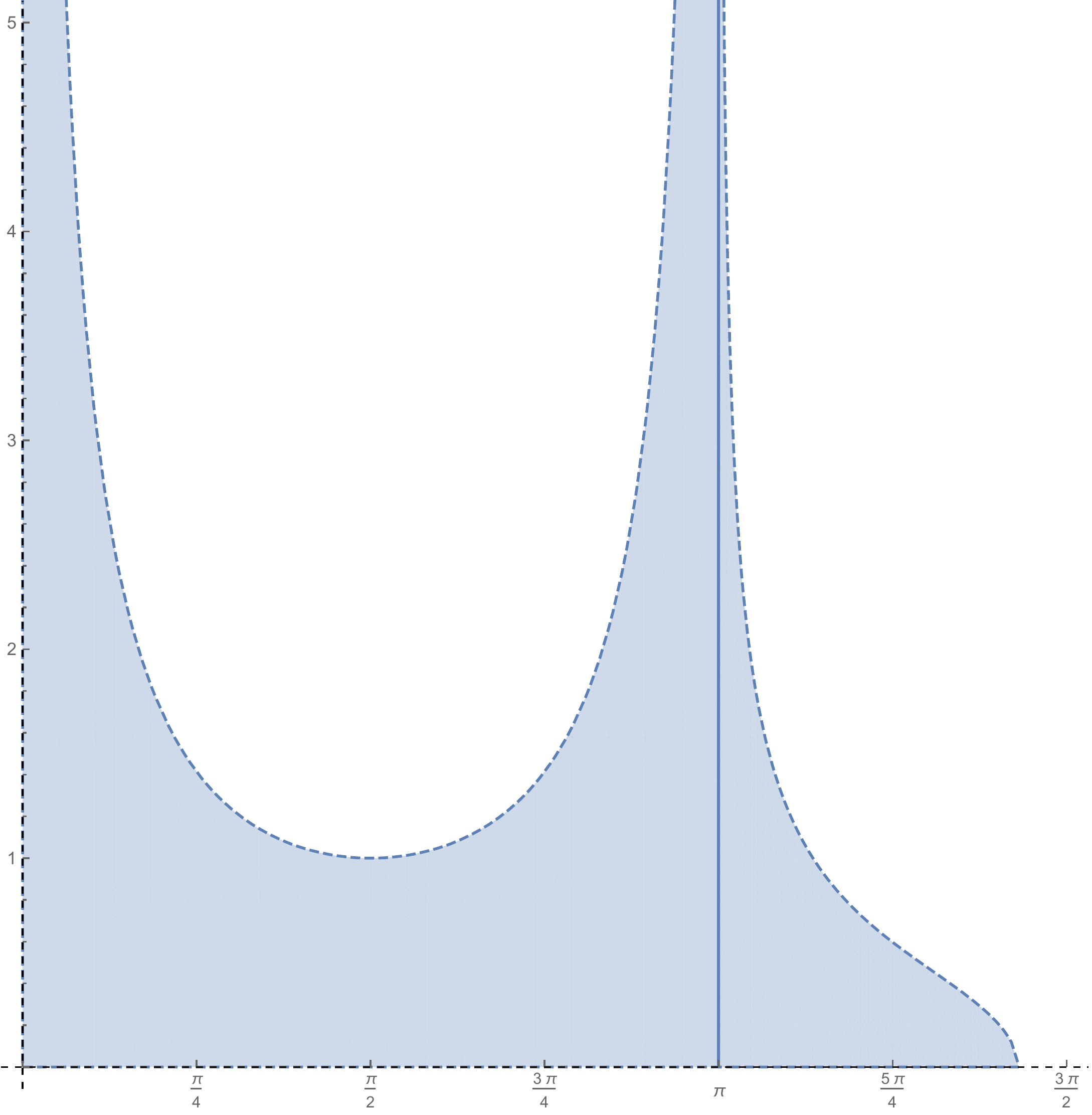}
\put(33,8){$\widetilde{\Omega}_+$}
\put(71,8){$\widetilde{\Omega}_-$}
\put(42.5,50){$\{\pi\} \times (0 , +\infty) \longrightarrow$}
\put(94,4){$\swarrow$}
\put(94,5.5){$\quad (\vartheta_1,0)$}
\put(42,30){$\rho = \frac{1}{\sin{r}} $}
\put(52,28.5){$\searrow$}
\put(69.5,28.5){$\swarrow$}
\put(73,30){$\rho = \sqrt{- \frac{2r^2\psi(r)}{h(r)}}$}
\put(100.5,2.4){$r$}
\put(1.5,68){$\rho$}
\end{overpic}
\caption[image]{Plot of $\widetilde{\Omega} = \widetilde{\Omega}_+ \cup \widetilde{\Omega}_- \cup (\{\pi\} \times (0, \ +\infty))$ }\label{fig1}
\end{figure}
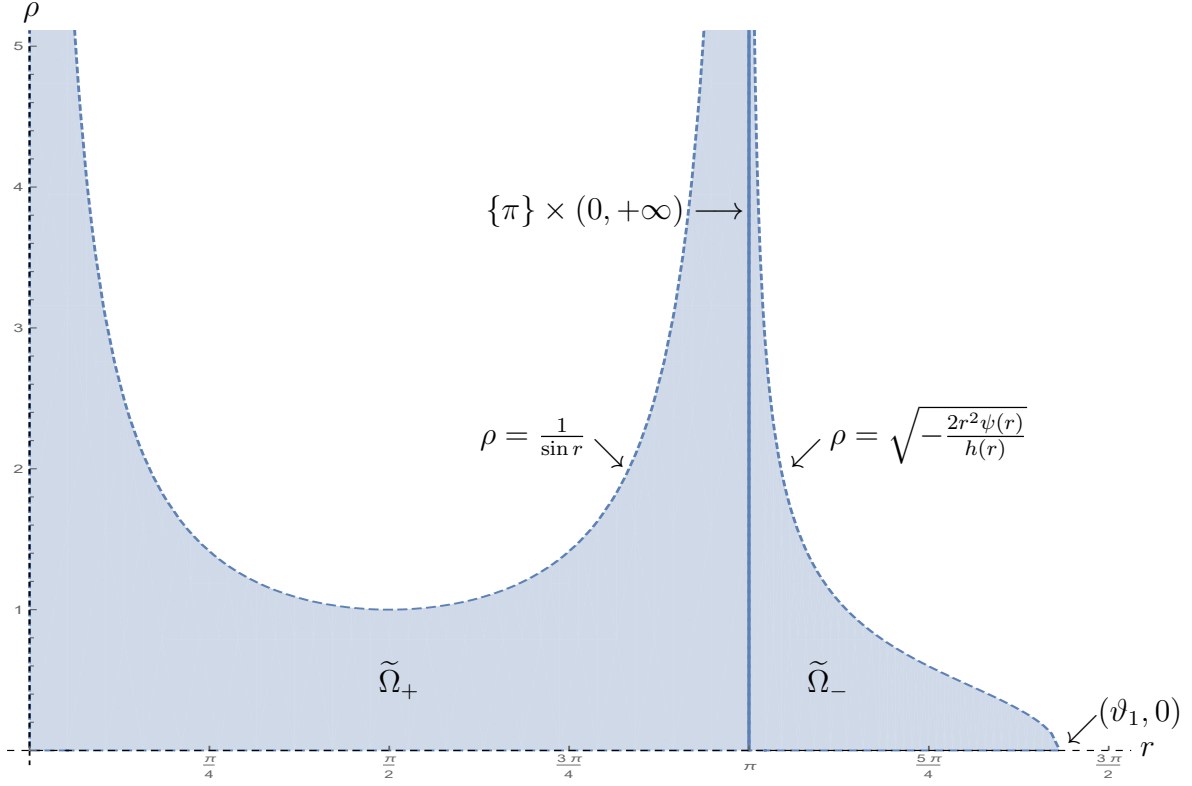

Moreover, notice that $\Xi$ is $C^{\infty}$ on $\widetilde{\Omega}$. A direct computation shows that the Jacobian determinant of $\Xi$ at $(\pi, \rho)$ ($\rho > 0$) equals
\begin{align*}
J(\Xi)(\pi, \rho)
= 2 \left( \pi^2 \rho^4 + 4 \rho^2 \right) > 0.
\end{align*}
As a consequence, $\Xi$ is a $C^\infty$-diffeomorphism
from $\widetilde{\Omega}$ onto $(0, \ +\infty) \times (0, \ +\infty)$.

In conclusion, by (iii), (iv) of Theorem \ref{RLT1} as well as Theorem \ref{t4}, we have
\begin{align}
d(e_1, \frac{1}{4} (u_1, u_2, 0))^2 = (r^2 - \sin^2{r}) \, \rho^2 + 1, \ \quad \forall \, u_1, u_2 > 0,
\end{align}
where $(r, \rho) = \Xi^{-1}(u_1,u_2)$. Finally, by using the smooth map
\begin{align*}
\widetilde{\Pi}: \mathrm{O}_3 \times (0, \ +\infty) \times \widetilde{\Omega} &\longrightarrow \R^3 \times \R^3 = N_{3,2} \\
(O, R, (r,\rho)) &\longmapsto  (R \, O \, e_1, \frac{R^2}{4} \, O \,(\Xi(r, \rho), 0)),
\end{align*}
we can get that $(e_1, \frac{1}{4} (u_1, \frac{2}{\sqrt{\pi}} \sqrt{u_1}, 0) \in \mathcal{S}$.

This completes the proof of this proposition.
\end{proof}

Indeed, the relation ``$\supseteq$'' in Proposition \ref{nP6n} can be improved to
``$=$''. In other words, we have the following:

\begin{corollary}\label{c4}
On $N_{3,2}$, we have $\mathrm{Cut}_o =  \{(x, t); \, x \  \mbox{and} \ t \  \mbox{are linearly dependent}\}$.
\end{corollary}

\begin{proof}
It follows from Proposition \ref{nP6n} that
\begin{align} \label{cutA}
\mathrm{Cut}_o \subseteq  \{(x, t); \, x \  \mbox{and} \ t \  \mbox{are linearly dependent}\}.
\end{align}
Moreover, we have that (see (v) of Theorem \ref{RLT1}):
\begin{align} \label{abnN}
\mathrm{Abn}^*_o = \{(x, 0); \ x \in \R^3\} = \widetilde{\M}_2.
\end{align}

By Theorem \ref{t2}, it remains to show that the classical cut locus of $o$ is
\begin{align} \label{n32CC}
\mathrm{Cut}_o^{\mathrm{CL}} = \{(x,t); \, t \neq 0 \mbox{ and } x = \lambda t \mbox{ for some } \lambda \in \R\} := \E,
\end{align}
which has been already proven in \cite{My02} and \cite{MM17} by completely different technique. Indeed, once we get \eqref{cutA} and \eqref{abnN}, via \cite[Lemma~9]{RS17},  \eqref{n32CC} is a direct consequence of the simple fact that there exist two distinct shortest geodesics from $o$ to any $g \in \E$. See Subsection \ref{SS76} below for more details.
\end{proof}

\medskip

\subsubsection{Description of shortest geodesic(s) from $o$ to any given $g \neq o$} \label{SS76}

\medskip

Let us begin by the following observation:

\begin{lemma} \label{n32nLn}
Let $\mathrm{SO}_3$ denote the $3 \times 3$ special orthogonal group, and $(x, t) = \exp(w, \tau)$. Then we have
\begin{gather}
 \exp(O \, w,O \, \tau) = (O \, x, O \,t), \quad \forall \, O \in \mathrm{SO}_3, \label{symN1} \\
\exp(O \, e^{\widetilde{U}(\tau)} \, w, O \, \tau) = (O \, x, O \,t), \quad \forall \, O \in \mathrm{O}_3 \setminus \mathrm{SO}_3. \label{symN4}
\end{gather}
\end{lemma}

\begin{proof}
Using the well-known basic property of the cross product:
\begin{align}\label{orthocro}
O \, (\tau \times \eta) = (O \, \tau) \times (O \, \eta), \quad \forall \, O \in \mathrm{SO}_3,
\end{align}
\eqref{symN1} can be checked directly by \eqref{expxN} and \eqref{exptpN}, or explained by \cite[\S~2.1]{MPAM06}.

To obtain \eqref{symN4},  we use the fact that  $\widetilde{U}(\tau) \, \eta = \tau \times \eta$ as well as \eqref{orthocro}, and get that
\begin{align}\label{orthoU}
\widetilde{U}(O \, \tau) \, \eta = (O \, \tau) \times \eta = O \, (\tau \times ( O^{\T} \, \eta)) = O \, \widetilde{U}(\tau) \, O^{\T} \, \eta, \quad \forall \, O \in \mathrm{SO}_3,
\end{align}
which implies that
\begin{align}\label{orthoU2}
\widetilde{U}(O \, \tau) =  O \, \widetilde{U}(\tau) \, O^{\T}, \qquad \forall \, O \in \mathrm{SO}_3.
\end{align}
Then for any $O \in \mathrm{O}_3 \setminus \mathrm{SO}_3$, we have $-O \in \mathrm{SO}_3$, and \eqref{symN1} implies that
\begin{align*}
\exp(-O \, w,-O \, \tau) = (-O \, x, -O \,t).
\end{align*}
Consequently, applying \eqref{symN3} to the last equation and using \eqref{orthoU2}, we obtain \eqref{symN4}.
\end{proof}

Combining  \eqref{symN1} and \eqref{symN4} with \eqref{symN2}, it suffices to determine all shortest geodesic(s) from $o$ to $g$, where: (1) $g = (e_1, \frac{1}{4}(u_1, u_2, 0))$ with $u_1 \ge 0$ and $u_2 > 0$; (2) $g = (e_1, 0)$; (3) $g = (0, e_1)$ or $g = (e_1, \frac{1}{4} \beta \, e_1)$ with $\beta > 0$.

{\em Case 1.} $g = (e_1, \frac{1}{4}(u_1, u_2, 0))$ with $u_1 \ge 0$ and $u_2 > 0$. In such case, we have $g \in \mathcal{S}$. So there exists a unique shortest geodesic from $o$ to $g$, which is strictly normal. If $(u_1, u_2) \in \R^2_> \cup \R_{<, +}^2$, it follows from \cite[Theorem~2.4 and Theorem~2.5]{Li19} that the shortest geodesic is given by $\gamma_{(w,2 \, \theta)}$ with $w = \frac{U(\theta)}{\sin{U(\theta)}} \, e^{-\widetilde{U}(\theta)} \, e_1$ and $\theta = (\Lambda^{-1}(u_1, u_2), 0)$. For $u_2 = \frac{2}{\sqrt{\pi}} \sqrt{u_1} > 0$, then it is given by $\gamma_{(w, 2 \, \theta)}$ with $w = (1, \sqrt{\pi \, u_1}, 0 )$ and $\theta = (\pi, 0, 0)$, see Remark \ref{r2} for more details.

{\em Case 2.} $g = (e_1, 0) \in \mathrm{Abn}_o^* \setminus \{ o \}$. The unique shortest geodesic joining $o$ to $g$ is a straight segment and it is abnormal.

{\em Case 3.} $g = (0, e_1)$ or $g = (e_1, \frac{1}{4} \beta \, e_1)$ with $\beta > 0$. In such case, $g \in \E$. By \eqref{symN1}, a trivial observation is that there exist at least two distinct shortest geodesics from $o$ to $g$. Indeed, a complete description can be found in \cite[\S~3]{MM17}. However, we will provide a completely different method to get it, which can be considered as a direct consequence of our main results, namely Theorems \ref{RLT1} and \ref{nT4n}. More precisely, we will use Remark \ref{RKn2} to determine the parameter $(w, \theta)$ of any shortest geodesic from $o$ to $g$, $\gamma_{(w, 2 \, \theta)}$.

{\em Case 3 (a).}
Let us begin with the case $g_{(\beta)} = (e_1, \frac{1}{4}(\beta, 0, 0))$ where $\beta > 0$. By Remark \ref{RKn2}, there exist $\{g_n : = (x^{(n)}, t^{(n)})\}_{n = 1}^{+ \infty} \subseteq \mathrm{Cut}_o^c$, with the corresponding unique shortest geodesic $\gamma_{(w^{(n)}, 2 \, \theta^{(n)})}$, such that $(g_n, w^{(n)}, \theta^{(n)}) \to (g_{(\beta)}, w, \theta)$ as $n \to +\infty$. Without loss of generality, we may assume that $|x^{(n)}| \ne 0$ and $x^{(n)} \cdot t^{(n)} > 0$ for all $n \ge 1$. For each $n \ge 1$, we pick an orthogonal matrix $O^{(n)} \in \mathrm{O}_3$ such that
\begin{align} \label{icNN}
O^{(n)} \, x^{(n)} = |x^{(n)}| \, e_1, \quad O^{(n)} \, t^{(n)} = \frac{|x^{(n)}|^2}{4} (u^{(n)}_1, u^{(n)}_2, 0) \ \mbox{with $u^{(n)}_1, u^{(n)}_2 > 0$.}
\end{align}
Combining this with the fact that $(x^{(n)}, t^{(n)}) \to (e_1, \frac{1}{4} \beta \, e_1)$ as $n \to +\infty$, we get that $(u^{(n)}_1, u^{(n)}_2) \to (\beta, 0)$ as $n \to +\infty$.
By arguing as in the proof of \eqref{NICN} and using \eqref{symN1}, \eqref{symN4} as well as \eqref{symN2}, we get that
\begin{align*}
O^{(n)} \, \theta^{(n)} \to (\Theta_1, \Theta_2, 0) \ \mbox{with} \  \Theta_1 = \sqrt{r^2  + 2 \, r \, \frac{ \psi(r)}{\psi'(r)}}, \  \Theta_2 = - \sqrt{-2 \, r \, \frac{\psi(r)}{\psi'(r)}},
\end{align*}
where $r$ is the unique solution of \eqref{DCUTP} in $(\pi, \ \vartheta_1)$. Since $\mathrm{O}_3$ is compact, up to subsequences, we may further assume that $O^{(n)} \to O'$ as $n \to +\infty$. Moreover, it follows from the first equation in \eqref{icNN} that the orthogonal matrix $O'$ satisfies $O' \, e_1 = e_1$. Since $O' \, \theta = (\Theta_1,\Theta_2,0)$, we yield that
\begin{align*}
\theta &= (\Theta_1, |\Theta_2| \cos{\sigma}, |\Theta_2| \sin{\sigma}) \  \mbox{for some $\sigma \in \R$}, \\
w = \lim_{n \to +\infty} w^{(n)}  &= \lim_{n \to +\infty} \frac{U(\theta^{(n)})}{\sin{U(\theta^{(n)})}} \, e^{-\widetilde{U}(\theta^{(n)})} \, x^{(n)} = \frac{U(\theta)}{\sin{U(\theta)}} \, e^{-\widetilde{U}(\theta)} \, e_1,
\end{align*}
where we have used \eqref{endpointx} in the second ``$=$'' of the last formula.

In other words, we have proven that every shortest geodesic from $o$ to $g_{(\beta)}$ can be expressed as $\gamma_{(w, 2 \, \theta)}$, where the parameter $(w, \theta)$ have the form
\begin{align}\label{thetaw}
\theta(\sigma) := (\Theta_1, |\Theta_2| \cos{\sigma}, |\Theta_2| \sin{\sigma}), \qquad
w(\sigma) := \frac{U(\theta(\sigma))}{\sin{U(\theta(\sigma))}} \, e^{-\widetilde{U}(\theta(\sigma))} \, e_1.
\end{align}
Furthermore, we will prove that the converse is also valid, namely every such parameter provides a shortest geodesic steering $o$ to $g_{(\beta)}$.

Let us fix a such shortest geodesic $\gamma_{(w(\sigma_0), 2 \, \theta(\sigma_0))}$, with $\sigma_0 \in \R$. By \eqref{symN1}, for any $\alpha \in \R$, $\gamma_{(O{(\alpha)} \, w(\sigma_0), 2 \, O{(\alpha)} \, \theta(\sigma_0))}$ is also
a shortest geodesic from $o$ to $g_{(\beta)}$, where
\[
O{(\alpha)} = \left(
                 \begin{array}{ccc}
                   1 & 0 & 0 \\
                   0 & \cos{\alpha} & - \sin{\alpha} \\
                   0 & \sin{\alpha} & \cos{\alpha} \\
                 \end{array}
               \right) \in \mathrm{SO}_3.
\]
It remains to show that
\begin{align} \label{nCvn}
(O{(\alpha)} \, w(\sigma_0), O{(\alpha)} \, \theta(\sigma_0)) = (w(\sigma_0 + \alpha),  \theta(\sigma_0 + \alpha)), \quad \forall \, \alpha \in \R,
\end{align}
since $(w(\sigma_0 + \alpha),  \theta(\sigma_0 + \alpha))$ runs over all possible $(w(\sigma),  \theta(\sigma))$ as the parameter $\alpha$ runs over $\R$.

In fact, it follows from \eqref{orthoU2} that
\begin{align}\label{invarw}
O(\alpha) \, w(\sigma_0) &= O(\alpha) \, \frac{U(\theta(\sigma_0))}{\sin{U(\theta(\sigma_0))}} \, e^{-\widetilde{U}(\theta(\sigma_0))} \, O(\alpha)^{\T}\, O(\alpha) \, e_1 \nonumber \\
&= \frac{U(O(\alpha) \, \theta(\sigma_0))}{\sin{U(O(\alpha) \, \theta(\sigma_0))}} \, e^{-\widetilde{U}(O(\alpha) \, \theta(\sigma_0))} \, e_1.
\end{align}
To finish the proof of \eqref{nCvn}, it suffices to notice that we have obviously
\[
O(\alpha) \, \theta(\sigma_0) = \theta(\sigma_0 + \alpha).
\]

Via some elementary but tedious calculations, our result can be identified with that of \cite[Theorem 3.2]{MM17} with
$(y, y^\perp, \theta)  = \left(\frac{\sqrt{\beta}}{2} \, e_2, \frac{\sqrt{\beta}}{2} \, e_3, r\right)$ therein, where we identify $\wedge^2 \R^3$ with $\R^3$ via the map $\T$ defined by
\begin{align*}
\T: \wedge^2 \R^3 &\longrightarrow \R^3 \\
a \wedge b & \longmapsto a \times b.
\end{align*}

{\em Case 3 (b).}
In the opposite case where $g = (0, e_1)$, similarly, it follows from Remark \ref{RKn2} and \eqref{expmap} that all shortest geodesics joining $o$ to $g$ are given by $\gamma_{(w, 2 \, \theta)}$, where
\begin{align*}
\theta = (\pi, 0, 0), \quad w = 2 \sqrt{\pi} \, (0,-\cos{\sigma},-\sin{\sigma}) \  \mbox{with} \ \sigma \in \R.
\end{align*}

Via some elementary but tedious calculations, our result can be identified with that of \cite[Theorem 3.2]{MM17} with
$(y, y^\perp, \theta)  = \left( e_2,  e_3, \pi \right)$ therein (via the map $\T$ as well).

\medskip

\renewcommand{\theequation}{\thesection.\arabic{equation}}
\section{Appendix A: Proof of Lemma \ref{LMng}}  \label{Axa}
\setcounter{equation}{0}

\medskip

The proof essentially follows that of \cite[Proposition~10.3]{Li19} and we include it for the sake of completeness.

Set $v := (v_1,v_2) \in \R^2$. Then from the definition of $\Upsilon((\yy,\xx,\xxa);\cdot)$,
for $|\yy|^2 + \xx^2 \neq 0$ and $\xxa \neq 0$, we have
\begin{align}\label{graexpUps}
(u_1, u_2) &:= \Upsilon((\yy,\xx,\xxa);(v_1, v_2)) \nonumber \\
&= \nabla_{v}\left[ \left( \frac{v_1^2}{4} \, \psi\left( \frac{|v|}{2} \right)
+ v_2^2 \, \psi(|v|) \right)|\xxa|^2 + f\left( \frac{|v|}{2} \right) (|\yy|^2 + \xx^2)\right].
\end{align}
More precisely,
\begin{align}\label{expUps}
\left\{ \begin{array}{ll}
 u_1 = v_1 \left[\mu\left(\frac{|v|}{2}\right) \frac{|\yy|^2 + \xx^2 }{2|v|}
+ \psi^{\prime}\left(\frac{|v|}{2}\right) \, \frac{v_1^2}{4} \frac{|\xxa|^2}{2|v|}
+ \psi\left(\frac{|v|}{2}\right) \frac{|\xxa|^2}{2}
 + \psi^{\prime}(|v|) \, v_2^2  \, \frac{|\xxa|^2}{|v|}  \right]\\
 \mbox{} \\
 u_2 = v_2 \left[ \mu\left(\frac{|v|}{2}\right) \frac{|\yy|^2 + \xx^2 }{2|v|}
+ \psi^{\prime}\left(\frac{|v|}{2}\right) \, \frac{v_1^2}{4} \frac{|\xxa|^2}{2|v|}
 + \psi^{\prime}(|v|) \, v_2^2  \, \frac{|\xxa|^2}{|v|}
+ 2 \, \psi(|v|) \, |\xxa|^2 \right]\\
\end{array} \right..
\end{align}

It follows from \cite[Lemma 3.3]{Li19} that
\begin{align*}
0 \le \frac{\psi^{\prime}(|v|)}{|v|} \, |v_1| \, v_2^2 \le \frac{\pi}{4} \, v_2^2 \left( 2 \, \psi(|v|) + \frac{\psi^{\prime}(|v|)}{|v|} \, v_2^2 \right)^2, \qquad \forall \, |v|< \pi,
\end{align*}
which implies that
\begin{align*}
|u_1| &\le \mu\left(\frac{|v|}{2}\right) \frac{|\yy|^2 + \xx^2 }{2} + \left( \psi^{\prime}\left(\frac{|v|}{2}\right) \, \frac{|v|^2}{8} + \psi\left(\frac{|v|}{2}\right) \frac{|v|}{2} \right) |\xxa|^2 \\
&+ \frac{\pi}{4} \, v_2^2 \left( 2 \, \psi(|v|)
  + \frac{\psi^{\prime}(|v|)}{|v|} \, v_2^2 \right)^2 |\xxa|^2.
\end{align*}
Using the identity $s^2 \, \psi^{\prime}(s) + 2 s \, \psi(s) = \mu(s)$, the second equality in \eqref{expUps} implies that
\begin{align*}
|u_1| \le \, \mu\left(\frac{|v|}{2}\right)\frac{|\yy|^2 + \xx^2 + |\xxa|^2 }{2} + \frac{\pi}{4} \, \frac{ u_2^2}{|\xxa|^2}
<  \, \frac{\pi}{4} \left(\frac{u_2^2}{|\xxa|^2} + |\yy|^2 + \xx^2 + |\xxa|^2 \right),
\end{align*}
where we have used in ``$<$'' the fact that $\mu(\pi/2) = \pi/2$ and $\mu$ is strictly increasing on $(-\pi,\pi)$ (see Lemma \ref{NL31}), so $\Upsilon((\yy,\xx,\xxa);\cdot)$ is from $B_{\R^2}(0, \pi)$ to $\R^2_{r}(\yy,\xx,\xxa)$.

Next, it follows from \eqref{graexpUps} that the Jacobian of $\Upsilon((\yy,\xx,\xxa);\cdot)$ is given by
\begin{align*}
\mathrm{Hess}_{v}\left[ \left( \frac{v_1^2}{4} \, \psi\left( \frac{|v|}{2} \right)
+ v_2^2 \, \psi(|v|) \right)|\xxa|^2 + f\left( \frac{|v|}{2} \right) (|\yy|^2 + \xx^2)\right] > 0,
\end{align*}
since $|\yy|^2 + \xx^2 \neq 0$, $\xxa \neq 0$, $\mathrm{Hess}_{v}(f(|v|)) > 0$ and $\mathrm{Hess}_{v}(v_1^2 \, \psi(|v|) ) \ge 0$ (so symmetrically $\mathrm{Hess}_{v}(v_2^2 \, \psi(|v|) ) \ge 0$) for $|v| < \pi$, which can be found in the proof of \cite[Proposition~10.1]{Li19}.

Finally, using Hadamard's theorem, it remains to show that $\Upsilon((\yy,\xx,\xxa);\cdot)$ is proper. To this end, we consider the behaviour of  $\{v^{(j)}\}_{j = 1}^{+\infty}
\subseteq B_{\R^2}(0, \pi)$
satisfying $v^{(j)} = (v^{(j)}_1,v^{(j)}_2) \longrightarrow \partial B_{\R^2}(0, \pi)$ and we split it into cases. \\
(1) If $|v^{(j)}| \longrightarrow \pi^-$ and $|v^{(j)}_2| \ge \varepsilon > 0$, then $\Upsilon((\yy,\xx,\xxa);v^{(j)}) \longrightarrow \infty$. \\
(2) If $|v^{(j)}| \longrightarrow \pi^-$ and $|v^{(j)}_2| \longrightarrow 0^+$, set $\Upsilon((\yy,\xx,\xxa);v^{(j)}) = (u^{(j)}_1, u^{(j)}_2)$. \\
In the case $|u^{(j)}_1| \longrightarrow +\infty$, it is easy to see $\Upsilon((\yy,\xx,\xxa);v^{(j)}) \longrightarrow \infty$. \\
In the opposite case, assume $|u^{(j)}_1| \longrightarrow a \, (\ge 0)$. Then from the first equation of \eqref{expUps}, using the fact that $s^2 \, \psi^{\prime}(s) + 2 s \, \psi(s) = \mu(s)$ as well as $\mu(\pi/2) = \pi/2$, we have that
\begin{align*}
a = \frac{\pi}{4} \left(  |\yy|^2 + \xx^2 + |\xxa|^2 \right) + \lim_{j \longrightarrow +\infty} \frac{1}{\pi} \left( \frac{v^{(j)}_2}{\pi - |v^{(j)}|} \right)^2 |\xxa|^2,
\end{align*}
where we have used the fact that (see also \cite[Lemma~3.4]{Li19})
\begin{align*}
\lim_{s \longrightarrow \pi^-} (\pi - s)^2 \, \psi'(s) = \frac{1}{\pi}.
\end{align*}
Since
\[
\lim_{s \longrightarrow \pi^-} (\pi - s) \, \psi(s) = \frac{1}{\pi},
\]
the second equation of \eqref{expUps} gives that
\begin{align*}
\lim_{j \longrightarrow +\infty} |u^{(j)}_2| = \lim_{j \longrightarrow +\infty} \frac{2}{\pi} \frac{|v^{(j)}_2|}{\pi - |v^{(j)}|} \, |\xxa|^2
= \frac{2 \, |\xxa|}{\sqrt{\pi}} \sqrt{a - \frac{\pi}{4}
\left(  |\yy|^2 + \xx^2 + |\xxa|^2 \right) }.
\end{align*}

In conclusion, we have
$\Upsilon((\yy,\xx,\xxa);v^{(j)}) \longrightarrow \partial \R^2_{r}(\yy,\xx,\xxa)$ when $v^{(j)}  \longrightarrow \partial B_{\R^2}(0, \pi)$
in $B_{\R^2}(0, \pi)$, which means it is proper.

This finishes the proof of this lemma.

\medskip

\renewcommand{\theequation}{\thesection.\arabic{equation}}
\section{Appendix B: Properties of the direct product of two 2-step groups}  \label{Axb}
\setcounter{equation}{0}

\medskip

Assume that $\G_j = \G(q_j, m_j, \U_j)$ ($j = 1, 2$), where
\[
\U_j = \{ \UU_j^{(1)}, \ldots, \UU_j^{(m_j)} \}.
\]

Consider the direct product $\G_1 \times \G_2 := \G(q_1 + q_2, m_1 + m_2, \U)$ with $\U$ defined by
\[
\left\{ \left(
              \begin{array}{cc}
                \UU_1^{(1)} & \mbox{} \\
                \mbox{} & \O_{q_2 \times q_2} \\
              \end{array}
            \right),
\ldots,
\left(
              \begin{array}{cc}
                \UU_1^{(m_1)} & \mbox{} \\
                \mbox{} & \O_{q_2 \times q_2} \\
              \end{array}
            \right),
\left(
              \begin{array}{cc}
                \O_{q_1 \times q_1} & \mbox{} \\
                \mbox{} & \UU_2^{(1)} \\
              \end{array}
            \right),
\ldots,
\left(
              \begin{array}{cc}
                \O_{q_1 \times q_1} & \mbox{} \\
                \mbox{} & \UU_2^{(m_2)} \\
              \end{array}
            \right)
   \right\},
\]
where $\O_{k_1 \times k_2}$ denotes the $k_1 \times k_2$ null matrix. Let $\g_j := (\x_j, \t_j) \in \G_j, j = 1, 2$.
We identify $(\g_1, \g_2)$ with $((\x_1, \x_2), (\t_1, \t_2))$.
Then, it is clear that
\begin{gather*}
\OA^{(\G_1 \times \G_2)} = \OA^{(\G_1)} \times \OA^{(\G_2)}, \\
\phi^{(\G_1 \times \G_2)}((\g_1, \g_2); (\theta_{(1)}, \theta_{(2)})) = \phi^{(\G_1)}(\g_1; \theta_{(1)}) + \phi^{(\G_2)}(\g_2;  \theta_{(2)}), \\
\M^{(\G_1 \times \G_2)} = \M^{(\G_1)} \times \M^{(\G_2)}, \quad  d_{\G_1 \times \G_2}((\g_1, \g_2))^2 = d_{\G_1}(\g_1)^2 + d_{\G_2}(\g_2)^2,
\end{gather*}
and the meaning of the notations herein is obvious. It is clear that $(\gamma_{\G_1}(s), \gamma_{\G_2}(s))$ is a normal geodesic if and only if both $\gamma_{\G_1}$ and $\gamma_{\G_2}$ are normal geodesics. Moreover, the normal geodesic $(\gamma_{\G_1}(s), \gamma_{\G_2}(s))$ is abnormal if and only if $\gamma_{\G_1}$ or $\gamma_{\G_2}$ is abnormal. Also, it is shortest if and only if both $\gamma_{\G_1}$ and $\gamma_{\G_2}$ are shortest. Hence, we have the following well-known properties: on $\G_1 \times \G_2$, it holds that:
\begin{gather}
\mathrm{Cut}_o^{(\G_1 \times \G_2)} = \left( \mathrm{Cut}_o^{(\G_1)} \times \G_2 \right) \cup \left( \G_1 \times \mathrm{Cut}_o^{(\G_2)} \right), \\
\mathrm{Cut}_o^{\mathrm{CL} \  (\G_1 \times \G_2)} = \left( \mathrm{Cut}_o^{\mathrm{CL} \ (\G_1)} \times \G_2 \right) \cup \left( \G_1 \times \mathrm{Cut}_o^{\mathrm{CL} \ (\G_2)} \right), \\
\mathrm{Abn}_o^{* \ (\G_1 \times \G_2)} = \left( \mathrm{Abn}_o^{* \ (\G_1)} \times \G_2 \right) \cup \left( \G_1 \times \mathrm{Abn}_o^{* \ (\G_2)} \right),
\end{gather}
and the meaning of the notations herein is obvious.

By the fact that $\overline{\M^{(\G_1 \times \G_2)}} = \overline{\M^{(\G_1)}} \times \overline{\M^{(\G_2)}}$, we have the following:

\begin{proposition}
$\G_1 \times \G_2$ is of type GM if and only if both $\G_1$ and $\G_2$ are GM-groups.
\end{proposition}

Similarly, on the direct product of the Euclidean space $\R^k$ with a step-two group $\G$,
we have
\begin{gather}
\mathrm{Cut}_o^{(\R^k \times \G)} = \R^k \times \mathrm{Cut}_o^{(\G)},  \quad \mathrm{Cut}_o^{\mathrm{CL} \ (\R^k \times \G)} = \R^k \times \mathrm{Cut}_o^{\mathrm{CL} \ (\G)},  \\
\mathrm{Abn}_o^{* \ (\R^k \times \G)} = \R^k \times \mathrm{Abn}_o^{* \ (\G)},
\end{gather}
and the meaning of the notations herein is obvious. In particular,

\begin{proposition}  \label{NSAa}
Let $\G$ be a step-two group. Consider the direct product of the Euclidean space
$\R^k$ with $\G$. Then:
\begin{description}
  \item[(1)] $\R^k \times \G$ is of type GM if and only if $\G$ is a GM-group;
  \item[(2)] $\R^k \times \G$ is of type SA if and only if $\G$ is a M\'etivier or SA-group.
\end{description}
\end{proposition}

\medskip

\renewcommand{\theequation}{\thesection.\arabic{equation}}
\section{Appendix C: Construction of SA-groups}  \label{Axc}
\setcounter{equation}{0}

\medskip

In \cite[\S~8.1]{Li19}, there is a simple method to construct an uncountable number of GM-groups (resp. GM-groups of M\'etivier type) from any given step-two group (resp. M\'etivier group). We will use the same method to produce SA-groups. More precisely, assume that $m \ge 2$ and $\G_j = \G(q_j, m, \U_j)$ ($j = 1, 2$), where
\[
\U_j = \{ \UU_j^{(1)}, \ldots, \UU_j^{(m)} \}.
\]

We consider $\G := \G(q_1 + q_2, m, \U)$ with $\U$ defined by
\[
\left\{ \left(
              \begin{array}{cc}
                \UU_1^{(1)} & \mbox{} \\
                \mbox{} & \UU_2^{(1)} \\
              \end{array}
            \right),
\ldots,
\left(
              \begin{array}{cc}
                \UU_1^{(m)} & \mbox{} \\
                \mbox{} & \UU_2^{(m)} \\
              \end{array}
            \right)
   \right\}.
\]

Let $g = (\xnx_1, \xnx_2, t) \in \R^{q_1} \times \R^{q_2} \times \R^m$. Obviously, we have:
\begin{gather*}
\OA^{(\G)} = \OA^{(\G_1)} \cap \OA^{(\G_2)}, \\
\phi^{(\G)}(g; \tau) = \langle U^{(\G_1)}(\tau) \cot{U^{(\G_1)}(\tau)} \, \xnx_1, \ \xnx_1 \rangle + \langle U^{(\G_2)}(\tau) \cot{U^{(\G_2)}(\tau)} \, \xnx_2, \ \xnx_2 \rangle + 4 \, t \cdot \tau,
\end{gather*}
and the meaning of the notations herein is clear. Moreover, $g \in \widetilde{\M}_2^{(\G)}$ if and only if there are $\theta \in \OA^{(\G_1)} \cap \OA^{(\G_2)}$ and $t', t''
\in \R^m$ such that:
\begin{gather*}
t = t' + t'',
 \quad t' = - \frac{1}{4} \nabla_{\theta} \langle U^{(\G_1)}(\theta) \cot{U^{(\G_1)}(\theta)} \, \xnx_1, \ \xnx_1 \rangle, \\
t''
= - \frac{1}{4} \nabla_{\theta} \langle U^{(\G_2)}(\theta) \cot{U^{(\G_2)}(\theta)} \, \xnx_2, \ \xnx_2 \rangle,
\end{gather*}
and the sum of two positive semidefinite matrices
\[
\left( - \mathrm{Hess}_{\theta} \langle U^{(\G_1)}(\theta) \cot{U^{(\G_1)}(\theta)} \, \xnx_1, \ \xnx_1 \rangle \right) + \left( - \mathrm{Hess}_{\theta} \langle U^{(\G_2)}(\theta) \cot{U^{(\G_2)}(\theta)} \, \xnx_2, \ \xnx_2 \rangle \right)
\]
is singular. In such case, we have $(\xnx_1, t') \in \widetilde{\M}_2^{(\G_1)}$ and $(\xnx_2, t'') \in \widetilde{\M}_2^{(\G_2)}$. Thus, we get immediately the following:

\begin{proposition}
With the above notations, $\G$ is a SA-group when $\G_1$ is of type SA and $\G_2$ is of type M\'etivier or SA.
\end{proposition}

\medskip

\section*{Acknowledgement}
\setcounter{equation}{0}
This work is partially supported by NSF of China (Grants  No. 11625102 and No. 11831004) and ``The Program of Shanghai Academic Research Leader'' (18XD1400700).  The authors would like to thank L. Rizzi for many useful suggestions.

\nocite{*}
\bibliographystyle{abbrv}
\bibliography{LZcc}

\mbox{}\\
Hong-Quan Li, Ye Zhang\\
School of Mathematical Sciences/Shanghai Center for Mathematical Sciences  \\
Fudan University \\
220 Handan Road  \\
Shanghai 200433  \\
People's Republic of China \\
E-Mail: hongquan\_li@fudan.edu.cn \\
17110180012@fudan.edu.cn \quad or \quad zhangye0217@126.com \mbox{}\\

\end{document}